\documentclass[a4paper,reqno]{amsart}

\usepackage{amsmath,amssymb,amsthm}

\usepackage[abbrev]{amsrefs}

\usepackage{comment,braket,stmaryrd}
\usepackage{color}
\usepackage{mathtools}
\mathtoolsset{showonlyrefs=true}

\usepackage[top=25mm, bottom=25mm, left=25mm, right=25mm]{geometry}

\usepackage{cite,mathrsfs}

\newtheorem{thm}{Theorem}[section]
\newtheorem{lem}[thm]{Lemma}

\newtheorem{prop}[thm]{Proposition}
\numberwithin{equation}{section}

\theoremstyle{definition}

\newtheorem{rmk}[thm]{Remark}
\newtheorem{conv}[thm]{Convention}

\newcommand{\relmiddle}[1]{\mathrel{}\middle#1\mathrel{}}

\newcommand{\BB}{\mathbb{B}}
\newcommand{\BC}{\mathbb{C}}

\newcommand{\BN}{\mathbb{N}}

\newcommand{\BR}{\mathbb{R}}
\newcommand{\BZ}{\mathbb{Z}}

\newcommand{\SA}{\mathscr{A}}
\newcommand{\SD}{\mathscr{D}}
\newcommand{\SF}{\mathscr{F}}

\newcommand{\SL}{\mathscr{L}}

\renewcommand{\SS}{\mathscr{S}}

\newcommand{\CF}{\mathcal{F}}
\newcommand{\CI}{\mathcal{I}}
\newcommand{\CJ}{\mathcal{J}}
\newcommand{\CK}{\mathcal{K}}
\newcommand{\CL}{\mathcal{L}}
\newcommand{\CN}{\mathcal{N}}
\newcommand{\CT}{\mathcal{T}}
\newcommand{\pd}{\partial}

\renewcommand{\Re}{\mathrm{Re}\,}
\renewcommand{\d}{\,\mathrm{d}}

\renewcommand{\d}{\,\mathrm{d}}
\renewcommand{\div}{\mathrm{div}\,}
\DeclareMathOperator{\curl}{curl}

\begin{document}

%

\title[]{
\uppercase{Navier--Stokes flow in the exterior of a moving obstacle \\ with a Lipschitz boundary}
}

\author[]{Tomoki Takahashi}
\address{(T. Takahashi) 
Tokyo Institute of Technology,
Meguro, Tokyo, 152-8511 Japan}
\email{takahashi.t.cx@m.titech.ac.jp}

\author[]{Keiichi Watanabe}
\address{(K. Watanabe) School of General and Management Studies,
Suwa University of Science, 5000-1, Toyohira, Chino, Nagano 391-0292, Japan}		
\email{\texttt{watanabe\_keiichi@rs.sus.ac.jp}}

\dedicatory{Dedicated to Professor Toshiaki Hishida on the Occasion of his 60th Birthday}


\thanks{The first author was partially supported by 
JSPS KAKENHI Grant Number 20H00118 and 23K12999, and 
the second author was partially supported by JSPS KAKENHI Grant Number 21K13826.}


\date{}   

\begin{abstract}
Consider the three-dimensional Navier--Stokes flow 
past a moving rigid body $\mathscr{O} \subset \mathbb{R}^3$
with prescribed translational and angular velocities, 
where $\mathscr{O}$ stands for a bounded Lipschitz domain. 
We prove that the solution to the linearized problem is 
governed by a $C_0$-semigroup on 
solenoidal $L^q$-vector spaces with the $L^q$-$L^r$ 
estimates provided that $|1/q-1/2|<1/6+\varepsilon$ 
with some $\varepsilon>0$, where $r \ge q$ may be taken 
arbitrary large. As an application, 
we prove the existence and 
uniqueness of global mild solutions to the Navier--Stokes problem 
if the translational and angular velocities as well as 
the initial are sufficiently small. 
\end{abstract}

\maketitle

\section{Introduction}
Let us consider the three-dimensional Navier--Stokes flow past 
an obstacle with the Lipschitz boundary, that 
moves with prescribed translational and angular velocities. 
In the reference frame attached to the obstacle, 
the motions of the incompressible Navier--Stokes liquid is 
described by the following system in a fixed exterior domain 
$\Omega\subset \BR^3$ with the Lipschitz boundary:   
\begin{equation}
\label{eq-main}
\left\{\begin{aligned}
\pd_t v + v \cdot \nabla v - \Delta v - (\eta + \omega \times x) \cdot \nabla v + \omega \times v  + \nabla \pi & = 0
& \quad & \text{in $\Omega \times \BR_+$}, \\
\div v & = 0 & \quad & \text{in $\Omega \times \BR_+$}, \\
v & = \eta + \omega \times x & \quad & \text{on $\pd \Omega \times \BR_+$}, \\
v & \to 0 & \quad & \text{as $\lvert x \rvert \to \infty$}, \\
v \vert_{t = 0} & = v_0 & \quad & \text{in $\Omega$},
\end{aligned}\right.
\end{equation}
see Galdi \cite[Sec.~1]{G02} for details. Here, 
$v = (v_1 (x,t), v_2 (x,t), v_3 (x,t))^\top$ and 
$\pi = \pi (x,t)$ are 
unknown velocity field and pressure of the fluid, respectively, whereas 
$\eta = (\eta_1,\eta_2,\eta_3)^\top$ and $\omega = (\omega_1, \omega_2, \omega_3)^\top$ are 
the translational and angular velocities of the rigid body, respectively, where
$\eta_i, \omega_i \in \BR$, $i = 1,2,3$, are assumed to be constants. Here, we 
have used the notation $(\,\cdot\,)^\top$ to express the transpose of vectors. 
Without loss of generality, we suppose that the origin of the coordinates is located interior to the complement of the domain $\Omega$. 
\par Our results include 
the case $\omega=0$, and 
the linearized system 
\begin{equation}
\label{eq-main2}
\left\{\begin{aligned}
\pd_t u - \Delta u - (\eta + \omega \times x) \cdot \nabla u + \omega \times u  + \nabla p & = 0
& \quad & \text{in $\Omega \times \BR_+$}, \\
\div u & = 0 & \quad & \text{in $\Omega \times \BR_+$}, \\
u & =0& \quad & \text{on $\pd \Omega \times \BR_+$}, \\
u & \to 0 & \quad & \text{as $\lvert x \rvert \to \infty$}, \\
u \vert_{t = 0} & = f & \quad & \text{in $\Omega$}
\end{aligned}\right. 
\end{equation}
with $\eta=\omega=0$ (resp. $\eta\ne 0,\omega=0$) 
is said to be the Stokes system (resp. the Oseen system). 
If the boundary is smooth (at least of class $C^{1,1}$), 
the well-posedness of the Stokes system (resp. the Oseen system)  
in the sense that generation of an analytic $C_0$-semigroup, 
which is called the Stokes semigroup  
(resp. the Oseen semigroup),
on $L^q_\sigma(\Omega)~(1<q<\infty)$ was established 
by \cite{G81,S77,FS94} (resp. by \cite{M82}). 
Here, $L^q_\sigma(\Omega)$ is the space of solenoidal $L^q$-vector 
fields with vanishing normal trace at $\partial \Omega$. 
Concerning studies of the attainability 
(raised by Finn \cite{F65}) 
as well as the stability of the steady Navier--Stokes flow 
within the $L^q$-framework, 
see \cite{BM95,ES05,GHS97,HM18,K17,KY98,S99,T22,Y00}. 
In these studies, 
the $L^q$-$L^r$ estimates of the semigroup play a crucial role. 
By \cite{I89,MS97} for $\eta=\omega=0$ and by  
\cite{KS98,ES04,ES05,H21} for $\eta\ne 0,\omega=0$, 
we know the following 
$L^q$-$L^r$ estimates of the semigroup
\begin{equation}
\label{lqlr1}
\lVert \nabla^jS_{\eta,\Omega}(t) f \rVert_{L^r(\Omega)}
\le C t^{-\frac32 (\frac1q - \frac1r) -\frac j2} 
\lVert f \rVert_{L^q(\Omega)},\qquad t>0,
\, f\in L^q_\sigma(\Omega),\, j=0,1
\end{equation}
provided that 
\begin{align}
\label{condi-qr}
1<q\leq r\leq \infty~(q\ne\infty)\quad \text{if}~j=0,
\qquad 1<q\leq r\leq 3 \quad \text{if}~j=1, 
\end{align}
where 
$(S_{\eta,\Omega}(t))_{t\ge 0}$ stands for the Oseen semigroup 
for the case $\eta \ne 0$ and 
$(S_{0,\Omega}(t))_{t\ge 0}$ stands for the Stokes semigroup 
for the case $\eta = 0$. 
If $\omega\ne 0, \eta=0$, 
then Hishida \cite{H99a,H99b} (in $L^2$) and 
Geissert, Heck and Hieber \cite{GHH06} (in $L^q$) 
proved that the solution to Problem \eqref{eq-main2} is governed 
by a $C_0$-semigroup, which is called  
the Stokes semigroup with the rotation effect. 
Indeed, the coefficient in the drift term 
$(\omega \times x) \cdot \nabla u$ is unbounded and 
the essential spectrum of 
$\Delta u+(\eta+\omega \times x)\cdot\nabla u-\omega \times u$,
provided that $\omega$ is parallel to the $x_1$-axis, 
is given by 
$\{a + i k \lvert \omega \rvert 
\mid a \le 0, \enskip k \in \BZ \}$ due to \cite{FNN07,FN07}, 
which means that the drift term 
induces a hyperbolic aspect and that the drift term is never 
subordinate to the viscous term. Nevertheless, some  
smoothing properties of the Stokes semigroup 
with the rotation effect were established by 
\cite{H99a,H99b,GHH06}. Moreover, 
the $L^q$-$L^r$ estimates of 
the semigroup deduced in \cite{GHH06} were developed by 
Hishida and Shibata \cite{HS09}, 
and they constructed a unique global solution 
to the nonlinear problem, which tends to a 
stationary solution as $t\rightarrow \infty$. The results in 
\cite{H99a,H99b,GHH06,HS09} were also extended to 
the case $\omega = (0,0,k)^\top\ne 0$, 
$\eta = (0,0,a)^\top \ne 0$, 
$(k,a\in\BR)$ by Shibata \cite{S08,S10}. 
Here, notice that 
the Mozzi--Chasles transform \cite[VIII.0.5]{G11} makes it 
possible to reduce 
the problem to the particular case 
that $\omega$ and $\eta$ are parallel to the $x_3$-axis. 
\par The aim of this paper is to extend the above results 
on the linearized problem 
to the case of the Lipschitz boundary 
and, as an application, 
we construct a unique global solution to 
Problem \eqref{eq-main} provided that 
$\eta$, $\omega$, and $v_0$ are suitably small. 
Indeed, the Lipschitz boundary case was 
less studied in contrast to the smooth
boundary case, even for the case $\eta=\omega=0$ and 
to the best of the authors' knowledge, 
we only know the results 
due to the second author \cites{TW20,W23} toward this topic. 
It was proved in \cite{TW20} that 
the Stokes operator 
generates an analytic semigroup on $L^q_\sigma(\Omega)$ 
provided 
\begin{equation}
\label{cond-p}
\bigg\lvert \frac1q - \frac12 \bigg\rvert < \frac16 + \varepsilon,
\end{equation}
where $\varepsilon > 0$ is a number depending on $\Omega$. 
This result gave an affirmative answer to 
Taylor's conjecture \cite{T00} 
in the case of exterior Lipschitz domains. 
Notice that Taylor's conjecture in the case of a bounded Lipschitz domain, say, $D$, was affirmatively solved by Shen \cite{S12}, who
proved that the Stokes operator generates an analytic semigroup on 
$L^q_\sigma(D)$ provided $q$ satisfies \eqref{cond-p}, see also \cite{KW17,T18,Tthesis} 
for the extention of Shen's paper. 
Some $L^q$-$L^r$ estimates of the Stokes semigroup 
and its gradient in exterior Lipschitz domains 
were also deduced from \cite{TW20}, where 
the condition 
\begin{equation}
\label{cond-r2}
\bigg\lvert \frac1r - \frac12 \bigg\rvert < \frac16 + \varepsilon
\end{equation}
was imposed. However, the condition \eqref{cond-r2} was
redundant and removed 
by the second author \cite{W23}. In particular, he established 
the desired $L^q$-$L^r$ estimates of the Stokes semigroup 
and its gradient, i.e., 
\eqref{lqlr1} with $\eta = 0$ for $q,r$ fulfilling
\eqref{condi-qr} as well as \eqref{cond-p}. 
These estimates 
enable us to apply the classical approach due to Kato \cite{K84} 
(see also Fujita and Kato \cite{FK64}) and Weigner's method 
\cite{W00} to the nonlinear problem, see \cite[Thm. 1.10]{W23}. 
\par In order to describe our main results, 
following \cite{TW20}, we first define the Stokes operator
on $L^q_\sigma (\Omega)$, $1 < q < \infty$. 
First, we define the Stokes
operator $A_{2,\Omega}$ on $L^2_\sigma(\Omega)$ 
by using a sesquilinear form, see, e.g.,
\cite[Ch. 4]{MM08} or \cite[Sec. III.2.1]{Sbook}. 
Then the Stokes operator $A_{q,\Omega}$
on $L^q_\sigma(\Omega)$, $1<q<\infty$, is defined in two steps. First, we take the part of
$A_{2, \Omega}$ in $L^q_\sigma (\Omega)$, i.e.,
\begin{equation}
\SD(A_{2,\Omega}|_{L^q_\sigma(\Omega)}):=
\{u\in \SD (A_{2,\Omega})\cap L^q_\sigma(\Omega)
\mid A_{2,\Omega} u\in L^q_\sigma(\Omega)\}
\end{equation}
and $A_{2,\Omega} \vert_{L^q_\sigma (\Omega)}$ is given by $A_{2,\Omega} u$ for $u \in \SD (A_{2,\Omega})$.
Notice that $A_{2,\Omega}|_{L^q_\sigma(\Omega)}$ is densely defined and is
closable. Second, we define $A_{q,\Omega}$ as the closure of $A_{2,\Omega}|_{L^q_\sigma(\Omega)}$
in $L^q_\sigma(\Omega)$. If there is no confusion, we omit the subscription $q$ in the notation of the Stokes operator. 
\par For $q$ fulfilling \eqref{cond-p}, 
let $P_\Omega$ denote the Helmholtz projection 
from $L^q(\Omega)$ onto $L^q_\sigma(\Omega)$ associated with 
the decomposition
\begin{align*}
L^q(\Omega)=L^q_{\sigma}(\Omega)\oplus\{\nabla p\in L^q(\Omega)\mid 
p\in L^q_{\rm{loc}}(\Omega)\},
\end{align*}
see Lang and Mendez \cite[Thm. 6.1]{LM06} and Tolksdorf and Watanabe \cite[Prop. 2.3]{TW20}. 
Given $\eta = (\eta_1, \eta_2, \eta_3)^\top \in \BR^3$ and 
$\omega=(\omega_1,\omega_2,\omega_3)^\top \in \BR^3$,
we define $B_{\Omega,\eta,\omega} \colon
L^q_{\sigma}(\Omega) \to L^q_{\sigma}(\Omega)$ by
\begin{align}
B_{\Omega,\eta,\omega}u & :=P_\Omega[-(\eta + \omega\times x)\cdot\nabla u
+\omega\times u], \\
\SD(B_{\Omega,\eta,\omega}) & :=
\{u\in W^{1,q}_{0,\sigma}(\Omega)\mid
(\eta + \omega\times x)\cdot\nabla u\in L^q(\Omega)\}.
\end{align}
With the aforementioned notation, define $\CL_{\Omega,\eta,\omega}:L^q_{\sigma}(\Omega)\rightarrow
L^q_{\sigma}(\Omega)$ by
\begin{equation}
\label{def-operator}
\CL_{\Omega,\eta,\omega}u :=A_{\Omega}u+B_{\Omega,\eta,\omega} u, \qquad
\SD(\CL_{\Omega,\eta,\omega}) := \SD (A_{\Omega}) \cap \SD (B_{\Omega,\eta,\omega}).
\end{equation}
\par We are now in a position to state the main theorems. 
The following theorem implies 
that the operator $- \CL_{\Omega,\eta,\omega}$ 
generates a $C_0$-semigroup
$(e^{- t \CL_{\Omega,\eta,\omega}})_{t \ge 0}$ on 
$L^q_\sigma (\Omega)$ equipped with the $L^q$-$L^r$ estimates 
provided that 
$\lvert \eta \rvert := \sqrt{\eta_1^2 + \eta_2^2 + \eta_3^2}$
and $\lvert \omega \rvert := 
\sqrt{\omega_1^2 + \omega_2^2 + \omega_3^2}$ are bounded.  

\begin{thm}\label{thmlqlr}
Let $\Omega \subset \BR^3$ be an exterior Lipschitz domain. Let $\eta, \omega \in \BR^3$.
Then the following assertions are valid.
\begin{enumerate}
\item There exists $\varepsilon > 0$
such that for every numbers $q$ satisfying
\begin{equation}
\label{cond-q}
\bigg\lvert \frac{1}{q} - \frac12 \bigg\rvert < \frac16 + \varepsilon	
\end{equation}
the operator $- \CL_{\Omega,\eta,\omega}$ defined by \eqref{def-operator} generates a bounded
$C_0$-semigroup $(e^{- t \CL_{\Omega,\eta,\omega}})_{t \ge 0}$ on $L^q_\sigma (\Omega)$.
\item Fix $c_0>0$ and assume $\lvert \eta \rvert + \lvert \omega \rvert \leq c_0$. 
Then for every $1 < q \le r \le \infty$ with $q$ satisfying \eqref{cond-q}
there exists a constant $C$ independent of $\eta$ and $\omega$ 
(but depending on $c_0$) such that
\begin{equation}
\label{lqlr}
\lVert e^{- t \CL_{\Omega,\eta,\omega}} f \rVert_{r, \Omega}
\le C t^{-\frac32 (\frac1q - \frac1r) } \lVert f \rVert_{q,\Omega}
\end{equation}
for $t > 0$ and $f \in L^q_\sigma (\Omega)$ as well as
\begin{equation}
\label{gradlqlr}
\lVert \nabla e^{- t \CL_{\Omega,\eta,\omega}} f \rVert_{r, \Omega}
\le C t^{-\frac32 (\frac1q - \frac1r) - \frac12} \lVert f \rVert_{q, \Omega}
\end{equation}
for $0 < t \le 2$ and $f \in L^q_\sigma (\Omega)$, where 
$\lVert \enskip\cdot\enskip \rVert_{q,\Omega}
=\lVert \enskip\cdot\enskip \rVert_{L^{q}(\Omega)}$. 
If $r$ satisfies additionally
\begin{align}\label{rcondi}
\left|\frac{1}{r}-\frac{1}{2}\right|
<\frac{1}{6}+\varepsilon,
\end{align}
then there holds
\begin{align}\label{gradlqlr2}
\|\nabla e^{-t\CL_{\Omega,\eta,\omega}}f\|_{r,\Omega}\leq
Ct^{-\min\{\frac{3}{2}\left(\frac{1}{q}-\frac{1}{r}\right)
+\frac{1}{2},\frac{3}{2q}\}}\|f\|_{q,\Omega}	
\end{align}
for $t\geq 2$ and $f\in L^q_\sigma(\Omega)$.
\end{enumerate}
\end{thm}

\begin{rmk}
As is expected from the case of the smooth boundary, 
if $\omega=0$, then the Oseen semigroup 
$(e^{- t \CL_{\Omega,\eta,0}})_{t \ge 0}$ is an analytic 
$C_0$-semigroup 
on $L^q_\sigma (\Omega)$ provided that $q$ satisfies \eqref{cond-q}. 
In fact, it follows from \cite[p.300]{W23} that 
\begin{equation}
\lVert \nabla (\lambda+A_\Omega)^{-1} u\rVert_{q,\Omega}
\le \frac{C}{|\lambda|^{\frac{1}{2}}} \lVert u\rVert_{q,\Omega}
\end{equation} 
for $u \in L^q_\sigma(\Omega)$ and 
$\lambda\in \Sigma_\theta := \{z \in \BC \setminus \{0\}
\mid \lvert \arg z \rvert < \theta\}$ 
fulfilling $|\lambda|\geq \beta$ with some $\theta\in(\pi/2,\pi)$ 
and $\beta\geq 1$, which implies that 
$(I+B_{\Omega,\eta,0}
(\lambda+A_\Omega)^{-1})^{-1}: L^q_\sigma (\Omega)\rightarrow 
L^q_\sigma(\Omega)$ is a bounded operator with 
\begin{equation}
\lVert (I+B_{\Omega,\eta,0}
(\lambda+A_\Omega)^{-1})^{-1} u\rVert_{q,\Omega}\leq 2
\lVert u\rVert_{q,\Omega}, \qquad u\in L^q_\sigma(\Omega), \, \lambda\in \Sigma_\theta, 
\, \lvert \lambda \rvert \ge \beta
\end{equation}
with large $\beta$. Applying this
operator to the form 
$(\lambda+\CL_{\Omega,\eta,0})^{-1}
=(\lambda+A_\Omega)^{-1}(I+B_{\Omega,\eta,0}
(\lambda+A_\Omega)^{-1})^{-1}$ and employing the estimates 
of the resolvent $(\lambda+A_\Omega)^{-1}$
established in \cite[Thm. 1.1, (5.24)]{TW20} asserts that 
$(e^{- t \CL_{\Omega,\eta,0}})_{t \ge 0}$ is an analytic 
$C_0$-semigroup on $L^q_\sigma (\Omega)$. 
\end{rmk}

\begin{rmk}
Theorem \ref{thmlqlr} with $\eta=\omega=0$ 
coincides with the results 
given in \cite{TW20} and \cite{W23}. In particular, 
the estimate 
\eqref{gradlqlr2} with $\eta=\omega=0$ implies that the $L^r$-norm 
of the gradient of the Stokes semigroup is bounded by 
$Ct^{-3/(2q)}\|f\|_{q,\Omega}$ if $r>3$. This decay rate might be 
optimal due to the result for the case of the smooth boundary, 
see Maremonti and Solonnikov \cite{MS97}. 
Moreover, the decay rate in \eqref{gradlqlr2} is   
$t^{-3(1/q-1/r)/2-1/2}$ if $r\leq 3$. 
The restriction $r\leq 3$ seems to be optimal according to 
the result for the case of the smooth boundary, see \cite{MS97,H11}, 
in which a key observation is the relation between the 
restriction of exponent and the summability of the 
steady Stokes flow at spatial infinity. 
From this observation, it is also conjectured by 
Hishida \cite[Sec. 5]{H11} that 
the restriction $r\leq 3$ may not be optimal if $\eta\ne 0$, 
but it is not proved even for the case of the smooth boundary. 
\end{rmk}

By means of real interpolation due to Yamazaki \cite{Y00} and 
Hishida and Shibata \cite{HS09}, we may have the following result.

\begin{thm}\label{thmlqlr2}
Let $\Omega \subset \BR^3$ be an exterior Lipschitz domain.
Let $\eta, \omega \in \BR^3$.
Fix $c_0>0$ and assume $\lvert \eta \rvert + \lvert \omega \rvert \leq c_0$.
Then the following assertions are valid.
\begin{enumerate}
\item For every $1\le\rho\le \infty$ and
$1 < q \le r<\infty$ with $q$ satisfying \eqref{cond-q},
there exists a constant $C$
independent of $\eta$ and $\omega$ 
(but depending on $c_0$) such that
\begin{equation}
\label{lqlr01}
\lVert e^{- t \CL_{\Omega,\eta,\omega}} f \rVert_{r,\rho,\Omega}
\le C t^{-\frac32 (\frac1q - \frac1r) }
\lVert f \rVert_{q,\rho,\Omega}
\end{equation}
for $t > 0$ and $f \in L^{q,\rho}_\sigma (\Omega)$ as well as
\begin{equation}
\label{gradlqlr01}
\lVert \nabla e^{- t \CL_{\Omega,\eta,\omega}} f \rVert_{r,\rho,\Omega}
\le C t^{-\frac32 (\frac1q - \frac1r) - \frac12}
\lVert f \rVert_{q,\rho,\Omega}
\end{equation}
for $0 < t \le 2$ and $f \in L^{q,\rho}_\sigma (\Omega)$.
If $1\le \rho<\infty$ and $r$ satisfies additionally
\eqref{rcondi},
then there holds
\begin{align}\label{gradlqlr02}
\|\nabla e^{-t\CL_{\Omega,\eta,\omega}}f\|_{r,\rho,\Omega}\leq
Ct^{-\min\{\frac{3}{2}\left(\frac{1}{q}-\frac{1}{r}\right)
+\frac{1}{2},\frac{3}{2q}\}}\|f\|_{q,\rho,\Omega}
\end{align}
for $t\geq 2$ and $f\in L^{q,\rho}_\sigma(\Omega)$.
\item Let $q$ satisfy \eqref{cond-q}. If $q$ and $r$ satisfy
$1/q-1/r=1/3$ and $1<q\le r\le 3$, then there holds
\begin{equation}
\label{duality}
\int_0^\infty\lVert\nabla e^{-\tau\CL_{\Omega,\eta,\omega}}f
\rVert_{r,1,\Omega}\d\tau\le
C\lVert f\rVert_{q,1,\Omega}
\end{equation}
for $f\in L^{q,1}_\sigma(\Omega)$.
\end{enumerate}
Here, $L^{q,\rho} (\Omega)$ stands for the Lorentz space defined on $\Omega$, equipped with the
norm $\lVert \enskip\cdot\enskip \rVert_{L^{q,\rho} (\Omega)} = \lVert \enskip\cdot\enskip \rVert_{q, \rho, \Omega}$
and 
$L^{q,\rho}_\sigma (\Omega)$ is defined as 
$L^{q,\rho}_\sigma(\Omega)=
(L^{q_0}_{\sigma}(\Omega),L^{q_1}_{\sigma}(\Omega))_{\theta,\rho}$
with 
\begin{align}
q_0<q<q_1,\quad \frac1q=\frac{(1-\theta)}{q_0}
+\frac{\theta}{q_1},\quad  
\left|\frac{1}{q_i}-\frac{1}{2}\right|
<\frac{1}{6}+\varepsilon\quad (i=0,1),\quad 
0 < \theta < 1, \quad 1\leq \rho\leq \infty,
\end{align}
where $(\,\cdot\,,\,\cdot\,)_{\theta,\rho}$ 
denotes the real interpolation functor. 
\end{thm}

As an application of the aforementioned theorems, we 
may construct a unique global solution to Problem \eqref{eq-main}. 
To describe its result, let $\zeta \in C^\infty_0 (\BR^3)$ be a cut-off function with $0 \le \zeta \le 1$ and $\zeta = 1$ near $\pd \Omega$. Define $b$ by
\begin{equation}
\label{def-b}
b (x) := \frac12 \curl \Bigl\{\zeta (x) \Bigl( \eta \times x - \lvert x \rvert^2 \omega \Bigr)\Bigr\}
\end{equation}
Clearly, $b$ is a compactly supported smooth function defined on $\BR^3$ and there hold
$\div b = 0$ in $\Omega$ as well as $b \vert_{\pd \Omega} = \eta + \omega \times x$.
We set $u := v - b (x)$ in \eqref{eq-main}. Using the Helmholtz projection
$P_\Omega \colon L^q (\Omega) \to L^q_\sigma (\Omega)$, we infer from Theorem \ref{thmlqlr} that $u$ satisfies
\begin{equation}
\label{eq-nonlinear}
\left\{\begin{aligned}
\pd_t u + \CL_{\Omega,\eta,\omega} u + P_\Omega (u \cdot \nabla u) & = P_\Omega \CN_1 (b, u) + P_\Omega \CN_2 (b)
& \quad & \text{in $\Omega \times \BR_+$}, \\
u \vert_{t = 0} & = v_0 - b & \quad & \text{in $\Omega$},
\end{aligned}\right.
\end{equation}
where we have set
\begin{align}
\CN_1 (b, u) & := - u \cdot \nabla b - b \cdot \nabla u,
\label{def-n1}\\
\CN_2 (b) & := \Delta b + (\eta + \omega \times x) \cdot \nabla b - \omega \times b - b \cdot \nabla b,\label{def-n2}
\end{align}
respectively. Notice that the compatibility condition $\div (u \vert_{t=0}) = 0$ in $\Omega$ is fulfilled.
Then \eqref{eq-nonlinear} is formally converted to the integral equation
\begin{equation}
\label{eq-integral}
u (t) = e^{- t \CL_{\Omega,\eta,\omega}} (v_0 - b) - \int_0^t e^{- (t - s) \CL_{\Omega,\eta,\omega}}
P_\Omega \Big[u \cdot \nabla u + \CN_1 (b, u) + \CN_2 (b)\Big] \d s.
\end{equation}

The existence and uniqueness of mild solutions to Problem 
\eqref{eq-integral} reads as follows. 
\begin{thm}
\label{thmnonlinear}
Let $\Omega \subset \BR^3$ be an
exterior Lipschitz domain and $\eta, \omega \in \BR^3$. Let $\varepsilon > 0$ be as in
Theorem \ref{thmlqlr2}.
In addition, let $b (x)$ be defined by \eqref{def-b}.
There exist constants $c_1, c_2 > 0$ such that if
$\lvert \eta \rvert + \lvert \omega \rvert <c_1$ and
$v_0 - b \in L^{3,\infty}_\sigma (\Omega)$ satisfies
$\lVert v_0 - b\rVert_{3,\infty,\Omega}<c_2$,
then Problem \eqref{eq-integral}
admits a global solution with the following properties:
\begin{enumerate}
\item $u\in BC((0,\infty);L^{3,\infty}_\sigma(\Omega))\cap
C((0,\infty);L^q(\Omega))$ and
$\nabla u\in C((0,\infty);L^r(\Omega))$ for every $3<q\leq\infty$
and $3<r<\infty$.
\item $u(t)\rightarrow v_0-b$
weakly $*$ in $L^{3,\infty}(\Omega)$ as $t \searrow +0$.
\item There is a constant $C>0$ such that
\begin{equation}
\lVert u(t)\rVert_{3,\infty, \Omega} \le C \Big(\lvert\eta\rvert
+ \lvert \omega \rvert +\lVert v_0 - b\rVert_{3,\infty,\Omega} \Big)
\end{equation}
for $t>0$. 
\item The solution $u$ is even unique among solutions 
with small $\sup_{t>0}\|u(t)\|_{3,\infty}$ to  
the integral equation in the weak form:
\begin{equation}
\begin{split}
(u(t),\varphi)& =(v_0-b,e^{-t\CL_{\Omega,- \eta, - \omega}}\varphi) \\
& \quad + \int_0^t \Big(u(s)\otimes u(s)+u(s)\otimes b
+b\otimes u(s)-\CF (b),
\nabla e^{-(t-s)\CL_{\Omega,- \eta, - \omega}}\varphi \Big) \d s
\end{split}
\end{equation}
for every $\varphi \in C_{0,\sigma}^\infty(\Omega)$, 
where $\CF (b)$ is given by $\div \CF (b) = \CN_2 (b)$, i.e., 
\begin{equation}
\label{def-F}
\CF (b) =(\eta + \omega\times x)\otimes b-b\otimes b+
\begin{pmatrix}
-\omega_2 f_2 - \omega_3 f_3 & - \Delta f_3 + \omega_1 f_2
& \Delta f_2 + \omega_1 f_3 \\
\Delta f_3 + \omega_2 f_1 & -\omega_1 f_1 - \omega_3 f_3 & -\Delta f_1
+ \omega_2 f_3 \\
-\Delta f_2 + \omega_3 f_1 & \Delta f_1 + \omega_3 f_2
& -\omega_1 f_1 - \omega_2 f_2	
\end{pmatrix}
\end{equation}	
with $f(x)=(f_1(x),f_2(x),f_3(x))^\top$ defined by
\begin{equation}
f (x) = \frac12 \Bigl\{\zeta(x) \Bigl(\eta \times x - |x|^2\omega \Bigr)\Bigr\}.
\end{equation}
\end{enumerate}
\end{thm}

\begin{rmk}
Since $\CN_2 (b)$ does not have a temporal decay property, 
it seems to be difficult to derive a decay property of the
solution $u$ to \eqref{eq-integral} unless $b$ is a stationary
solution to \eqref{eq-main}, i.e., $b$ satisfies $\CN_2 (b) = 0$.
Notice that the existence of a stationary solution of
class $u_s \in L^{3,\infty} (\Omega)$ and 
$(\nabla u_s, p_s) \in L^{3 \slash 2,\infty} (\Omega)$ 
in the case of exterior Lipschitz domains
is still open as far as the authors know.
\end{rmk}
Let us give the strategy of the proof of 
Theorem \ref{thmlqlr}. We employ the idea in 
Geissert, Heck, and Hieber \cite{GHH06} to construct the 
$C_0$-semigroup $(e^{- t \CL_{\Omega,\eta,\omega}})_{t \ge 0}$ 
on $L^q_\sigma (\Omega)$ provided that $q$ 
satisfies \eqref{cond-q}. 
Namely, we consider the parametrix $(u,p)$ of the resolvent 
$(\lambda+\CL_{\Omega,\eta,\omega})^{-1}$, which consists of 
the resolvent in the whole space and the one in the bounded 
Lipschitz domain $D$ (near the boundary $\partial \Omega$)
and obeys the problem 
\begin{equation}
\label{eq-res}
\left\{\begin{aligned}
\lambda u - \Delta u - (\eta + \omega \times x) \cdot \nabla u + \omega \times u  + \nabla p & =f+\Psi(\lambda)f
& \quad & \text{in $\Omega$}, \\
\div u & =0& \quad & \text{in $\Omega$}, \\
u & =0& \quad & \text{on $\pd \Omega$}, \\
u & \to 0 & \quad & \text{as $\lvert x \rvert \to \infty$}.
\end{aligned}\right. 
\end{equation} 
By deducing the decay estimate with respect to $\lambda>0$ 
of the reminder term $\Psi(\lambda)f$, 
the term $(I+\Psi(\lambda))^{-1}$ is written by 
$(I+\Psi(\lambda))^{-1}=\sum_{j=0}^\infty(-\Psi(\lambda))^j$ 
for large $\lambda>0$. 
From this observation, we use the Laplace transform and the lemma on 
iterated convolutions \cite[Lem. 4.6]{GHH06}, 
see also Lemma \ref{lemserires} below, to construct the 
$C_0$-semigroup $(e^{- t \CL_{\Omega,\eta,\omega}})_{t \ge 0}$
on $L^q_\sigma (\Omega)$ provided that $q$ satisfies \eqref{cond-q}. 
The desired $L^q$-$L^r$-smoothing rates near the initial time 
of $(e^{- t \CL_{\Omega,\eta,\omega}})_{t \ge 0}$ 
are also induced by corresponding estimates of 
the semigroup in the whole space and in the bounded Lipschitz domain $D$. 
However, unlike the case of the smooth boundary, 
the $L^q$-$L^r$-smoothing rates of the 
gradient of the semigroup in $D$ do not
directly follow from the Gagliard--Nirenberg inequality since 
we may not expect $W^{2,q}$-regularity for the velocity field 
in $D$. Hence, in Section 3, we characterize the domains of 
fractional powers of the Stokes operator $A_{q,D}$
in terms of suitable Bessel potential spaces and employ 
its characterization to obtain the $L^q$-$L^r$-estimates 
of the gradient of the semigroup in $D$. 
In particular, we aim to characterize 
$\SD(A_{q,D}^\alpha)$ with $\alpha>1/2${, which was given in Gabel and Tolksdorf \cite{GT22} for the case of two-dimensional bounded Lipschitz domains}, even though 
the case $0<\alpha<3/4$ and $q=2$ as well as the case 
$0<\alpha\leq 1/2$ and $q$ fulfilling \eqref{cond-q} 
have been established in \cite{MM08,KW17,T18,Tthesis}. 
Then, by virtue of the aforementioned 
construction of the $C_0$-semigroup 
$(e^{- t \CL_{\Omega,\eta,\omega}})_{t \ge 0}$, we obtain 
\begin{equation}
\lVert \nabla^j 
e^{- t \CL_{\Omega,\eta,\omega}} f \rVert_{r, \Omega} 
\le C t^{-\frac32 (\frac1q - \frac1r) - \frac1j} 
\lVert f \rVert_{q, \Omega}
\end{equation} 
for $t\leq 1$ and $j=0,1$. 
\par To deduce the $L^q$-$L^r$-decay estimates of 
$(\nabla^j e^{-t \CL_{\Omega,\eta,\omega}})_{t>0}$ 
described in Theorem \ref{thmlqlr}, we will adapt 
Hishida's method in the case of $C^{1,1}$-boundary 
\cite{H18,H20}, where 
the $L^q$-$L^r$-decay estimates of the evolution operator 
$(T(t,s))_{t \ge s \ge 0}$ constructed 
by Hansel and Rhandi \cite{HR14} 
in the non-autonomous case 
(i.e., $\eta=\eta(t)$, $\omega=\omega(t)$) and of 
its adjoint $(T(t,s)^*)_{t \ge s \ge 0}$ were established. 
We thus do not rely on 
the expansion of the resolvent of $-\CL_{\Omega,\eta,\omega}$ 
near the origin to derive the large time behavior of 
$(e^{-t \CL_{\Omega,\eta,\omega}})_{t > 0}$, 
which is completely different from 
the strategy used in \cite{ES04,ES05,HS09,I89,KS98,S08,S10,W23}. 
In the following, let us summarize Hishida's method and 
make comments about how our argument differs from his method, 
although our proof is inspired by him. 
After proving the $L^r$-boundedness of 
$T(t,s)$ and $T(t,s)^*$ 
for $r\in(2,\infty)$ on account of the duality argument, 
the $L^q$-$L^r$-estimates of the solution 
to the same system in $\BR^3$ and the first energy relation, 
he obtain the $L^q$-$L^r$-decay estimates of 
$T(t,s)$ and $T(t,s)^*$ with $1<q\leq r<\infty$, 
see \cite[Thm.~2.1]{H18}. 
With these estimates at hand, the decay estimates of 
$T(t,s)$ and $T(t,s)^*$ near the boundary, which is so-called the 
local energy decay estimates, were established in \cite[Thm.~2.1]{H20}. 
However, in our situation, 
the condition \eqref{cond-q} is needed to 
define the $C_0$-semigroup 
$(e^{- t \CL_{\Omega,\eta,\omega}})_{t \ge 0}$ on 
$L^q_\sigma(\Omega)$, and hence we only know 
the $L^r$-boundedness of 
$(e^{- t \CL_{\Omega,\eta,\omega}})_{t \ge 0}$ 
for $r\in(2,\infty)$ fulfilling \eqref{rcondi}, from which we 
conclude \eqref{lqlr} under the conditions \eqref{cond-q} and 
\eqref{rcondi}, see Section 5 below. 
\par We then obtain the rate $t^{-1/2-\varepsilon_1}$ 
of the local energy decay estimates with some small	
positive constant $\varepsilon_1$ depending on $\varepsilon$, 
which is slower than the one in \cite[Prop.~6.1]{H20}, 
thereby, the $L^q$-$L^\infty$-decay estimates of 
$(e^{- t \CL_{\Omega,\eta,\omega}})_{t \ge 0}$ as well as 
the decay estimates of its gradient may not directly follow from 
\cite{H20}. To overcome this difficulty, 
in Proposition \ref{propr0} below, 
we will provide the relation between the rate of the 
local energy decay estimates and 
the range of exponent $r$ for the $L^q$-$L^r$ estimates of 
$(e^{- t \CL_{\Omega,\eta,\omega}})_{t \ge 0}$. 
This may be observed by investigating 
how the rate of the decay estimates 
near the boundary for the general date 
and the one at spatial infinity inherit from 
the one of the local energy decay estimates. 
Carrying out the bootstrap argument 
with the aid of Proposition \ref{propr0} 
leads us to the desired decay estimates of 
$(e^{- t \CL_{\Omega,\eta,\omega}})_{t \ge 0}$ 
asserted in Theorem \ref{thmlqlr}. 
It should be also emphasized that, since 
$(e^{- t \CL_{\Omega,\eta,\omega}})_{t \ge 0}$
is a $C_0$-semigroup unless $\omega=0$, 
unlike the case of the Stokes or Oseen semigroup 
\cite{I89,KS98,ES05,H21,W23}, 
it is required to analyze the asymptotic behavior of 
$(\partial_te^{- t \CL_{\Omega,\eta,\omega}})_{t \ge 0}$ 
and the pressure in the Sobolev space of order $(-1)$ 
over the bounded domain near the boundary to 
deduce the decay estimates at spatial infinity 
via cut-off procedure as in \cite{HS09,H20}. 
\par Combined Theorem \ref{thmlqlr} with real 
interpolation, we may show Theorem \ref{thmlqlr2}. 
Then we may construct a solution to the nonlinear problem 
in weak-Lebesgue spaces by Yamazaki's method \cite{Y00} 
(cf. Hishida and Shibata \cite{HS09}). 
In addition, the deduction of better regularities of the solution 
by identifying this solution with a local solution 
in a neighborhood of each time may be done by the argument due to 
Kozono and Yamazaki \cite{KY98}. 
Although these methods are standard, 
we will give the outline of the proof of 
Theorem \ref{thmnonlinear} 
in the last section of the present paper 
for the reader's convenience. 
\par This paper is organized as follows.
In the next section, we prepare 
the notation and preliminary results that we will use throughout 
this paper. In Section \ref{sec-3}, we discuss properties of solutions to
the linearized problem on bounded Lipschitz domains. Using a cut-off technique,
we prove a generation of the $C_0$-semigroup $(e^{- t \CL_{\Omega,\eta,\omega}})_{t\ge0}$ 
in Section \ref{sec-4}. Section \ref{sec-5} is devoted to showing the $L^q$-$L^r$-estimates
for the $C_0$-semigroup $(e^{- t \CL_{\Omega,\eta,\omega}})_{t\ge0}$ with an additional 
restriction for $r$, but this restriction may be removed in Section \ref{sec-6}. 
Finally, in Section \ref{sec-7}, 
we deal the proof of Theorem \ref{thmnonlinear}
on account of Theorems \ref{thmlqlr} and \ref{thmlqlr2}.

\section{Preliminaries}
\subsection{Notation}
We fix the notation that we will use throughout the paper.
As usual, $\BC$, $\BN$, and $\BR$ stand for the set of all complex, natural, and real numbers, respectively.
In addition, $\BR_+$ stands for $(0, \infty)$,
whereas $\BC_+$ stands for the set of all $z \in \BC$ such that $\Re z > 0$, where
$\Re z$ means the real part of $z$. 
We denote various constants by $C$ 
and they may change from line to line 
whenever there is no confusion. 
The constant dependent on $A,B,\ldots$ is denoted by $C(A,B,\ldots)$. 
\par
Given two Banach spaces $X$ and $Y$, the Banach space consisting of
all bounded linear operators from $X$ into $Y$ is denoted by $\SL (X, Y)$.
In particular, if $X=Y$, we write $\SL (X) = \SL (X, X)$ for short. By $(\,\cdot\, , \,\cdot\,)_E$
we denote the duality product over the domain $E \subset \BR^3$. In addition, the dual space of $X$
is denoted by $X^*$. \par
Given an open set $E \subset \BR^3$, $1 < q < \infty$, $1 \le \rho \le \infty$, and $k \in \BN$, the
standard Lebesgue, Lorentz, and Sobolev spaces are denoted by $L^q (E)$, $L^{q,\rho} (E)$, and $W^{k, q} (E)$, respectively.
Given a Banach space $X$, its norm is denoted by $\lVert\enskip \cdot \enskip \rVert_X$,
and in particular, we abbreviate the norm
$\lVert \enskip\cdot\enskip \rVert_{L^q (E)} = \lVert \enskip\cdot\enskip \rVert_{q, E}$ for the case $X = L^q (E)$,
whereas $\lVert \enskip\cdot\enskip \rVert_{L^{q,\rho} (E)} = \lVert \enskip\cdot\enskip \rVert_{q, \rho, E}$ for the case $X = L^{q,\rho} (E)$.
The set of all smooth functions with compact support in $E \subset \BR^3$ is
denoted by $C^\infty_0 (E)$. For $k \in \BN$, the completion of $C^\infty_0 (E)$
in $W^{k,q} (E)$ is denoted by $W^{k,q}_0 (E)$. Sobolev spaces of negative order $W^{-1,q} (E)$ is \textit{defined} by the dual space of $W_0^{1,q'} (E)$,
where $q' := q \slash (q - 1)$ is the H\"older conjugate exponent of $q$.
We further define the set of all compactly supported smooth and solenoidal vector fields in $E \subset \BR^3$ as
\begin{equation} 
C^\infty_{0,\sigma} (E) := \{f \in C^\infty_0 (E)^3
\mid \div f = 0 \enskip \text{in $E$}\}.
\end{equation}
Then $L^q_\sigma (E)$ and 
$W^{1,q}_{0,\sigma} (E)$ stand for the completion of
$C^\infty_{0,\sigma} (E)$ in $L^q (E)$ and 
$W^{1,q}(E)$, respectively, endowed with their natural norms. 
For $q$ fulfilling \eqref{cond-q}, 
the solenoidal Lorentz spaces 
$L^{q,\rho}_\sigma (E)$ are defined as 
$L^{q,\rho}_\sigma(E)=
(L^{q_0}_{\sigma}(E),L^{q_1}_{\sigma}(E))_{\theta,\rho}$
with 
\begin{align}
q_0<q<q_1,\quad \frac1q=\frac{(1-\theta)}{q_0}
+\frac{\theta}{q_1},\quad  
\left|\frac{1}{q_i}-\frac{1}{2}\right|
<\frac{1}{6}+\varepsilon\quad (i=0,1),\quad
0 < \theta < 1, \quad
1\leq \rho\leq \infty,
\end{align}
where $(\cdot,\cdot)_{\theta,\rho}$ denotes the real interpolation 
functor. 
\par 
For $R > 0$, set $B_R (0) := \{x \in \BR^3 \mid \lvert x \rvert < R\}$.
If $R$ is sufficiently large such that $B_R (0) \supset \BR^3 \setminus \Omega$ holds,
then we will write $\Omega_R := \Omega \cap B_R (0)$. In addition, for $L > R$, we set
$A_{R, L} := \{x \in \BR^3 \mid R < \lvert x \rvert < L \}.$

\subsection{Whole space problem}
For $1 < q < \infty$, define a $C_0$-semigroup $(T_{\BR^3,\eta,\omega}(t))_{t\ge0}$
on $L^q_{\sigma}(\BR^3)$ by
\begin{equation}
\label{r3semi}
T_{\BR^3,\eta,\omega}(t)f :=Q(t)^\top 
e^{t\Delta}f\left(Q(t)x+\int_0^t Q(\tau)\eta\,\d\tau\right),
\quad
e^{t\Delta}f(x) =
\left(\frac{1}{4\pi t}\right)^{\frac{3}{2}}\int_{\BR^3}
e^{-\frac{|x-y|^2}{4t}}f(y)\d y,
\end{equation}
which is related to transformation, see Galdi \cite[Sec.~1]{G02}. 
Here, $Q(t)$ is a $3\times 3$ orthogonal matrix fulfilling 
$\d Q(t)/\d t=-\omega \times Q$. 
If $f$ satisfies $\div f=0$ in the sense of distributions, 
we see that $\nabla p=0$ and
that $u(x,t):= T_{\BR^3,\eta,\omega}(t)f$ solves the problem:
\begin{align*}
\left\{
\begin{aligned}
\pd_t u -\Delta u -(\eta + \omega\times x)\cdot\nabla u
+\omega\times u +\nabla p & = f, &\quad & \text{in $\BR^3 \times \BR_+$}, \\
\div u & = 0 &\quad & \text{in $\BR^3 \times \BR_+$},\\
u \vert_{t = 0} & = f & \quad & \text{in $\BR^3$}.
\end{aligned}\right.
\end{align*}
According to Kobayashi and Shibata \cite{KS98}, 
Hishida and Shibata \cite{HS09}, and Shibata \cite{S08,S10}, 
we have the following $L^q$-$L^r$-estimates 
of $(T_{\BR^3,\eta,\omega}(t))_{t\ge0}$ from those of the 
heat semigroup. 
\begin{lem}
Let $\eta, \omega \in \BR^3$.
Given $c_0>0$, assume $\lvert \eta \rvert + \lvert \omega \rvert \leq c_0$.
Let $1<q\leq r\leq\infty$, $q \ne \infty$, and $j \in \BN \cup \{0\}$. Then there holds
$\nabla^j T_{\BR^3,\eta,\omega}(t)f\in C((0,\infty);L^r(\BR^3))$ with
\begin{align}\label{lqlrwhole}
\|\nabla^jT_{\BR^3,\eta,\omega}(t)f\|_{r,\BR^3}
\leq Ct^{-\frac{3}{2}(\frac{1}{q}-\frac{1}{r})-\frac{j}{2}}
\|f\|_{q,\BR^3}
\end{align}
for $t>0$ and $f\in L^q_\sigma(\BR^3)$.
\end{lem}

Moreover, we recall the following result 
due to Hishida \cite{H20}.
\begin{lem}[\!\!{\cite[Lem. 3.2]{H20}}] 
\label{lempartialtwhole}
Let $\eta, \omega \in \BR^3$. Given $c_0>0$, assume $\lvert \eta \rvert + \lvert \omega \rvert \leq c_0$.
Let $1<q<\infty$ and $R,T>0.$ Given $f\in L^q_\sigma(\BR^3)$, set
$u(t)=T_{\BR^3,\eta,\omega}(t)f$.
Then there holds $u\in C^1((0,\infty);W^{-1,q}(B_R(0)))$ with
\begin{align*}
\|\partial_tT_{\BR^3,\eta,\omega}(t)f\|_{W^{-1,q}(B_R(0))}
\leq Ct^{-\frac{1}{2}}\|f\|_{q,\BR^3}
\end{align*}
for $t\leq T$ and $f\in L^q_\sigma(\BR^3)$. Here, a constant $C$ is independent of $T$.
Furthermore, there holds
\begin{align*}
(\partial_tu,\psi)_{B_R(0)}+(\nabla u+u\otimes
(\eta + \omega\times x)-(\omega\times x)\otimes u,\nabla \psi)_{B_R(0)}=0
\end{align*}
for $t>0$ and $\psi\in W_0^{1,q'}(B_R(0))^3$.
\end{lem}

\subsection{Bogovski\u{\i} operator}
To carry out the cut-off procedure, we introduce the Bogovski\u{\i}
operator in order to keep the divergence-free condition.
This operator was constructed by Bogovski\u{\i} \cite{B79}, see also
Borchers and Sohr \cite{BS90} and Galdi \cite[Sec. III.3]{G11}.
To describe the result on the Bogovski\u{\i} operator, let $D\subset \BR^3$
be a bounded Lipschitz domain and set $L^q_0(D):=
\{f\in L^q(D)\,|\,\int_Df\d x=0\}.$
\begin{prop}
\label{prop-Bogovskii}
Let $D\subset \BR^3$ be a bounded Lipschitz domain, $1<q<\infty$.
Then there exists a continuous operator
$\BB_D \colon L^q(D) \to W_0^{1,q}(D)$ with
$\BB_D\in \SL(W_0^{k,q}(D),W_0^{k+1,q}(D))$ 
for positive integers $k$
such that
\begin{align*}
\div \BB_D[f]=f, \qquad f\in L^q_0(D).
\end{align*}
\end{prop}

We also need the following proposition to deal with the Bogovski\u{\i}
operator in Sobolev spaces of negative order. This was proved by
Geissert, Heck, and Hieber \cite{GHH062}, where the operator
$\BB_D$ extends to a bounded operator
from $W^{1,q'}(D)^*$ to $L^q(D)$.

\begin{prop}
\label{prop-Bogovskii2}
Let $\BB_D$ be the operator defined in Proposition \ref{prop-Bogovskii}.
The operator $\BB_D$ may also extend to a bounded operator
from $W^{1,q'}(D)^*$ to $L^q(D)$.
\end{prop}

\section{Interior problem}\label{sec-3}
Throughout this section, let $D\subset \BR^3$ be a bounded Lipschitz domain. Following Shen \cite{S12},
we define for $1 < q < \infty$ the Stokes operator
$A_{D} \colon L^q_{\sigma}(D)\rightarrow L^q_{\sigma}(D)$ to be
\begin{align*}
A_Du &:=-\Delta u +\nabla p,\\
\SD(A_D) & :=\{u \in W^{1,q}_{0,\sigma}(D)
\mid\exists \; p \in L^q(D)
\enskip\text{such that}\enskip-\Delta u+\nabla p
\in L^q_\sigma(D)\},
\end{align*}
where the relation $-\Delta u+\nabla p
\in L^q_\sigma(D)$ is understood in the sense of distributions.
It is well-known that $-A_D$ generates a bounded analytic semigroup on
$L^q_\sigma (\Omega)$ for all $q$ satisfying
\begin{equation}
\label{cond-skotes}
\bigg\lvert \frac1q - \frac12 \bigg\rvert < \frac16 + \varepsilon_0,
\end{equation}
where $\varepsilon_0 > 0$ is a number depending on $D$,
and furthermore, the $H^\infty$-calculus of $A_D$ is bounded,
see \cite[Thm.~1.1]{S12} and \cite[Thm.~16]{KW17} 
for the details, respectively.
We also record the following known propositions. 

\begin{prop}[\!\!{\cite[Thm. 10.6.2]{MW12}}]
	\label{propstokeswell}
There exists
a constant $\delta_1\in (0,1]$ such that if either
\begin{equation}
\label{condidelta0}
\frac{1}{q}<\frac{s}{2}+\frac{1+\delta_1}{2}, \qquad 0<s<\delta_1
\end{equation}
or
\begin{equation}
\label{condidelta1}
\frac{\delta_1}{2}<\frac{1}{q}-\frac{s}{2}<\frac{1+\delta_1}{2},
\qquad \delta_1\leq s<1
\end{equation}
holds,
then for $f\in H^{s+\frac{1}{q}-2,q}(D)$ the Stokes problem
\begin{align*}
-\Delta u+\nabla p=f,\quad \div  u=0\quad {\rm in}~D,
\qquad u|_{\partial D}=0
\end{align*}
has a unique solution $(u,p)\in H^{s+\frac{1}{q},q}(D)\times
H^{s+\frac{1}{q}-1,q}(D)$.
Moreover, there exists a constant $C>0$ such that
\begin{align*}
\|u\|_{H^{s+\frac{1}{q},q}(D)}+\|p\|_{H^{s+\frac{1}{q}-1,q}(D)}
\leq C\|f\|_{H^{s+\frac{1}{q}-2,q}(D)}.
\end{align*}
Here, for $1 < q < \infty$ and $s \in \BR$, the Bessel potential space defined on $D$ is 
denoted by $H^{s,q} (D)$.
\end{prop}

\begin{rmk}\label{rmkembe}
If $f\in L^q_\sigma(D)$, then we have $f\in H^{s+1/q-2,q}(D)$ for all $s\in(0,1)$.
Hence, together with Proposition \ref{propstokeswell}, we deduce
\begin{equation}
\label{abessel}
\SD_q(A_D)\subset\bigcap_{t<\min
\left\{1+\frac{1}{q},\frac{3}{q}+\delta_1\right\}}H^{t,q}(D).
\end{equation}
In fact, if $\delta_1=1,$ then \eqref{condidelta1} is void,
while the \eqref{condidelta0} is valid
for all $s\in(0,1)$, thereby, we infer
\begin{align*}
\SD_q(A_D)\subset\bigcap_{t<1+\frac{1}{q}}H^{t,q}(D).
\end{align*}
If $\delta_1\in(0,1)$, then the condition \eqref{condidelta1}
is equivalent to
$2/q-(1+\delta_1)<s<2/q+\delta_1$ with $\delta_1\leq s<1$. 
Thus we may take $s$ so that $s<\min\{1,2/q+\delta_1\}$ 
is valid, which implies \eqref{abessel}. 
The relation \eqref{abessel} will be used in the proof of 
the characterization of domains of fractional powers 
for the Stokes operator, 
see the second assertion of Lemma \ref{lemfractional} below. 
\end{rmk}

\begin{prop}[\!\!{\cite[Prop. 2.16]{MM08}},
{\cite[Lem. 4.4]{TW20}}]\label{propdelta2}
\label{prop-projection}
There exists a constant $\delta_2\in(0,1]$ such that if
either
\begin{align*}
0\leq 2\alpha<\frac{1}{q},\quad q\leq \frac{2}{1-\delta_2}
\qquad {\rm or}\qquad
0\leq 2\alpha<\frac{3}{q}-1+\delta_2,\quad \frac{2}{1-\delta_2}<q,
\end{align*}
then the Helmholtz projection $P_D$ is a bounded operator
from $H^{2\alpha,q}(D)$ to $\SD(A_D^{\alpha})$.
\end{prop}

We shall agree on the following convention for $\varepsilon_1$
and $q$, see \cite[Thm.~1.1]{T18}.
\begin{conv}
\label{conv-1}
Let $D\subset \BR^3$ be a bounded Lipschitz domain.
Let $\varepsilon_0 > 0$ be the number appearing in \eqref{cond-skotes}.
Furthermore, let $\delta_1$ and $\delta_2$ be the numbers given in Propositions
\ref{propstokeswell} and \ref{propdelta2}, respectively.
Let $\varepsilon_1\in(0,\min\{\varepsilon_0, \delta_1,\delta_2\})$
be such that for all $q\in(1,\infty)$ satisfying
\begin{align}\label{p}
\left|\frac{1}{q}-\frac{1}{2}\right|<\frac{1}{6}+\varepsilon_1,
\end{align}
there holds
\begin{align}\label{aresol}
\|\nabla^j(\lambda+A_D)^{-1}u\|_{q,D}
\leq C|\lambda|^{-1+\frac{j}{2}}\|u\|_{q,D}
\end{align}
with some constant $C=C(D,\theta,q)$
for $u\in L^q_\sigma(D)$ and $\lambda\in \Sigma_\theta$, $\theta \in (0, \pi)$, and that
\begin{align}\label{afrac}
\SD(A_{D}^{\frac{1}{2}})=W^{1,q}_{0,\sigma}(D)
\end{align}
with equivalence of the respective norms. Here, $\Sigma_\theta := \{z \in \BC \setminus \{0\}
\mid \lvert \arg z \rvert < \theta\}$ is a sector with the opening angle $\theta \in (0, \pi)$.
\end{conv}
Given $\eta = (\eta_1, \eta_2, \eta_3)^\top \in \BR^3$ and $\omega=(\omega_1,\omega_2,\omega_3)^\top \in\BR^3$,
we define the operator
$B_{D, \eta, \omega} \colon L^q_{\sigma}(D)\rightarrow L^q_{\sigma}(D)$ by
\begin{align*}
B_{D, \eta, \omega}u:=-(\eta + \omega\times x)\cdot\nabla u
+\omega\times u,\qquad
\SD(B_{D, \eta, \omega}):= W^{1,q}_{0,\sigma}(D).
\end{align*}
Then define $\CL_{D,\eta,\omega} \colon L^q_{\sigma}(D)\rightarrow L^q_{\sigma}(D)$ by
\begin{align}
\label{def-operator-D}
\CL_{D,\eta,\omega}u:= A_{D}u+B_{D, \eta, \omega} u,\qquad
\SD (\CL_{D,\eta,\omega}) :=\SD (A_{D}).
\end{align}

This section aims to show that 
$- \CL_{D,\eta,\omega}$ generates a bounded analytic 
$C_0$-semigroup $(e^{- t \CL_{D,\eta,\omega}})_{t \ge 0}$ 
on $L^q_\sigma (D)$. For later use, we also prepare 
the $L^q$-$L^r$-smoothing rate of 
$(e^{- t \CL_{D,\eta,\omega}})_{t \ge 0}$ 
and the estimates near $t=0$ of 
$\partial_t e^{- t \CL_{D,\eta,\omega}}$ and of 
the associated pressure with 
$e^{- t \CL_{D,\eta,\omega}} f$.

\begin{thm}\label{proplqlrsmooth}
Let $\varepsilon_1$ and $q$ be subject to Convention \ref{conv-1}
and $\eta, \omega \in \BR^3$.
Given $c_0>0$, we assume $\lvert \eta \rvert + \lvert \omega \rvert \leq c_0$.
Then the following assertions are valid.
\begin{enumerate}
\item The operator $- \CL_{D,\eta,\omega}$ defined by \eqref{def-operator-D} generates a bounded analytic
$C_0$-semigroup $(e^{- t \CL_{D,\eta,\omega}})_{t \ge 0}$ on $L^q_\sigma (D)$.
\item Let $T>0$, $j=0,1$, and let $r\leq\infty~
({\rm resp.}~r<\infty)$ if $j=0~({\rm resp.}~j=1)$.
Then there exists a constant $C=C(D,c_0,q,r,T,j)>0$ such that
\begin{align}\label{lqlrsmooth}
\|\nabla^j e^{-t\CL_{D,\eta,\omega}}f\|_{r,D}
\leq Ct^{-\frac{3}{2}(\frac{1}{q}-\frac{1}{r})-\frac{j}{2}}
\|f\|_{q,D}
\end{align}
for $0<t\leq T$ and $f\in L^q_\sigma(D)$.
\item Let $\alpha=\alpha(q)>0$ satisfy
\begin{align}\label{alphacondi}
2\alpha<1-\frac{1}{q}\quad{\rm if}~q\geq \frac{2}{1+\delta_2}
\qquad {\rm and}\qquad 2\alpha<2-\frac{3}{q}+\delta_2
\quad{\rm if}~q<\frac{2}{1+\delta_2},
\end{align}
where $\delta_2\in(0,1]$ is the same number as in Proposition \ref{propdelta2}.
Suppose that $r\in[q,\infty)$ fulfills $1/q-1/r\leq (1-2\alpha)/3$.
Then there holds $e^{-t\CL_{D,\eta,\omega}}f
\in C^1((0,\infty);W^{-1,r}(D))$ and,
for each $T>0$,
there exists a constant $C=C(D, c_0,q,r,\alpha,T)>0$
such that
\begin{align}
\|p(t)\|_{r,D}
& \leq Ct^{-1+\alpha}\|f\|_{r,D},\label{pest}\\
\|\partial_te^{-t\CL_{D,\eta,\omega}}f\|_{W^{-1,r}(D)}
& \leq Ct^{-1+\alpha}\|f\|_{r,D}\label{partialtest}
\end{align}
for $f\in L^q_\sigma(D)\cap L^r(D)$ and $0<t\leq T$,
where $p\in L^q_0(D)$ is an associated pressure
with $e^{-t\CL_{D,\eta,\omega}}f$.
\end{enumerate}
\end{thm}

To prove Theorem \ref{proplqlrsmooth}, the following lemma is crucial.
In particular, it is necessary in our argument to use the characterization of the domain
of the fractional power $\CL_{D,\eta,\omega}^\alpha$ of $\CL_{D,\eta,\omega}$ with $\alpha > 1 \slash 2$.
\begin{lem}\label{lemfractional}
Let $\varepsilon_1$ and $q$ be subject to Convention \ref{conv-1}
and $\eta, \omega \in \BR^3$.
Given $c_0>0$, assume $\lvert \eta \rvert + \lvert \omega \rvert \leq c_0.$
\begin{enumerate}
\item The operator $\CL_{D,\eta,\omega}$
admits a bounded $H^\infty$-calculus.
For the domain of $\CL^\alpha_{D,a}$, there holds
\begin{align*}
\SD(\CL_{D,\eta,\omega}^\alpha)=[L^q_\sigma(D),
\SD(\CL_{D,\eta,\omega})]_\alpha
=[L^q_\sigma(D),\SD (A_D)]_\alpha=
\SD(A_D^\alpha),
\qquad 0<\alpha<1,
\end{align*}
with equivalence of the respective norms,
where $[\,\cdot\,,\,\cdot\,]_\alpha$ denotes complex interpolation of
order $\alpha$.
\item For every $\alpha>0$ fulfilling
\begin{equation}
\label{alpha2}
0<\alpha<\min\left\{\frac{1}{2}+\frac{3}{2q},\frac{3}{2q}+\frac{\delta_1}{2}\right\},
\end{equation}
where $\delta_1$ is the same number as in Proposition \ref{propstokeswell},
there holds
\begin{align*}
\SD(\CL_{D,\eta,\omega}^\alpha)=\SD(A_D^\alpha)=H^{2\alpha,q}_{0,\sigma}(D).
\end{align*}
\end{enumerate}
\end{lem}

To deduce the properties of the fractional power of $\CL_{D,\eta,\omega}$,
we begin by introducing the following lemma.
\begin{lem}\label{delta}
Let $\varepsilon_1$ and $q$ be subject to Convention \ref{conv-1} 
and $\eta, \omega \in \BR^3$.
Given $c_0>0,$ assume $\lvert \eta \rvert + \lvert \omega \rvert \leq c_0$. For each $\delta>0,$
there exists a constant $C=C(q,c_0,\delta) > 0$ independent of  $\eta$ and $\omega$
such that
\begin{align*}
\|B_{D, \eta, \omega}u\|_{q,D}\leq \delta\|A_{D}u\|_{q,D}+C\|u\|_{q,D}
\end{align*}
for $u\in \SD(A_{D}).$
\end{lem}
\begin{proof}
It follows from \eqref{afrac} and the moment inequality that
\begin{align*}
\|B_{D, \eta, \omega}u\|_{q,D}\leq Cc_0\|u\|_{W^{1,q}(D)} &\leq
Cc_0(\|A_{D}^{\frac{1}{2}}u\|_{q,D}+\|u\|_{q,D}) \\
&\leq
Cc_0(\|A_D u\|_{q,D}^{\frac{1}{2}}
\|u\|_{q,D}^{\frac{1}{2}}+\|u\|_{q,D})\\
&\leq \delta\|A_D u\|_{q,D}+\left(\frac{(Cc_0)^2}{4 \delta}
+Cc_0\right)\|u\|_{q,D}
\end{align*}
for $u\in\SD(A_{D})$, where we have used
\begin{align}
\|A_{D}u\|_{q,D}^{\frac{1}{2}}
\|u\|_{q,D}^{\frac{1}{2}}
& =\left(\frac{2 \delta}{Cc_0}\|A_{D}u\|_{q,D}
\right)^{\frac{1}{2}}
\left(\frac{Cc_0}{2\delta}\|u\|_{q,D}\right)^{\frac{1}{2}} \\
& \leq \frac{\delta}{Cc_0}\|A_{D}u\|_{q,D}+
\frac{Cc_0}{4 \delta}\|u\|_{q,D}
\end{align}
in the last inequality. The proof is complete.
\end{proof}

\begin{proof}[Proof of Lemma \ref{lemfractional}]
We note that there holds
$\SD(\CL_{D,\eta,\omega})=\SD(A_D)$
with the equivalence of norms.
To prove the first assertion, it suffices to prove that
the operator $\CL_{D,\eta,\omega}$ is sectorial and invertible
due to Corollary 3.3.15 and Theorem 3.3.7 in \cite{PS16}.
To this end, we first prove that
there exist constants $R$ and $C$, 
which are independent of $\eta$ and $\omega$, such that
\begin{align}\label{lresol}
\|\nabla^j(\lambda+\CL_{D,\eta,\omega})^{-1}u\|_{q,D}\leq
C|\lambda|^{-1+\frac{j}{2}}
\|u\|_{q,D}, \qquad j \in \{0, 1\},
\end{align}
for all $u\in L^q_{\sigma}(D)$ and $\lambda\in
\{\lambda\in\Sigma_\theta\mid|\lambda|\geq R\}$
with $\theta \in (0, \pi)$.
It follows from Lemma \ref{delta} that
\begin{align*}
\|B_{D, \eta, \omega}(\lambda+A_{D})^{-1}u\|_{q,D}&\leq
\delta\|A_{D}(\lambda+A_{D})^{-1}u\|_{q,D}
+C\|(\lambda+A_{D})^{-1}u\|_{q,D}\\
&\leq \delta(1+C_1)\|u\|_{q,D}+C_2|\lambda|^{-1}\|u\|_{q,D}
\end{align*}
for all $u\in L^q_{\sigma}(D)$ and $\lambda\in\Sigma_\theta$,
where $C_1$ is independent of $\delta$,
while $C_2$ depends on $\delta.$
Thus, taking $\delta>0$ and $R_0>0$
such that $\delta(1+C_1)<1/4$ and $C_2/R_0<1/4$ are valid, then
for all $\lambda \in \{\lambda\in\Sigma_\theta \mid|\lambda|\geq R_0\}$
we observe
\begin{equation}
\|(I+B_{D, \eta, \omega}(\lambda+A_{D})^{-1})^{-1}\|_{\SL(L^q_{\sigma}(D))}
\leq 2
\end{equation}
Together with the identity
$\lambda+A_{D}+B_{D, \eta, \omega}=(I+B_{D, \eta, \omega}(\lambda+A_{D})^{-1})
(\lambda+A_{D})$
and the estimate \eqref{aresol}, we have
$\{\lambda\in\Sigma_\theta \mid|\lambda|\geq R_0\}\subset
\rho(-\CL_{D,\eta,\omega})$ as well as \eqref{lresol}
for all $u\in L^q_{\sigma}(D)$ and $\lambda\in
\{\lambda\in\Sigma_\theta\mid|\lambda|\geq R_0\}$.
Here, $\rho (- \CL_{D,\eta,\omega})$ stands for the resolvent set of $- \CL_{D,\eta,\omega}$. \par
We next prove that
\begin{align}\label{resol}
\overline{\BC}_+:=\{\lambda\in\BC\mid\Re \lambda\geq 0\}\subset
\rho(-\CL_{D,\eta,\omega})
\end{align}
and that
$\|(\lambda+\CL_{D,\eta,\omega})^{-1}u\|_{q,D}\leq C\|u\|_{q,D}$
for all $u\in L^q_{\sigma}(D)$ and $\lambda\in \overline{B}_{R_0,+}:=
\{\lambda\in \overline{\BC}_+\mid|\lambda|\leq R_0\}$
with some fixed constant $R_0 > 0$.
Given $\lambda\in\overline{\BC}_+$, we define
\begin{align*}
K_{D,\eta,\omega}(\lambda)u:=(\lambda-R_0)(R_0+\CL_{D,\eta,\omega})^{-1}u,\qquad
\SD(K_{D,\eta,\omega}(\lambda)):=L^q_{\sigma}(D).
\end{align*}
Then we have
\begin{align*}
\lambda u+\CL_{D,\eta,\omega}u=(I+K_{D,\eta,\omega}(\lambda))
(R_0+\CL_{D,\eta,\omega})u
\end{align*}
for all $u\in\SD(\CL_{D,\eta,\omega})$ and $\lambda \in \overline{\BC}_+$,
where $I$ denotes the identity mapping.
Since  there holds 
\begin{equation}
\|K_{D,\eta,\omega}(\lambda)u\|_{W^{1,q}(D)}\leq C_{R_0,\lambda} \|u\|_{q,D}
\qquad \text{for $u\in L^q_{\sigma}(D)$},
\end{equation}
a compactness of the embedding
$W^{1,q}_{0,\sigma} (D) \hookrightarrow L^q_\sigma (D)$ implies that
$K_{D,\eta,\omega}(\lambda) \colon L^q_\sigma(D)\rightarrow L^q_\sigma(D)$ is a
compact operator. Hence, from the Fredholm alternative theorem,
it suffices to prove the injectivity of $I+K_{D,\eta,\omega}(\lambda)$
in order to observe $(I+K_{D,\eta,\omega}(\lambda))^{-1}\in \SL(L^q_\sigma(D))$.
Suppose $(I+K_{D,\eta,\omega}(\lambda))u=0$ for $u \in L^q_\sigma (D)$.
Then it follows from the definition of $K_{D,\eta,\omega} (\lambda)$ that
$u=-(\lambda-R_0)(R_0+\CL_{D,\eta,\omega})^{-1}u\in\SD(\CL_{D,\eta,\omega})$.
Since $(I+K_{D,\eta,\omega}(\lambda))u=0$ yields
$\lambda u+\CL_{D,\eta,\omega}u=0$, we find that
\begin{align}\label{distribution}
(\lambda u,\varphi)_D+\sum_{\ell=1}^3(\nabla u_\ell,\nabla \varphi_\ell)_D +
(p,\div \varphi)_D+(-(\eta + \omega\times x)\cdot\nabla u
+\omega\times u,\varphi)_D
\end{align}
for all $\varphi=(\varphi_1,\varphi_2,\varphi_3)
\in C_0^\infty(D)^3$.
Suppose $q\geq 2$.
Since $u \in\SD(\CL_{D,\eta,\omega})\subset W_{0,\sigma}^{1,q}(D)$
satisfies
$u\in W_{0,\sigma}^{1,2}(D)$ and since
$C_{0,\sigma}^\infty(D)$ is dense in $W_{0,\sigma}^{1,2}(D)$,
we see from \eqref{distribution} that
\begin{align*}
\lambda\|u\|_{2,D}+\|\nabla u\|_{2,D}^2+
(-(\eta + \omega\times x)\cdot\nabla u
+\omega\times u,u)_D=0.
\end{align*}
Moreover, we infer from
$\Re ((\eta + \omega\times x)\cdot\nabla u,u)_D=0$ and
$(\omega\times u,u)_D=0$
that
\begin{equation}
(\Re \lambda)\|u\|_{2,D}+\|\nabla u\|_{2,D}^2=0,
\end{equation}
which yields $\|\nabla u\|_{2,D}=0$ if $\Re\lambda\geq 0$.
Together with the boundary condition $u|_{\partial D}=0$,
we conclude $u=0$.
Hence, we have \eqref{resol} for $q \geq 2$
fulfilling \eqref{p} and $\lvert \eta \rvert + \lvert \omega \rvert \leq c_0$.
For the case $1<q<2$, since we find $q'>2$ and
$\CL_{D,\eta,\omega}^*=A_{D}+B_{D,- \eta, - \omega}$, the above argument asserts
$\overline{\BC}_+\subset\rho(-\CL_{D,\eta,\omega}^*)$
for $1<q<2$ fulfilling \eqref{p} and $\lvert \eta \rvert + \lvert \omega \rvert \leq c_0$.
Since we have $\rho(-\CL_{D,\eta,\omega})=\rho(-\CL_{D,\eta,\omega}^*)$ (cf.
\cite[Cor. A.4.3]{Hbook})
we see that \eqref{resol} holds independent of $q$.
\par Fix $\lambda_0 \in \overline{B}_{R_0,+}$. Then we may write
\begin{equation}
(\lambda+\CL_{D,\eta,\omega})^{-1}=(\lambda_0+\CL_{D,\eta,\omega})^{-1}
[I+(\lambda-\lambda_0)
(\lambda_0+\CL_{D,\eta,\omega})^{-1}]^{-1},
\end{equation}
which yields
\begin{equation}
\|(\lambda+\CL_{D,\eta,\omega})^{-1}\|_{\SL(L^q_\sigma(D))}\leq
2\|(\lambda_0+\CL_{D,\eta,\omega})^{-1}\|_{\SL(L^q_\sigma(D))}
\end{equation}
provided that
$|\lambda-\lambda_0|<1/2
\|(\lambda_0+\CL_{D,\eta,\omega})^{-1}\|_{\SL(L^q_\sigma(D))}.$
Since $\overline{B}_{R_0,+}$ is a compact set and there holds
\begin{align}
\overline{B}_{R_0,+}\subset
\bigcup_{\lambda_0\in\overline{B}_{R_0,+}}
\left\{\lambda\in \overline{B}_{R_0,+}\relmiddle|
|\lambda-\lambda_0|
<\frac{1}
{2\|(\lambda_0+\CL_{D,\eta,\omega})^{-1}\|_{\SL(L^q_\sigma(D))}}
\right\},
\end{align}
there exists a sequence
$\{\lambda_k\}_{k=1}^m\subset \overline{B}_{R_0,+}$ such that
\begin{align*}
\|(\lambda+\CL_{D,\eta,\omega})^{-1}\|_{\SL(L^q_\sigma(D))}
\leq 2\max_{1\leq k \leq m}
\|(\lambda_k +\CL_{D,\eta,\omega})^{-1}\|_{\SL(L^q_\sigma(D))}
\end{align*}
for all $\lambda\in \overline{B}_{R_0,+}$. We thus conclude that
the operator $\CL_{D,\eta,\omega}$ is sectorial and invertible, thereby,
the proof of the first assertion is complete.
\par It remains to prove the second assertion. 
The following proof is similar to the one by Gabel and Tolksdorf 
\cite[Thm. 3.17]{GT22}.
From the first assertion and
the interpolation result in Mitrea and Monniaux \cite[Thm. 2.12]{MM08},
we have
\begin{equation}
\label{inclusion}
H^{2s,r}_{0,\sigma}(D)=[L^r_\sigma(D),W^{2,r}_{0,\sigma}(D)]_s
\subset [L^r_\sigma(D),\SD_r(A_D)]_s=\SD_r(A_D^s)
\end{equation}
for any $s\in(0,1)$ and $r$ fulfilling $|1/r-1/2|<1/6+\varepsilon_1$.
Here and in the following, $\SD_r(A_D)$ and $\SD_r(A_D^\alpha)$ stand for 
the domain of $A_D$ and $A^\alpha_D$ defined on $L^r_\sigma(D)$, respectively.
To show the converse inclusion,
it suffices to prove that
$A_D^{-\alpha}:L^q_{\sigma}(D)\rightarrow H_0^{2\alpha,q}(D)$
is bounded since $H_{0,\sigma}^{2\alpha,q}(D)
=L^q_\sigma(D)\cap H_0^{2\alpha,q}(D)$ (cf. \cite[Cor. 2.11]{MM08})
and since
$A_D^{-\alpha}:L^q_{\sigma}(D)\rightarrow L^q_\sigma(D)$ is bounded.
To this end, we deduce
\begin{align}
\label{theta}
\|A_D^{-\alpha}P_Dg\|_{q',D}\leq C\|g\|_{H^{-2\alpha,q'}(D)}
\end{align}
for all $g\in C^{\infty}_0(D)$ provided that \eqref{p} and
\begin{equation}
\label{theta2}
\frac{1}{2}\left(1+\frac{1}{q}-\delta_1\right)<\alpha<\min\left\{
\frac{1}{2}+\frac{1}{2q},\frac{3}{2q}+\frac{\delta_1}{2}\right\},
\end{equation}
where $1/q+1/q'=1$. In fact,
by \eqref{inclusion} and $\SD_{q'}(A_D)\subset
H_{0,\sigma}^{2(1-\alpha),q'}(D)$ (cf. Remark \ref{rmkembe}), we have
\begin{equation}
\label{atheta}
\|A_D^{-\alpha}P_Dg\|_{q',D}=
\|A_D^{1-\alpha}A_D^{-1}P_Dg\|_{q',D}\leq
\|A_D^{-1}P_Dg\|_{H_{0,\sigma}^{2-2\alpha,q'}(D)}.
\end{equation}
Set $u=A_D^{-1}P_Dg$ and let $p$ denote the associated pressure. Then
$(u,p)$ solves the Stokes system
\begin{align*}
\left\{
\begin{aligned}
-\Delta u+\nabla p & = P_Dg &\quad & \text{in~$D$}, \\
\div u & = 0 &\quad & \text{in~$D$},\\
u & = 0& \quad & \text{on $\partial D$}.
\end{aligned}\right.
\end{align*}
Since we may write $P_Dg=g+\nabla h$ with some $h\in W^{1,q'}(D)$,
it follows that $(u,p-h)$ solves the same system but
with the right-hand side replaced by $g$. Set $s_0=1+1/q-2\alpha$.
Then, in view of \eqref{theta2}, we apply Proposition \ref{propstokeswell}
with $s=s_0$, $q=q'$, and $f=g$ to obtain
\begin{equation}
\|A_D^{-1}P_Dg\|_{H_{0,\sigma}^{2-2\alpha,q'}(D)}
=\|u\|_{H_{0,\sigma}^{2-2\alpha,q'}(D)}\leq
C\|g\|_{H^{-2\alpha,q'}(D)},
\end{equation}
which together with \eqref{atheta} yields \eqref{theta} provided that
\eqref{theta2} is valid.
Since $H_0^{2\alpha,q}(D)=H^{-2\alpha,q'}(D)^*$ and since
$C_0^\infty(D)$ is dense in $H^{-2\alpha,q'}(D)$ due to Triebel
\cite[Thm. 3.5 (i)]{T02}, applying \eqref{theta} to
\begin{equation}
|(A_D^{-\alpha}f,g)|=|(f,A_D^{-\alpha}P_Dg)|\leq
C\|f\|_{q,D}\|A_D^{-\alpha}P_Dg\|_{q',D}
\end{equation}
for $f\in C_{0,\sigma}^\infty(D)$ and
$g\in C_0^\infty(D)$ asserts the boundedness of
$A_D^{-\alpha}:L^q_{\sigma}(D)\rightarrow H_0^{2\alpha,q}(D)$ provided
\eqref{theta2}. To get rid of the lower bound in \eqref{theta2},
given $\alpha$ satisfying \eqref{alpha2}, we take $\theta$ so that
$\max\{(1+1/q-\delta_1)/2,\alpha\}<\theta<\min\{
1/2+1/(2q),3/(2q)+\delta_1/2\}$ is fulfilled. Then the aforementioned result for $\theta$
and the result due to \cite[Thm. 2.12]{MM08} yield
\begin{equation}
\SD_q(A_D^\alpha)=[L^q_\sigma(D),\SD_q(A_D^\theta)]_\tau=
[L^q_\sigma(D),H_{0,\sigma}^{2\theta,q}(D)]_\tau=H_{0,\sigma}^{2\alpha,q}(D)
\end{equation}
with $\alpha=\theta \tau$.
The proof is complete.
\end{proof}

\begin{proof}[Proof of Theorem \ref{proplqlrsmooth}]
The first assertion follows from the proof of Lemma \ref{lemfractional}. \par	
To show the second assertion, in view of $e^{-t\CL_{D,\eta,\omega}}f\in W^{1,q}_0(D)$, we use
the Gagliardo--Nirenberg inequality to see
\begin{align*}
\|e^{-t\CL_{D,\eta,\omega}}f\|_{r,D}\leq C
\|\nabla e^{-t\CL_{D,\eta,\omega}}f\|^{3(\frac{1}{q}-\frac{1}{r})}
_{q,D}
\|e^{-t\CL_{D,\eta,\omega}}f\|^{1-3(\frac{1}{q}-\frac{1}{r})}_{q,D}
\leq Ct^{-\frac{3}{2}(\frac{1}{q}-\frac{1}{r})}\|f\|_{q,D}
\end{align*}
for $f\in L^q_{\sigma}(D)$ and $0<t\leq T$
provided that \eqref{p} and $1/q-1/r<1/3$.
However, we may eliminate the condition $1/q-1/r<1/3$
by using semigroup property, which gives \eqref{lqlrsmooth} with $j=0$.
Concerning the estimate \eqref{lqlrsmooth} with $j=1$,
we infer from Lemma \ref{lemfractional} and the moment inequality that
\begin{align*}
\|\nabla e^{-t\CL_{D,\eta,\omega}}f\|_{r,D}
&\leq C\|e^{-t\CL_{D,\eta,\omega}}
f\|^{\frac{\beta-\frac{1}{2}}{\beta-\alpha}}_{H^{2\alpha,r}(D)}
\|e^{-t\CL_{D,\eta,\omega}}
f\|^{\frac{\frac{1}{2}-\alpha}{\beta-\alpha}}_{H^{2\beta,r}(D)}\\
&\leq C\|e^{-t\CL_{D,\eta,\omega}}
f\|^{\frac{\beta-\frac{1}{2}}{\beta-\alpha}}_{H^{2\gamma,q}(D)}
\|e^{-t\CL_{D,\eta,\omega}}
f\|^{\frac{\frac{1}{2}-\alpha}{\beta-\alpha}}_{H^{2\delta,q}(D)}
\\
&\leq C\|A_D^{\gamma}e^{-t\CL_{D,\eta,\omega}}
f\|^{\frac{\beta-\frac{1}{2}}{\beta-\alpha}}_{q,D}
\|A_D^{\delta}e^{-t\CL_{D,\eta,\omega}}
f\|^{\frac{\frac{1}{2}-\alpha}{\beta-\alpha}}_{q,D} \\
& \leq Ct^{-\frac{3}{2}(\frac{1}{q}-\frac{1}{r})-\frac{1}{2}}
\|f\|_{q,D}
\end{align*}
for $f\in L^q_\sigma(D)$ and $0<t\leq T$, where
we have taken $\alpha$, $\beta$, $\gamma$, and $\delta$ fulfilling
\begin{equation}
\alpha<1/2<\beta,\quad\alpha<\beta, \gamma <\delta,\quad
2\gamma-\frac3q=2\alpha-\frac3r,\quad2\delta-\frac3q=2\beta-\frac3r.
\end{equation}
Thus, the proof of the second assertion is complete. \par
To prove the third assertion, we notice that
the condition $1/q-1/r\leq (1-2\alpha)/3$
is equivalent to $1-3/r'\geq 2\alpha-3/q'$.
Let $\BB_D$ be the Bogovski\u{\i} operator defined on $D$.
Then, we observe
\begin{align*}
|(\partial_te^{-t\CL_{D,\eta,\omega}}f,\BB_D g)|
& \leq \|\CL_{D,\eta,\omega}^{1-\alpha}e^{-t\CL_{D,\eta,\omega}}f\|_{q,D}
\|\CL_{D,- \eta, - \omega}^\alpha\BB_D g\|_{q',D} \\
&\leq Ct^{-1+\alpha}
\|f\|_{q,D}\|\BB_D g\|_{H^{2\alpha,q'}(D)}\\
&\leq
Ct^{-1+\alpha}\|f\|_{q,D}\|\BB_D g\|_{W^{1,r'}(D)}\\
&\leq Ct^{-1+\alpha}\|f\|_{r,D}\|g\|_{r',D}
\end{align*}
for all
$f\in L^q_\sigma(D)\cap L^r(D)$, $g\in L^{r'}_0(D)$,
and $0<t\leq T$.
Together with the equation and \eqref{lqlrsmooth}, we obtain
\begin{align}
|(p,g)| &\leq C\|e^{-t\CL_{D,\eta,\omega}}f\|_{W^{1,r}(D)}
\|\BB_D g\|_{W^{1,r'}(D)}
+|(\partial_te^{-t\CL_{D,\eta,\omega}}f,\BB_D g)| \\
& \leq Ct^{-1+\alpha}\|f\|_{r,D}\|g\|_{r',D}
\end{align}
for $g\in L^{r'}_0(D)$ and $0 < t\leq T$ yielding \eqref{pest}.
Finally, the estimate \eqref{partialtest} follows from
the equation.
\end{proof}

For the latter use, we also record the result on the decay of
the pressure $p_\lambda$ with respect to the resolvent parameter
$\lambda \in \{\lambda \in \Sigma_\theta \mid 
\lvert \lambda \rvert \ge R\}$ with $\theta \in (0, \pi)$,
where $(u_\lambda, p_\lambda)$ is a solution to
\begin{equation}
\label{eq-resolvent}
\left\{
\begin{aligned}
\lambda u_\lambda-\Delta u_\lambda-(\eta + \omega\times x)\cdot\nabla u_\lambda
+\omega\times u_\lambda +\nabla p_\lambda & = f, &\quad & \text{in $D$}, \\
\div u_\lambda & = 0, &\quad & \text{in $D$},\\
u_\lambda & = 0, &\quad & \text{in $\pd D$}
\end{aligned}\right.
\end{equation}
for $f \in L^q_\sigma (D)$ and
$\lambda \in \{\lambda \in \Sigma_\theta \mid \lvert \lambda \rvert \ge R\}$, $\theta \in (0, \pi)$. 
Notice that
the case $\eta=\omega=0$ 
was observed by Tolksdorf and the second author \cite[Prop. 4.3]{TW20}.  
By taking into account \eqref{lresol} and 
Lemma \ref{lemfractional}, 
the proof of the following proposition is 
essentially the same as the one by 
\cite[Prop. 4.3]{TW20}, and thus 
we shall only give the sketch of the proof. 
We also note that \eqref{plambda} below 
with $\alpha=(1-1/q)/2$ was 
found by Hishida and Shibata \cite{HS09} in the case of 
$C^{1,1}$-boundary by using the trace estimate \cite[(3.15)]{HS09} (cf. Galdi \cite[Thm.~II.4.1]{G11}). 
However, we may not employ their idea since 
we may not expect the $W^{2,q}$-regularity for the velocity 
field. 
\begin{prop}\label{proppdecay}
Let $\varepsilon_1$ and $q$ be subject to Convention \ref{conv-1}.
Given $c_0>0$, assume $\lvert \eta \rvert + \lvert \omega \rvert \leq c_0$.
For $f\in L^q_\sigma(D)$, let $(u_\lambda,p_\lambda)$
be the unique solution to Problem \eqref{eq-resolvent}
such that $u_\lambda\in \SD(\CL_{D,\eta,\omega})$
and $p_\lambda\in L^q_0(D)$.
Define the operator
$P_\lambda:L^q_\sigma(D)\rightarrow L^q_0(D)$ by
$P_\lambda f:= p_\lambda$.
For every $\alpha$ satisfying
\begin{align}\label{alpha}
0\leq 2\alpha<1-\frac{1}{q}\quad{\rm if}~q\geq \frac{2}{1+\delta_2}
\quad {\rm and}\quad 0\leq 2\alpha<2-\frac{3}{q}+\delta_2
\quad{\rm if}~q<\frac{2}{1+\delta_2},
\end{align}
there exist constants $C$ and $R$ such that
\begin{align}
\label{plambda}
\|P_\lambda\|_{\SL(L^q_{\sigma}(D),L^q_0(D))}\leq
C|\lambda|^{-\alpha}
\end{align}
with every $\lambda\in \Sigma_\theta$, 
$\theta \in (0, \pi)$, fulfilling $|\lambda|\geq R$,
where $\delta_2\in(0,1]$ is
the same number as in Proposition \ref{propdelta2}.
\end{prop}
\begin{proof}
Since $p_\lambda\in L^q_0(D)$, we may have
\begin{align*}
\|p_\lambda\|_{q,D}=\sup_{g\in L^{q'}_0(D),\|g\|_{q',D}\leq 1}
\left|\int_Dp_\lambda \overline{g}\d x\right|.
\end{align*}
By the definition of the Stokes operator, 
we infer from Proposition 
\ref{prop-Bogovskii} and Proposition \ref{prop-projection} that
\begin{align*}
\left|\int_Dp_\lambda \overline{g}\d x\right|&=
\left|\int_Dp_\lambda \overline{\nabla\cdot\BB[g]}\d x\right|\\
& \leq
C(\|\nabla(\lambda+\CL_{D,\eta,\omega})^{-1}f\|_{q,D}
\|\nabla \BB[g]\|_{q',D}
+\|A_{D}^{1-\alpha}
(\lambda+\CL_{D,\eta,\omega})^{-1}f\|_{q,D}
\|A_D^\alpha P_D\BB[g]\|_{q',D}\\
&\leq C|\lambda|^{-\alpha}\|g\|_{q',D}
\end{align*}
for all $g\in L^{q'}_0(D)$, which implies the assertion.  
\end{proof}

\section{Generation of a $C_0$-semigroup}
\label{sec-4}

In the following, let $\Omega\subset \BR^3$ be an exterior Lipschitz domain.
The aim of this section is to prove that $- \CL_{\Omega,\eta,\omega}$ generates a
$C_0$-semigroup on $L^q_\sigma (\Omega)$ and to show its mapping properties.
To this end, we need to take $\varepsilon_2> 0$ so small that the Helmholtz decomposition
of $L^q (\Omega)^3$ exists for all $q$ subject to the condition \eqref{p'} below,
if necessary. Namely, in the following, we assume that $q$ and $\varepsilon_2$
satisfy the following convention.
\begin{conv}
\label{conv-2}
Let $\Omega\subset \BR^3$ be an exterior Lipschitz domain and
$\varepsilon_1$ be subject to Convention \ref{conv-1}.
Let $\varepsilon_2\in(0, \varepsilon_1]$ be a number such that
for all $q\in(1,\infty)$ satisfying
\begin{equation}
\label{p'}
\bigg\lvert \frac1q - \frac12 \bigg\rvert < \frac16 + \varepsilon_2
\end{equation}
the Helmholtz decomposition of $L^q(\Omega)^3$ exists.
\end{conv}
Notice that the existence of the Helmholtz decomposition of $L^q(\Omega)^3$ for all $q \in (1, \infty)$ satisfying \eqref{conv-2}
was proved in \cite[Prop. 2.3]{TW20}. This section especially aims to prove the following theorem.

\begin{thm}\label{thm1}
Let $\varepsilon_2$ and $q$ be subject to Convention \ref{conv-2}
and $\eta, \omega \in \BR^3$. Then the following assertions are valid.
\begin{enumerate}
\item The operator $-\CL_{\Omega,\eta,\omega}$ defined by \eqref{def-operator} generates a $C_0$-semigroup
$(e^{-t\CL_{\Omega,\eta,\omega}})_{t \ge 0}$ on $L^q_\sigma(\Omega)$.
\item Given $c_0>0,$ assume $\lvert \eta \rvert + \lvert \omega \rvert \leq c_0.$
Let $j=0,1$ and let $q \le r\leq\infty~
({\rm resp.}~q \le r<\infty)$ if $j=0~({\rm resp.}~j=1)$.
For $T>0$, there exists a constant
$C=C(\Omega,j,T,c_0,q,r)$ independent of $\eta$ and $\omega$ 
such that
\begin{align}\label{lqlrsmooth2}
\|\nabla^j e^{-t\CL_{\Omega,\eta,\omega}}f\|_{r,\Omega}
\leq Ct^{-\frac{3}{2}(\frac{1}{q}-\frac{1}{r})-\frac{j}{2}}
\|f\|_{q,\Omega}
\end{align}
for $0<t\leq T$ and $f\in L^q_\sigma(\Omega)$.
\end{enumerate}
\end{thm}

In the following, we take $R_0 > 0$ sufficiently large such that
$\BR^3\setminus\Omega \subset B_{R_0}(0)$. To prove Theorem \ref{thm1},
we shall start from the following proposition.

\begin{prop}\label{propsurj}
Let $\varepsilon_2$ and $q$ be subject to Convention \ref{conv-2}
and $\eta, \omega \in \BR^3$.
There exists a constant $\kappa\geq 1$ such that
for $\lambda\in\BR$ satisfying $\lambda\geq \kappa$ and
for $f\in L^q_\sigma(\Omega)$,
there exists $u\in \SD(\CL_{\Omega,\eta,\omega})$ that satisfies
\begin{align*}
\lambda u+\CL_{\Omega,\eta,\omega}u=f.
\end{align*}
\end{prop}

\begin{proof}
Let $R>R_0$ and set
\begin{align}
\label{def-D}
\begin{split}
D &:=\Omega\cap B_{R+5}(0),\\
K_1&:=\{x\in\Omega\mid R<|x|<R+3\},\\
K_2&:=\{x\in\Omega\mid R+2<|x|<R+5\}.
\end{split}
\end{align}
For $n \in \{1,2\}$, let $\BB_{K_n}$ be the Bogovski\u{\i} operator on $K_n$.
We take functions $\varphi,\chi\in C^\infty (\BR^3)$ such that
$0\leq \varphi(x), \chi(x)\leq 1$ for $x\in\BR^3$ and
\begin{align}
\varphi (x) = \begin{cases}
0 & \text{for $\lvert x \rvert \le R + 1$}, \\
1 & \text{for $\lvert x \rvert \ge R + 2$},
\end{cases}
\qquad
\chi (x) = \begin{cases}
1 & \text{for $\lvert x \rvert \le R + 3$}, \\
0 & \text{for $\lvert x \rvert \ge R + 4$}.
\end{cases}
\end{align}
For $f\in L^q_\sigma(\Omega)$,
we define $f^R\in L^q_\sigma(\BR^3)$ as the zero extension of $f$ to $\BR^3$.
Given $f\in L^q_\sigma(\Omega)$,
we also define $f^D\in L^q_\sigma(D)$
by $f^D:=\chi f-\BB_{K_2} [(\nabla\chi)\cdot f]$.
For $\lambda>0$,
let $u_\lambda^R$ be a solution to the problem
\begin{equation}
\left\{\begin{aligned}
\lambda u^R_\lambda - \Delta u^R_\lambda - (\eta + \omega\times x)\cdot\nabla u^R_\lambda
+ \omega\times u^R_\lambda & = f^R & \quad & \text{in $\BR^3$}, \\
\div u^R_\lambda & =0 & \quad & \text{in $\BR^3$}
\end{aligned}\right.
\end{equation}
and let $(u_\lambda^D,\pi_\lambda^D)$ be a solution to the problem
\begin{equation}
\left\{\begin{aligned}
\lambda u^D_\lambda - \Delta u^D_\lambda - (\eta + \omega\times x)\cdot\nabla u^D_\lambda
+ \omega\times u^D_\lambda + \nabla \pi^D_\lambda & = f^D & \quad & \text{in $D$}, \\
\div u^D_\lambda & =0 & \quad & \text{in $D$}, \\
u_\lambda^D & = 0 & \quad & \text{on $\pd D$}.
\end{aligned}\right.
\end{equation}
We define
\begin{align*}
U_\lambda f:=\varphi u_\lambda^R+(1-\varphi)u_\lambda^D
+\BB_{K_1}[\nabla\varphi\cdot(u_\lambda^R-u_\lambda^D)],
\qquad P_\lambda f:=(1-\varphi)\pi_\lambda^D,
\end{align*}
then $U_\lambda f\in\{u\in W^{1,q}_{0,\sigma}(\Omega)
\mid(\eta + \omega\times x)\cdot\nabla u\in L^q(\Omega)\}$ and
there holds
\begin{align}\label{ulambda}
\|U_\lambda f\|_{W^{1,q}(\Omega)}+
\|-(\eta + \omega\times x)\cdot\nabla U_\lambda f
+\omega\times U_\lambda f\|_{q,\Omega}
\leq C\|f\|_{q,\Omega}.
\end{align}
In addition, the pair $(U_\lambda f,P_\lambda f)$ satisfies
\begin{equation}\label{paramtrixeq}
\left\{\begin{aligned}
\lambda U_\lambda f - \Delta U_\lambda f - (\eta + \omega\times x)\cdot\nabla U_\lambda f
+ \omega\times U_\lambda f + \nabla P_\lambda f & = f + \Phi_\lambda f & \quad & \text{in $\Omega$}, \\
\div U_\lambda f & =0 & \quad & \text{in $\Omega$}, \\
U_\lambda f & = 0 & \quad & \text{on $\pd \Omega$},
\end{aligned}\right.
\end{equation}
where we have set
\begin{align*}
\Phi_\lambda f & :=-2\nabla\varphi\cdot\nabla(u_\lambda^R-u_\lambda^D)
-(\Delta \varphi
+(\eta + \omega\times x)\cdot\nabla \varphi)(u^R_\lambda-u_\lambda^D)
-\pi_\lambda^D\nabla \varphi \\
& \quad +\lambda\BB_{K_1}[(\nabla\varphi)\cdot(u_\lambda^R-u_\lambda^D)]
-\Delta\BB_{K_1}[(\nabla\varphi)\cdot(u_\lambda^R-u_\lambda^D)] \\
& \quad -(\eta + \omega\times x)\cdot\nabla
\BB_{K_1}[(\nabla\varphi)\cdot(u_\lambda^R-u_\lambda^D)]
+\omega\times\BB_{K_1}[(\nabla\varphi)\cdot(u_\lambda^R-u_\lambda^D)].
\end{align*}
Notice that we may find
$\Phi_\lambda f=P_{\Omega}\Phi_\lambda f+\nabla \Pi_\lambda f$
on account of the Helmholtz decomposition. 
The term
$\lambda\BB_{K_1}[(\nabla\varphi)\cdot(u_\lambda^R-u_\lambda^D)]$
is written as
\begin{align}
& \lambda\BB_{K_1}[(\nabla\varphi)\cdot(u_\lambda^R-u_\lambda^D)] \\
& =
\BB_{K_1}[(\nabla\varphi)\cdot
\{\Delta(u_\lambda^R-u_\lambda^D)
+(\eta + \omega\times x)\cdot\nabla (u_\lambda^R-u_\lambda^D)
-\omega\times(u_\lambda^R-u_\lambda^D)\}]
+\BB_{K_1}[\nabla\varphi\cdot\nabla \pi_\lambda^D].
\end{align}
For $\psi\in W^{1,q'}(K_1)$
and $\lambda\geq 1$, it holds that
\begin{align}
\big|&\big((\nabla\varphi)\cdot
\{\Delta(u_\lambda^R-u_\lambda^D)
+(\eta + \omega\times x)\cdot\nabla (u_\lambda^R-u_\lambda^D)
-\omega\times(u_\lambda^R-u_\lambda^D)\},\psi\big)_{K_1}\big|\\
&\leq
C\|u_\lambda^R-u_\lambda^D
\|_{W^{1,q}(K_1)}\|\psi\|_{W^{1,q'}(K_1)} \\
& \leq C\lambda^{-\frac{1}{2}}\|f\|_{q,\Omega}
\|\psi\|_{W^{1,q'}(K_1)},
\end{align}
which implies
\begin{align*}
\|(\nabla\varphi)\cdot
\{\Delta(u_\lambda^R-u_\lambda^D)
+(\eta + \omega\times x)\cdot\nabla (u_\lambda^R-u_\lambda^D)
-\omega\times(u_\lambda^R-u_\lambda^D)\}\|_{W^{1,q'}(K_1)^*}\leq
C\lambda^{-\frac{1}{2}}\|f\|_{q,\Omega}
\end{align*}
for $f\in L^q_\sigma(\Omega)$ and $\lambda\geq 1$.
Let $\alpha$ satisfy \eqref{alpha}.
By the same investigation as above together
with Proposition \ref{proppdecay}, we also have
$\|\nabla\varphi\cdot\nabla \pi_\lambda^D\|_{W^{1,q'}(K_1)^*}
\leq C\lambda^{-\alpha}\|f\|_{q,\Omega}$
for $f\in L^q_\sigma(\Omega)$ and $\lambda\geq 1$, thereby, 
Proposition \ref{prop-Bogovskii2} asserts that
\begin{align}
\|\lambda\BB_{K_1}[(\nabla\varphi)\cdot(u_\lambda^R-u_\lambda^D)]
\|_{q,\Omega}\leq C\lambda^{-\min\{\frac{1}{2},\alpha\}}
\|f\|_{q,\Omega}
\end{align}
for $f\in L^q_\sigma(\Omega),\,\lambda\geq 1$.
The other terms may be treated more easily and we deduce that
$\|P_\Omega \Phi_\lambda\|_{ \SL(L^q_\sigma(\Omega))}
\leq C\lambda^{-\min\{1/2,\alpha\}}$
for $\lambda\geq 1$.
Hence, there exists $\kappa\geq 1$ such that
$(I+P_\Omega \Phi_\lambda)^{-1}\in \SL(L^q_\sigma(\Omega))$
for $\lambda\geq \kappa$.
We may conclude from \eqref{paramtrixeq} that
\begin{align}\label{surj}
u:=U_\lambda(I+P_\Omega \Phi_\lambda)^{-1}f,\qquad
p:=P_\lambda (I+P_\Omega\Phi_\lambda)^{-1}f-\Pi_{\lambda}f
\end{align}
satisfies
\begin{equation}
\label{resolexterior}
\left\{
\begin{aligned}
\lambda u -\Delta u -(\eta + \omega\times x)\cdot\nabla u
+\omega\times u +\nabla p & = f, &\quad & \text{in $\Omega$}, \\
\div u & = 0, &\quad & \text{in $\Omega$},\\
u & = 0, &\quad & \text{on $\pd \Omega$}.
\end{aligned}\right.
\end{equation}
If $f\in L^2_\sigma(\Omega)$, then by taking into account
the definition of the Stokes
operator, we find $\lambda u+\CL_{\Omega,\eta,\omega}u=f$.
If $f\in L^q_\sigma(\Omega)$,
then we take $f_k\in C_{0,\sigma}^\infty(\Omega)
\subset L^2_{\sigma}(\Omega)\cap L^q_{\sigma}(\Omega)$ so that
$f_k \rightarrow f$ as $k\rightarrow \infty$
in $L^q_\sigma(\Omega)$.
Since $u_k:=U_\lambda(I+P_\Omega \Phi_\lambda)^{-1}f_k$ and
$p_k:=P_\lambda (I+P_\Omega\Phi_\lambda)^{-1}f_k-\Pi_{\lambda}f_k$
solve \eqref{resolexterior} with $f=f_k$, we find
$u_k\in \SD(A_{2,\Omega})$ with
\begin{equation}
A_{2,\Omega}u_k=-\lambda u_k
+(\eta + \omega\times x)\cdot\nabla u_k
-\omega\times u+f_k\in L^q_\sigma(\Omega).
\end{equation}
It also follows from \eqref{ulambda} that
$u_k\in L^q_\sigma(\Omega)$ and that
\begin{equation}
u_k\rightarrow u \quad \text{and} \quad A_{2,\Omega}u_k\rightarrow
-\lambda u+(\eta + \omega\times x)\cdot\nabla u-\omega\times u+f
\quad \text{in $L^q_\sigma(\Omega)$ as $k \to \infty$}.
\end{equation}
Hence, we have $\lambda u+\CL_{\Omega,\eta,\omega}u=f$ 
from the definition of
the Stokes operator. 
\end{proof}
The construction of a semigroup with some smoothing effects 
is based on the the following lemma. For the proof, we refer to
\cite[Lem. 4.6]{GHH06} and \cite[Lem. 3.3]{HR11}.
\begin{lem}\label{lemserires}
Let $X_1$ and $X_2$ be two Banach spaces and fix 
$T\in(0,\infty]$. 
Suppose that operator families
$\{A_0(t)\mid 0<t<T\}\subset\SL(X_1,X_2)$ and
$\{Q(t)\mid 0<t<T\}\subset\SL(X_1)$ satisfy
\begin{align*}
\|A_0(t)\|_{\SL(X_1,X_2)}\leq Ct^{-\alpha_0}e^{\sigma_0 t},\qquad
\|Q(t)\|_{\SL(X_1)}\leq Ct^{-\beta_0}e^{\sigma_0 t}
\end{align*}
for every $0<t<T$ with some constants
$\alpha_0,\beta_0\in [0,1)$ and $\sigma_0,\,C>0$. 
For $f\in X_1$ and $0<t<T$, define
\begin{align*}
A_{j+1}(t)f:=\int_0^tA_{j}(t-\tau)Q(\tau)f\d \tau,
\qquad j \in \BN \cup \{0\}.
\end{align*}
Then the operator
\begin{align*}
A(t)f:=\sum_{j=0}^\infty A_j(t)f
\end{align*}
in $X_2$ converges absolutely and uniformly 
in every bounded intervals $I\subset (0,T)$.
Moreover, there exist constants $\sigma>\sigma_0$ and $C$ such that
\begin{align*}
\|A(t)f\|_{X_2}\leq \sum_{j=0}^\infty\|A_j(t)f\|_{X_2}\leq
Ct^{-\alpha_0}e^{\sigma t}\|f\|_{X_1}
\end{align*}
for $t<T$ and $f\in X_1$. 
\end{lem}

We are now in a position to prove Theorem \ref{thm1}.

\begin{proof}[Proof of Theorem \ref{thm1}]
In this proof, the operator $\CL_{\Omega,\eta,\omega}$ on
$L^q_\sigma(\Omega)$ is denoted by $\CL_q$ 
to simplify the notation. 
By Theorem \ref{proplqlrsmooth}, we know that 
$(e^{- t \CL_{D,\eta,\omega}})_{t \ge 0}$ is a bounded
analytic $C_0$-semigroup generated by $-\CL_{D,\eta,\omega}$.
Given $f\in L^q_\sigma(\Omega)$, we set
\begin{align}
S_{\eta,\omega}(t)f & :=\varphi T_{\BR^3,\eta,\omega}(t)f^R+(1-\varphi) e^{- t \CL_{D,\eta,\omega}}f^D
+\BB_{K_1}[(\nabla\varphi)\cdot(T_{\BR^3,\eta,\omega}(t)f^R-e^{- t \CL_{D,\eta,\omega}}f^D)],
\notag\\
T_{\Omega,\eta,\omega,0}(t)f & :=S_{\eta,\omega}(t)f,
\label{const1}\\
T_{\Omega,\eta,\omega,n}(t)f & :=-\int_0^t T_{\Omega,\eta,\omega,n-1}(t-\tau)
H(\tau)f\d \tau,\quad n\in\BN,\label{const2}
\end{align}
where $(T_{\BR^3,\eta,\omega}(t))_{t \ge 0}$ is the $C_0$-semigroup 
given by \eqref{r3semi} and $H(t)$ is the inverse 
Laplace transform of $P_\Omega \Phi_\lambda$, see 
\cite[Lem. 4.4]{GHH06}. 
Here, $\varphi$, $f^R$, and $f^D$ are the same functions
as in the proof of Proposition \ref{propsurj}.
Then the boundedness of $T_{\BR^3,\eta,\omega}(t)$ and $e^{- t \CL_{D,\eta,\omega}}$ together with
\cite[Lem. 4.4]{GHH06} yields
\begin{align}\label{shest}
\|S_{\eta,\omega}(t)\|_{\SL(L^q_\sigma(\Omega))}\leq C,
\qquad \|H(t)\|_{\SL(L^q_\sigma(\Omega))}\leq
Ct^{\alpha-1}e^{\sigma_0 t}
\end{align}
with some constants $\sigma_0,C>0$,
where $\alpha>0$ is a fixed constant satisfying \eqref{alpha}.
We then infer from Lemma \ref{lemserires} with $A_0=S_{\eta,\omega}$ and $Q=H$ that there exists a constant $\sigma > \sigma_0$ such that there hold
\begin{align}\label{tseries}
T_{\Omega,\eta,\omega} f:=\sum_{n=0}^\infty T_{\Omega,\eta,\omega,n}(t)f,\qquad
\sum_{n=0}^\infty\|T_{\Omega,\eta,\omega,n}(t)f\|_{q,\Omega}
\leq Ce^{\sigma t}
\|f\|_{q,\Omega}
\end{align}
for $t>0$ and $f\in L^q_\sigma(\Omega)$.
Since we have
\begin{align*}
\|S_{\eta,\omega} (t)\|_{\SL(L^q_\sigma(\Omega),L^{r}(\Omega))}\leq
Ct^{-\frac{3}{2}(\frac{1}{q}-\frac{1}{r})},\qquad
\|S_{\eta,\omega} (t)\|_{\SL(L^q_\sigma(\Omega),W^{1,r}(\Omega))}\leq
Ct^{-\frac{3}{2}(\frac{1}{q}-\frac{1}{r})-\frac{1}{2}}
\end{align*}
for $0 < t<T$, we apply Lemma \ref{lemserires}
to obtain 
\begin{align*}
\|\nabla^j T_{\Omega,\eta,\omega} (t)f\|_{r,\Omega}
\leq Ct^{-\frac{3}{2}(\frac{1}{q}-\frac{1}{r})-\frac{j}{2}}
\|f\|_{q,\Omega}
\end{align*}
for $0 < t<T$ and $f\in L^q_\sigma(\Omega)$ 
whenever $1/q-1/r<(2-j)/3$ is satisfied. 
However, this restriction may be eliminated 
by virtue of the semigroup property. 
\par It remains to prove that 
$(T_{\Omega,\eta,\omega}(t))_{t\ge 0}$ is indeed a 
$C_0$-semigroup $(e^{- t \CL_q})_{t\ge 0}$ generated by $-\CL_{q}$.
To this end, we notice that $-\CL_{2}$ generates a contraction
$C_0$-semigroup $e^{-t\CL_{2}}$ on $L^2_\sigma(\Omega)$
since $-\CL_{2}$ is $m$-dissipative.
We thus observe
\begin{align}\label{lap2}
\int_0^\infty e^{-\lambda t}e^{-t\CL_{2}}f\d t
=(\lambda+\CL_{2})^{-1}f
\end{align}
for $f\in L^2_\sigma(\Omega)$ and $\lambda\in\BC$
satisfying $\Re\lambda>0$. In addition, since there hold
\begin{align}
\|T_{\Omega,\eta,\omega}(t)f\|_{q,\Omega}\leq
\sum_{n=0}^\infty\|T_{\Omega,\eta,\omega,n}(t)f\|_{q,\Omega}
\leq Ce^{\sigma t}\|f\|_{q,\Omega}
\end{align}
and
\begin{align}
\int_0^\infty e^{-\lambda t}T_{\Omega,\eta,\omega,n}(t)f\d t=U_\lambda
(-P_\Omega \Phi_\lambda)^nf,
\end{align}
we have
\begin{align}\label{lapt}
\int_0^\infty e^{-\lambda t}T_{\Omega,\eta,\omega}(t)f\d t=
\sum_{n=0}^\infty\int_0^\infty e^{-\lambda t}T_{\Omega,\eta,\omega,n}(t)f\d t
=U_\lambda(I+P_\Omega\Phi_\lambda)^{-1}f
\end{align}
for $t>0$, $f\in L^q_\sigma(\Omega)$,
and $\lambda>\max\{\sigma,\kappa\}$,
where $\kappa$ is the same number as in Proposition \ref{propsurj}.
Since there holds
\begin{align}\label{resoll}
(\lambda+\CL_{q})^{-1}f=(\lambda+\CL_{2})^{-1}f
=U_\lambda(I+P_\Omega \Phi_\lambda)^{-1}f
\end{align}
for all $f\in L^q_\sigma(\Omega)\cap L^2_\sigma(\Omega)$
and $\lambda>\max\{\sigma,\kappa\}$, see
Proposition \ref{propsurj} and \eqref{surj}, it follows from
\eqref{lap2} and \eqref{lapt} that
\begin{align*}
\int_0^\infty e^{-\lambda t}e^{-t\CL_{2}}f\d t
=\int_0^\infty e^{-\lambda t}T_{\Omega,\eta,\omega}(t)f\d t
\end{align*}
for $f\in L^q_\sigma(\Omega)\cap L^2_\sigma(\Omega)$.
Hence, we arrive at
\begin{align}\label{coin}
T_{\Omega,\eta,\omega}(t)f=e^{-t\CL_{2}}f
\end{align}
for $f\in L^q_\sigma(\Omega)\cap L^2_\sigma(\Omega)$. \par
From \eqref{coin}, we shall prove that $(T_{\Omega,\eta,\omega}(t))_{t \ge 0}$ is a
semigroup $(e^{- t \CL_{\Omega,\eta,\omega}})_{t\ge0}$ on $L^q_\sigma(\Omega)$ generated by $-\CL_{q}$.
Let $f\in L^q_\sigma(\Omega)$, then there exists a sequence
$\{f_n\}_{n=1}^\infty\subset L^2_{\sigma}(\Omega)
\cap L^q_{\sigma}(\Omega)$ such that
$f_n\rightarrow f$ as $n\rightarrow\infty$
in $L^q_\sigma(\Omega)$.
Then we find from \eqref{coin} and the semigroup property of
$(e^{-t\CL_{2}})_{t\ge0}$ that
\begin{equation}
\|T_{\Omega,\eta,\omega}(t+s)f_n-T_{\Omega,\eta,\omega}(t)T_{\Omega,\eta,\omega}(s)f_n
\|_{q,\Omega} = 0,
\end{equation}
and thus we have
\begin{align*}
&\|T_{\Omega,\eta,\omega}(t+s)f-T_{\Omega,\eta,\omega}
(t)T_{\Omega,\eta,\omega}(s)f\|_{q,\Omega}\\
& \quad \le
\|T_{\Omega,\eta,\omega}(t+s)(f-f_n)\|_{q,\Omega}
+\|T_{\Omega,\eta,\omega}(t)T_{\Omega,\eta,\omega}(s)(f-f_n)\|_{q,\Omega}\\
& \quad \leq \big(\|T_{\Omega,a}(t+s)\|_{\SL(L^q_\sigma(\Omega))}
+\|T_{\Omega,\eta,\omega}(t)\|_{\SL(L^q_\sigma(\Omega))}
\|T_{\Omega,\eta,\omega}(s)\|_{\SL(L^q_\sigma(\Omega))}\big)
\|f-f_n\|_{q,\Omega}\rightarrow 0
\end{align*}
as $n\rightarrow \infty$, which yields
$T_{\Omega,\eta,\omega}(t+s)=T_{\Omega,\eta,\omega}(t)T_{\Omega,\eta,\omega}(s)$
for $s,t\geq 0$. This shows that
$(T_{\Omega,\eta,\omega}(t))_{t\ge0}$ is a semigroup on $L^q_\sigma(\Omega)$.
Let $-G$ be a generator of $T_{\Omega,\eta,\omega}(t)$.
By \eqref{lapt} and \eqref{resoll}, we have
\begin{align}
(\lambda+G)^{-1}f=\int_0^\infty e^{-\lambda t}T_{\Omega,\eta,\omega}(t)f\d t
=U_\lambda(I+P_\Omega\Phi_\lambda)^{-1}f
=(\lambda+\CL_{q})^{-1}f
\end{align}
for all $f\in L^q_\sigma(\Omega)\cap L^2_\sigma(\Omega)$ and
$\lambda>\max\{\sigma,\kappa\}$,
from which $-G=-\CL_{q}$ follows.
Finally, we shall prove the strong continuity of $(T_{\Omega,\eta,\omega}(t))_{t\ge0}$ on $L^q_\sigma (\Omega)$.
Let $t_0\in[0,\infty)$, $q\geq 2$, and $f\in L^q_\sigma(\Omega)$.
We fix a sequence $\{f_n\}_{n=1}^\infty\subset L^2_{\sigma}(\Omega)
\cap L^q_{\sigma}(\Omega)$ such that
$f_n\rightarrow f$ as $n\rightarrow\infty$
in $L^q_\sigma(\Omega)$. We then obtain
\begin{equation}
\label{f-fn}
\begin{split}
\|T_{\Omega,\eta,\omega}(t_0)f-T_{\Omega,\eta,\omega}(t)f\|_{q,\Omega}&\leq
\big(\|T_{\Omega,\eta,\omega}(t_0)\|_{\SL(L^q_\sigma(\Omega))}
+\|T_{\Omega,\eta,\omega}(t)\|_{\SL(L^q_\sigma(\Omega))}\big)
\|f-f_n\|_{q,\Omega}\\
&\quad+\|T_{\Omega,\eta,\omega}(t_0)f_n-T_{\Omega,\eta,\omega}(t)f_n\|_{q,\Omega}
\\
&\leq C(e^{\sigma t_0}+e^{\sigma t})
\|f-f_n\|_{q,\Omega}
+\|T_{\Omega,\eta,\omega}(t_0)f_n-T_{\Omega,\eta,\omega}(t)f_n\|_{q,\Omega}.
\end{split}
\end{equation}
Since we have
\begin{align}
T_{\Omega,a}(t_0)f_n-T_{\Omega,a}(t)f_n=
\begin{cases}
e^{-t_0\CL_{2}}(I-e^{-(t-t_0)\CL_{2}})f_n & \text{if $t>t_0$}, \\
e^{-t\CL_{2}}(e^{-(t_0-t)\CL_{2}}-I)f_n & \text{if $t<t_0$},
\end{cases}
\end{align}
applying the $L^2$-$L^q$-estimate asserts that
the second term of the right-hand side of \eqref{f-fn} tends to
$0$
as $t\rightarrow t_0$.
We thus conclude that $T_{\Omega,\eta,\omega} (t) = e^{- t \CL_q}$ if $q\geq 2$.
In view of
the duality relation $\CL_{\Omega,\eta,\omega}^*=\CL_{\Omega,-\eta,- \omega}$,
the case $q<2$ directly follows from the case $q>2$.
The proof is complete.
\end{proof}

\section{
$L^q$-$L^r$ estimate with additional restriction for $r$}
\label{sec-5}

In this section, we aim to establish the $L^q$-$L^r$ estimates
for the $C_0$-semigroup $(e^{- t \CL_{\Omega,\eta,\omega}})_{t\ge0}$ provided
that $r$ satisfies the stronger condition in contrast to the
condition posed in Theorem \ref{thmlqlr}.
To be precise, we intend to prove the following lemma.
\begin{lem}\label{proplqlr}
Let $\varepsilon_2$ and $q$ be subject to Convention~\ref{conv-2}
and $\eta, \omega \in \BR^3$. Given $c_0>0,$ assume $\lvert \eta \rvert + \lvert \omega \rvert \leq c_0.$
Let $r \in [q, \infty)$ satisfy
\begin{equation}
\label{cond-r}
\bigg\lvert \frac1r - \frac12 \bigg\rvert < \frac16 + \varepsilon_2.
\end{equation}
Then there exists a constant $C$ independent of 
$\eta$ and $\omega$ such that 
\begin{align}
\|e^{-t\CL_{\Omega,\eta,\omega}}f\|_{r,\Omega}
\leq Ct^{-\frac{3}{2}(\frac{1}{q}-\frac{1}{r})}\|f\|_{q,\Omega}
\end{align}
holds for $t>0$ and $f \in L^q_\sigma (\Omega)$.
\end{lem}
Lemma \ref{proplqlr} was essentially proved by 
Hishida \cite{H18}, but to reveal the additional condition 
\eqref{cond-r}, let us give the proof.   
A key ingredient to prove Lemma \ref{proplqlr} is that the operator $(\eta + \omega\times x)\cdot \nabla$
is skew-symmetric, which gives the energy relation 
\begin{align}\label{en1}
\frac{1}{2}\frac{\mathrm{d}}{\d t}\|e^{-t\CL_{\Omega,\eta,\omega}}f\|_{2,\Omega}^2
=-\|\nabla e^{-t\CL_{\Omega,\eta,\omega}}f\|_{2,\Omega}^2
\end{align}
for $t>0$ and $f\in\SD_2(\CL_{\Omega,\eta,\omega})$. Integrating with respect to $t$ implies
\begin{align}\label{en2}
\frac{1}{2}\|e^{-s\CL_{\Omega,\eta,\omega}}f\|^2_{2,\Omega}=
\int_s^t\|\nabla e^{-\tau\CL_{\Omega,\eta,\omega}}f\|_{2,\Omega}^2\d \tau
+\frac{1}{2}\|e^{-t\CL_{\Omega,\eta,\omega}}f\|^2_{2,\Omega}
\end{align}
for $0\leq s\leq t$ and $f\in\SD_2(\CL_{\Omega,\eta,\omega})$.
Using this identity, we may show the following lemma.

\begin{lem}\label{lem0}
Let $\varepsilon_2$ and $q$ be subject to Convention~\ref{conv-2}
and $\eta, \omega \in \BR^3$. 
Given $c_0>0,$ assume $\lvert \eta \rvert + \lvert \omega \rvert \leq c_0.$
Suppose
$\|e^{-t\CL_{\Omega,\eta,\omega}}f\|_{r_0,\Omega}
\leq C\|f\|_{r_0,\Omega}$
for $t\geq 2$ and $f\in C_{0,\sigma}^\infty(\Omega)$ with
$r_0\in (2,\infty)$ fulfilling
\begin{align}\label{q0condi}
\left|\frac{1}{r_0}-\frac{1}{2}\right|
<\frac{1}{6}+\varepsilon_2.
\end{align}
Then the following assertions are valid.
\begin{enumerate}
\item Let $2\leq q\leq r\leq r_0$.
Then there exists a constant $C=C(c_0,r_0,q,r,\Omega)$ such that
\begin{align}\label{lplqomega}
\|e^{-t\CL_{\Omega,\eta,\omega}}f\|_{r,\Omega}
\leq Ct^{-\frac{3}{2}(\frac{1}{q}-\frac{1}{r})}\|f\|_{q,\Omega}
\end{align}
for $t>0$ and $f\in L^q_\sigma(\Omega).$
\item Let $r_0'\leq q \leq r\leq 2$.
Then there exists a constant $C=C(c_0,r_0,q,r,\Omega)$ such that
\begin{align}\label{lplqomegadual}
\|e^{-t\CL_{\Omega,- \eta, - \omega}}f\|_{r,\Omega}
\leq Ct^{-\frac{3}{2}(\frac{1}{q}-\frac{1}{r})}\|f\|_{q,\Omega}
\end{align}
for $t>0$ and $f\in L^q_\sigma(\Omega).$
\end{enumerate}
\end{lem}

\begin{proof}
The lemma was essentially proved by Hishida \cite[Lem. 4.1]{H18}, but
we shall give the proof for the reader's convenience.
In view of the duality,
we infer from the assumption that
\begin{align*}
\|e^{-t\CL_{\Omega,- \eta, - \omega}}f\|_{r_0',\Omega}
\leq C\|f\|_{r_0',\Omega}
\end{align*}
for $t>0$ and $f\in L^{r_0'}_\sigma(\Omega)$.
It follows from \eqref{en2} that
$\|e^{-t\CL_{\Omega,- \eta, - \omega}}f\|_{2,\Omega}\leq \|f\|_{2,\Omega}$ for
$t>0$ and $f\in C^\infty_{0,\sigma}(\Omega)$.
Hence, we infer from complex interpolation that
\begin{align}\label{qbounddual}
\|e^{-t\CL_{\Omega,- \eta, - \omega}}f\|_{q,\Omega}
\leq C\|f\|_{q,\Omega}
\end{align}
for $q\in[r_0',2]$. Let $\mu := (2-q) \slash (6 - q)$ so that there holds $[L^q (\Omega), L^6 (\Omega)]_\mu = L^2 (\Omega)$.
Then, together with the Sobolev embedding and the relation \eqref{en1}, we find
\begin{align*}
\|e^{-t\CL_{\Omega,- \eta, - \omega}}f\|_{2,\Omega}
& \leq \|e^{-t\CL_{\Omega,- \eta, - \omega}}f\|_{6,\Omega}^\mu
\|e^{-t\CL_{\Omega,- \eta, - \omega}}f\|_{q,\Omega}^{1-\mu} \\
&\leq C
\|\nabla e^{-t\CL_{\Omega,- \eta, - \omega}}f\|_{2,\Omega}^\mu
\|f\|_{q,\Omega}^{1-\mu}\\
&\leq C\left(-\frac{1}{2}\frac{\mathrm{d}}{\d t}
\|e^{-t\CL_{\Omega,- \eta, - \omega}}f\|_{2,\Omega}^2\right)^{\frac{\mu}{2}}
\|f\|_{q,\Omega}^{1-\mu}.
\end{align*}
Hence, for $t>0$ and $f\in C_{0,\sigma}^\infty(\Omega)\setminus\{0\}$ we see that
\begin{align*}
F(t):=
-\frac{\mu}{1-\mu}
(\|e^{-t\CL_{\Omega,\eta,\omega}}f\|^2_{2,\Omega})^{-\frac{1}{\mu}+1}
+\frac{2}{C}\|f\|_{q,\Omega}^{-\frac{2}{\mu}+2}t
\end{align*}
is monotone decreasing with respect to $t > 0$.
Therefore, we have
\begin{align*}
-\frac{\mu}{1-\mu}
\|e^{-t\CL_{\Omega,\eta,\omega}}f\|_{2,\Omega}^{-\frac{2}{\mu}+2}
+\frac{2}{C}\|f\|_{q,\Omega}^{-\frac{2}{\mu}+2}t
=F(t)\leq F(0)=-\frac{\mu}{1-\mu}
\|f\|_{2,\Omega}^{-\frac{2}{\mu}+2}\leq 0,
\end{align*}
which yields
\eqref{lplqomegadual} with $r=2$.
From \eqref{qbounddual} together with complex interpolation,
we obtain \eqref{lplqomegadual} for $r_0'\leq q\leq r \leq 2$.
The other estimate \eqref{lplqomega}
follows from the duality. The proof is complete.
\end{proof}

\begin{rmk}
It turns out that
\begin{align*}
\left|\frac{1}{r_0}-\frac{1}{2}\right|=\frac{1}{2}-\frac{1}{r_0}
\geq \frac{1}{2}-\frac{1}{q}=\left|\frac{1}{q}-\frac{1}{2}\right|
\end{align*}
if $2\leq q\leq q_0$ and that
\begin{align*}
\left|\frac{1}{r_0}-\frac{1}{2}\right|=
\left|\frac{1}{r_0'}-\frac{1}{2}\right|=
\frac{1}{r_0'}-\frac{1}{2}
\geq \frac{1}{q}-\frac{1}{2}=\left|\frac{1}{q}-\frac{1}{2}\right|
\end{align*}
if $r_0'\leq q \leq 2$. Then, we find that
the condition \eqref{p'}
is automatically fulfilled if either $2\leq q \leq r_0$ or
$r_0'\leq q \leq 2$ under the assumption \eqref{q0condi}.
\end{rmk}

\begin{proof}[Proof of Lemma \ref{proplqlr}]
Let $f\in C_{0,\sigma}^\infty(\Omega)$. We fix a cut-off function
$\phi_1 \in C_0^\infty(\BR^3)$ such that $\phi_1 =1$ 
on $B_{2R}(0)$ and
$\phi_1 =0$ on $\BR^3\setminus B_{3R}(0)$.
We set
\begin{align}\label{vperturb}
v(t)=e^{-t\CL_{\Omega,\eta,\omega}}f
-(1-\phi_1)T_{\BR^3,\eta,\omega}(t)f
-\BB_{K_1}[T_{\BR^3,\eta,\omega}(t)f\cdot\nabla\phi_1],
\end{align}
where $K_3=A_{2R,3R}$.
Then $v$ together with the pressure $p(t)$ associated with
$e^{-t\CL_{\Omega,\eta,\omega}}f$ obeys
\begin{equation}
\left\{\begin{aligned}
\partial_t v - \Delta v - (\eta + \omega\times x)\cdot\nabla v
+ \omega\times v + \nabla p & = F_1 & \quad & \text{in $\Omega \times \BR_+$}, \\
\div v & =0 & \quad & \text{for $\Omega \times \BR_+$},\\
v & =0 & \quad & \text{on $\pd \Omega \times \BR_+$}, \\
v(x,0) & = \widetilde{f} & \quad & \text{for $\Omega$},
\end{aligned}\right.
\end{equation}
where we have set
\begin{align*}
\widetilde{f} & :=\phi_1 f-\BB_{K_3}[f\cdot\nabla \phi_1], \\
F_1(x,t) & := -2\nabla\phi_1\cdot\nabla T_{\BR^3,\eta,\omega} f
-\Big(\Delta\phi_1+(\eta + \omega\times x)\cdot\nabla\phi_1\Big)T_{\BR^3,\eta,\omega}f
-\BB_{K_3}[\partial_tT_{\BR^3,\eta,\omega}f\cdot\nabla \phi_1]\\
&\;\quad +\Delta\BB_{K_3}[T_{\BR^3,\eta,\omega}f\cdot\nabla \phi_1]
+(\eta + \omega\times x)\cdot\nabla\BB_{K_3}
[T_{\BR^3,\eta,\omega}f\cdot\nabla \phi_1]
-\omega\times\BB_{K_3}[T_{\BR^3,\eta,\omega}f\cdot\nabla \phi_1].
\end{align*}
Note that the function $F_1$ fulfills
\begin{align}\label{F}
\|F_1(t)\|_{r,\Omega}\leq
\begin{cases}
C(c_0+1)t^{-\frac{1}{2}}\|f\|_{r,\Omega},\quad &t<1,\\
C(c_0+1)t^{-\frac{3}{2r}}\|f\|_{r,\Omega},\quad &t\geq 1
\end{cases}
\end{align}
for $1<r<\infty$. In view of \eqref{vperturb}
and $f\in C_{0,\sigma}^\infty(\Omega)$, we have
$v\in C^1((0,\infty);L^q_{\sigma}(\Omega))$ for $1<q<\infty$,
which implies
\begin{align*}
(v(t),\psi)_\Omega=(\widetilde{f},e^{-t\CL_{\Omega,- \eta, - \omega}}
\psi)_\Omega+
\int_0^t(F_1(\tau),e^{-(t-\tau)\CL_{\Omega,- \eta, - \omega}}\psi)_\Omega\d \tau
\end{align*}
for $\psi\in C_{0,\sigma}^\infty(\Omega).$
Let $r\in (2,\infty)$ satisfy \eqref{cond-r}.
To prove the desired assertion, it suffices to show
\begin{align}\label{vrbound}
\|v(t)\|_{r,\Omega}\leq C\|f\|_{r,\Omega}
\end{align}
for $t\ge2$ and $\lvert \eta \rvert + \lvert \omega \rvert \leq c_0$ on account of Lemma \ref{lem0}.
By \eqref{en2}, \eqref{F}, and Theorem \ref{thm1}, we deduce that
\begin{equation}
\label{fest}
\begin{split}
|(\widetilde{f},e^{-t\CL_{\Omega,- \eta, - \omega}}\psi)_\Omega|
& \leq
\|\widetilde{f}\|_{2,B_{3R}(0)}\|e^{-t\CL_{\Omega,- \eta, - \omega}}
\psi\|_{2,\Omega} \\
& \leq C\|f\|_{r,\Omega}
\|e^{-\CL_{\Omega,- \eta, - \omega}}\psi\|_{2,\Omega} \\
& \leq C\|f\|_{r,\Omega}
\|\psi\|_{r',\Omega}
\end{split}
\end{equation}
as well as
\begin{equation}
\label{Fest}
\begin{split}
& \left|\left(\int_0^1+\int_{t-1}^t\right)(F_1(\tau),
e^{-(t-\tau)\CL_{\Omega,- \eta, - \omega}}\psi)_\Omega\d \tau \right| \\
& \leq
\left(\int_0^1+\int_{t-1}^t\right)\|F_1(\tau)\|_{2,K_3}
\|e^{-(t-\tau)\CL_{\Omega,- \eta, - \omega}}\psi\|_{2,\Omega}\d \tau \\
& \leq C\|e^{-\CL_{\Omega,- \eta, - \omega}}\psi\|_{2,\Omega}
\int_0^1\|F_1(\tau)\|_{2,K_3}\d \tau+C\int_{t-1}^t
\|F_1(\tau)\|_{r,\Omega}
(t-\tau)^{-\frac{3}{2}(\frac{1}{r'}-\frac{1}{2})}\d \tau\,
\|\psi\|_{r',\Omega} \\
& \leq C(c_0+1)\|f\|_{r,\Omega}\|\psi\|_{r',\Omega}
+C(c_0+1)\int_{t-1}^t
\tau^{-\frac{3}{2r}}
(t-\tau)^{-\frac{3}{2}(\frac{1}{r'}-\frac{1}{2})}\d \tau\,
\|f\|_{r,\Omega}\|\psi\|_{r',\Omega} \\
& \leq C(c_0+1)\|f\|_{r,\Omega}\|\psi\|_{r',\Omega}
\end{split}
\end{equation}
for $t\ge2$ and $\psi\in C_{0,\sigma}^\infty(\Omega)$.
Set
\begin{align}
J:=\int_1^{t-1}(F_1(\tau),e^{-(t-\tau)\CL_{\Omega,- \eta, - \omega}}
\psi)_\Omega\d \tau.
\end{align}
Since $e^{-t\CL_{\Omega,- \eta, - \omega}}\psi\in W^{1,2}_0(\Omega)$
vanishes on $\partial \Omega$, we use the Poincar\'{e} inequality
in $\Omega\cap B_{3R}(0)$ to see
\begin{align}
|J|&\leq \int_1^{t-1}\|F_1(\tau)\|_{2,K_3}
\|e^{-(t-\tau)\CL_{\Omega,- \eta, - \omega}}
\psi\|_{2,\Omega \cap B_{3R}(0)}\d \tau\notag\\
&\leq C(c_0+1)\|f\|_{r,\Omega}
\int_1^{t-1}\tau^{-\frac{3}{2r}}
\|\nabla e^{-(t-\tau)\CL_{\Omega,- \eta, - \omega}}\psi\|_{2,\Omega}\d \tau
\label{integral}\\
&\leq C(c_0+1)\|f\|_{r,\Omega}
\left(\int_1^{t-1}\tau^{-\frac{3}{r}}\d \tau\right)^{\frac{1}{2}}
\left(\int_1^{t-1}
\|\nabla e^{-(t-\tau)\CL_{\Omega,- \eta, - \omega}}\psi\|^2_{2,\Omega}\d \tau
\right)^{\frac{1}{2}}\notag\\
&\leq C\|f\|_{r,\Omega}\|e^{-\CL_{\Omega,- \eta, - \omega}}\psi\|_{2,\Omega} \\
& \leq C\|f\|_{r,\Omega}\|\psi\|_{r',\Omega}\notag
\end{align}
for $t\geq 2$ and $2<r<3$,
where we have used \eqref{en2}.
Combined with \eqref{fest} and \eqref{Fest}, we have
\eqref{vrbound} for $2<r<3$. Hence, together with
Lemma~\ref{lem0}, we see that
\eqref{lplqomegadual} holds for $3/2<q\leq r\leq 2$. \par
Let $3<r<\infty$ satisfy \eqref{rcondi}.
Given $\delta>0$ arbitrarily small,
we take $p_0\in(3/2,2)$ (so close to $3/2$) to deduce
\begin{align}\label{lq'l2est}
\|e^{-t\CL_{\Omega,- \eta, - \omega}}\psi\|_{2,\Omega}
\leq C(t-1)^{-\frac{1}{4}+\delta}
\|e^{-\CL_{\Omega,- \eta, - \omega}}\psi\|_{p_0,\Omega}\leq
C(t-1)^{-\frac{1}{4}+\delta}
\|\psi\|_{r',\Omega}
\end{align}
for $t>1$. Then we infer from
\begin{equation}
\label{int0}
\begin{split}
\int_1^{\frac{t}{2}}
\|\nabla e^{-(t-\tau)\CL_{\Omega,- \eta, - \omega}}\psi\|_{2,\Omega}^2\d \tau
& =
\int_{\frac{t}{2}}^{t-1}
\|\nabla e^{-\tau'\CL_{\Omega,- \eta, - \omega}}
\psi\|_{2,\Omega}^2\d \tau' \\
& \leq
\|e^{-\frac{t}{2}\CL_{\Omega,- \eta, - \omega}}\psi\|^2_{2,\Omega} \\
& \leq C(t-1)^{-\frac{1}{2}+2\delta}\|\psi\|^2_{r',\Omega}
\end{split}
\end{equation}
that
\begin{equation}
\label{int3}
\begin{split}
\int_{1}^{\frac{t}{2}}\tau^{-\frac{3}{2r}}
\|\nabla e^{-(t-\tau)\CL_{D,-\eta,-\omega}}\psi\|_{2,D}\d \tau
&\leq \left(\int_1^{\frac{t}{2}}
\tau^{-\frac{3}{r}}\d \tau\right)^{\frac{1}{2}}
\left(\int_1^{\frac{t}{2}}
\|\nabla e^{-(t-\tau)\CL_{D,-\eta,-\omega}}\psi\|^2_{2,D}\d \tau
\right)^{\frac{1}{2}} \\
&\leq Ct^{-\frac{3}{2r}+\frac{1}{2}}(t-1)^{-\frac{1}{4}+\delta}
\|\psi\|_{r',\Omega}
\end{split}
\end{equation}
for $t\geq 2$ and $\psi\in C_{0,\sigma}^\infty(\Omega).$
By \eqref{en2} and \eqref{lq'l2est}, we observe
\begin{align*}
\int_{1+t}^{1+2t}
\|\nabla e^{-\tau'\CL_{\Omega,- \eta, - \omega}}\psi\|^2_{2,\Omega}\d \tau'
\leq Ct^{-\frac{1}{2}+2\delta}\|\psi\|_{r',\Omega}
\end{align*}
for $t>0$ and $\psi\in C_{0,\sigma}^\infty(\Omega)$, thereby,
we may employ \cite[Lem. 3.4]{H18} with $s=1$ and $t=t/2-1$ to obtain
\begin{align}\label{int1}
\int_{\frac{t}{2}}^{t-1}
\|\nabla e^{-(t-\tau)\CL_{\Omega,- \eta, - \omega}}\psi\|_{2,\Omega}\d \tau\leq
C(t-1)^{\frac{1}{4}+\delta}\|\psi\|_{r',\Omega},
\end{align}
which leads us to 
\begin{align}\label{int2}
\int_{\frac{t}{2}}^{t-1}\tau^{-\frac{3}{2r}}
\|\nabla e^{-(t-\tau)\CL_{\Omega,- \eta, - \omega}}\psi\|_{2,\Omega}\d \tau
\leq Ct^{-\frac{3}{2r}}(t-1)^{\frac{1}{4}+\delta}
\|\psi\|_{r',\Omega}.
\end{align}
Thus, there holds
\begin{align*}
|J|\leq Ct^{-\frac{3}{2r}+\frac{1}{4}+\delta}
\|f\|_{r,\Omega}\|\psi\|_{r',\Omega}
\leq C\|f\|_{r,\Omega}\|\psi\|_{r',\Omega}
\end{align*}
for $t\geq 2$ with $3<r<6$
by applying \eqref{int3} and \eqref{int2} to \eqref{integral}.
Finally, together with
\eqref{fest} and \eqref{Fest}, we arrive at \eqref{vrbound} for $3<r<6$.
\par It remains to prove the case $6\leq r<\infty$ (if $\varepsilon_2 > 1 \slash 6$).
Suppose $6\leq r<\infty$ with \eqref{cond-r}.
Then by Lemma~\ref{lem0}, it turns out that
the estimate \eqref{lq'l2est} is replaced by
\begin{align*}
\|e^{-t\CL_{\Omega,- \eta, - \omega}}\psi\|_{2,\Omega}
\leq C(t-1)^{-\frac{1}{2}+\delta}\|\psi\|_{r',\Omega}.
\end{align*}
We thus find that \eqref{int0} and \eqref{int1} are replaced by
\begin{align*}
\int_1^{\frac{t}{2}}
\|\nabla e^{-(t-\tau)\CL_{\Omega,- \eta, - \omega}}\psi\|^2_{2,\Omega}\d \tau
& \leq
C(t-1)^{-1+2\delta}\|\psi\|^2_{r',\Omega},\\
\int_{\frac{t}{2}}^{t-1}
\|\nabla e^{-(t-\tau)\CL_{\Omega,- \eta, - \omega}}\psi\|_{2,\Omega}\d \tau & \leq
C(t-1)^{\delta}\|\psi\|_{r',\Omega},
\end{align*}
respectively.
By applying these estimates to the integral in \eqref{integral}, we have
$|J|\leq Ct^{-\frac{3}{2r}+\delta}\|f\|_{r,\Omega}
\|\psi\|_{r',\Omega}$
for $t\geq 2$, which yields \eqref{vrbound}
for $r\in[6,\infty)$.
The proof is complete.
\end{proof}

\section{Proofs of Theorems \ref{thmlqlr} and \ref{thmlqlr2}}
\label{sec-6}

This section is devoted to proving Theorems \ref{thmlqlr} 
and \ref{thmlqlr2}.
In the proof of Theorem \ref{thmlqlr}, 
the local energy estimates of the
semigroup $(e^{- t\CL_{\Omega,\eta,\omega}})_{t\ge0}$ over
a bounded domain $\Omega_R = \Omega \cap B_R (0)$ 
(near the boundary), cf.
Lemma \ref{lemsigma1} below, will play a crucial role. 
In order to accomplish the proofs of 
Theorems \ref{thmlqlr} and \ref{thmlqlr2}, 
in Proposition \ref{propr0} below, 
we also clarify the relation between the decay rate 
in the local energy decay estimates and 
the range of exponents $q,r$ for $L^q$-$L^r$ estimates of 
$(e^{-t\CL_{\Omega,\eta,\omega}})_{t\ge0}$. 
\subsection{Local energy decay estimates}
We first establish the local energy estimates of the
semigroup $(e^{- t\CL_{\Omega,\eta,\omega}})_{t\ge0}$ over $\Omega_R$.
To this end, we first prepare some regularity properties of
the semigroup $(e^{- t\CL_{\Omega,\eta,\omega}})_{t\ge0}$.

\begin{prop}
\label{prop-6.1}
Let $\varepsilon_2$ and $q$ be subject to Convention \ref{conv-2}
and $\eta, \omega \in \BR^3$. 
Given $c_0>0,$ assume $\lvert \eta \rvert + 
\lvert \omega \rvert \leq c_0.$
Let $\alpha>0$ be the same number as in
Theorem \ref{proplqlrsmooth}.
Given $f\in L^q_\sigma(\Omega)$,
set $u(t):=e^{-t\CL_{\Omega,\eta,\omega}}f$.
Fix $R\in (R_0+1,\infty)$ and let $\phi\in C_0^\infty(B_R(0))$
fulfill $\phi=1$ in $B_{R_0+1}(0)$. Then
there holds $u\in C^1((0,\infty);W^{-1,q}(\Omega_R))$.
For each $T>0$, there exists a constant $C>0$ such that
\begin{align}
\|\partial_te^{-t\CL_{\Omega,\eta,\omega}}f
\|_{W^{-1,q}(\Omega_R)} &\leq Ct^{-1+\alpha}
\|f\|_{q,\Omega}\label{partialtest2}, \\
\|\BB_{A_{R_0,R}}[\partial_te^{-t\CL_{\Omega,\eta,\omega}}f\cdot\nabla \phi
]\|_{q,\Omega_R} & \leq Ct^{-1+\alpha}
\|f\|_{q,\Omega}\label{partialtest3}
\end{align}
for $t\leq T$ and $f\in L^q_\sigma(\Omega)$, 
where $A_{R_0,R}=\{x\in\BR^3\mid R_0<|x|<R\}$. 
Moreover, there exists the pressure $p(t)$ subject to
$\int_{\Omega_R}p(t)\d x=0$ for $t>0$ such that
the pair $(u,p)$ obeys
\begin{align}\label{eq2}
(\partial_tu,\psi)_{\Omega_R}=(\nabla u+u\otimes(\eta + \omega\times x)
-(\omega\times x)\otimes u,\nabla\psi)_{\Omega_R}
-(p,\div\psi)_{\Omega_R}=0
\end{align}
for $t>0$ and $\psi\in W_0^{1,q'}(\Omega_R)$ and that
\begin{align}
\|p(t)\|_{q,\Omega_R}
\leq Ct^{-1+\alpha}\|f\|_{q,\Omega}\label{pest2}
\end{align}
for $t\leq T$ and $f\in L^q_\sigma(\Omega)$ as well as
\begin{align}\label{pest3}
\|p(t)\|_{q,\Omega_R}
\leq C\|\partial_te^{-t\CL_{\Omega,\eta,\omega}}f\|_{W^{-1,q}(\Omega_R)}
+C\|u(t)\|_{W^{1,q}(\Omega_R)}
\end{align}
for $t>0$ and $f\in L^q_\sigma(\Omega)$.
\end{prop}
\begin{proof}
We carry out the same argument in the proof of Hishida 
\cite[Prop. 5.1]{H20} to complete the proof. 
Let us use the form introduced in the proof of Theorem \ref{thm1}.
From Theorem \ref{proplqlrsmooth} and
Lemma \ref{lempartialtwhole}, there holds
$T_{\Omega,\eta,\omega,0}(t)f\in C^1((0,\infty);W^{-1,q}(\Omega_R))$ with
\begin{equation}
\|\partial_t T_{\Omega,\eta,\omega,0}(t)f\|_{W^{-1,q}(\Omega_R)}
\leq Ct^{-1+\alpha}\|f\|_{q,\Omega}
\end{equation}
for $t\leq T$ and $f\in L^q_\sigma(\Omega)$.
Hence, together with \eqref{shest},
we obtain
$T_{\Omega,\eta,\omega,j}(t)f\in C^1((0,\infty);W^{-1,q}(\Omega_R))$ for every $j \in \BN$ and
\begin{alignat}{3}
\partial_t T_{\Omega,\eta,\omega,1}(t)f & =H(t)f+\int_0^t\partial_tT_{\Omega,\eta,\omega,0}(t-\tau)
H(\tau)f\d \tau,\label{partialt0}\\
\partial_t T_{\Omega,\eta,\omega, j+1}(t)f & =\int_0^t\partial_tT_{\Omega,\eta,\omega,j}(t-\tau)H(\tau)f
\d \tau, & \qquad j\in \BN. \label{partialt2}
\end{alignat}
We deduce
\begin{align*}
\|\partial_t T_{\Omega,\eta,\omega,0}(t)f\|_{W^{-1,q}(\Omega_R)}
&\leq c_0t^{-1+\alpha}\|f\|_{q,\Omega},\\
\|\partial_t T_{\Omega,\eta,\omega,1}(t)f\|_{W^{-1,q}(\Omega_R)}
&\leq (c_0t^{-1+\alpha}
+c_0^2B(\alpha,\alpha)t^{-1+2 \alpha})\|f\|_{q,\Omega}\\
\|\partial_t T_{\Omega,\eta,\omega,j+1}(t)f\|_{W^{-1,q}(\Omega_R)}
&\leq \left(\frac{(c_0\Gamma(\alpha))^{j+1}}{\Gamma((j+1)\alpha)}
+\frac{(c_0\Gamma(\alpha))^{j+2}}{\Gamma((j+2)\alpha)}\right)
t^{-1+\alpha}\|f\|_{q,\Omega}
\end{align*}
for $t\leq T$ and $f\in L^q_\sigma(\Omega)$,
where $c_0$ is some fixed constant and
$B(\,\cdot\,,\,\cdot\,)$ and $\Gamma(\,\cdot\,)$ denote the beta function
and the gamma function, respectively. Therefore, the series
$\sum_{j=0}^\infty\partial_tT_{\Omega,\eta,\omega,j}(t)f$
converges in $W^{-1,q}(\Omega_R)$ uniformly with respect to
$t\in [\delta,T]$ for $0<\delta<T$ and
$\partial_t T(t)f=\sum_{j=0}^\infty\partial_tT_{\Omega,\eta,\omega,j}(t)f$
in $W^{-1,q}(\Omega_R)$ with
$\|\partial_t T(t)f\|_{W^{-1,q}(\Omega_R)}
\leq Ct^{-1+\alpha}\|f\|_{q,\Omega},$
which yields \eqref{partialtest2}. \par
Suppose $f\in C_{0,\sigma}^\infty(\Omega)$.
Then we have $e^{-t\CL_{\Omega,\eta,\omega}}f\in \SD(\CL_{\Omega,\eta,\omega})$.
By the definition of $\CL_{\Omega,\eta,\omega}$ and by
\begin{align*}
\|p(t)\|_{q,\Omega}=\sup_{g\in L^{q'}_0(\Omega),\|g\|_{q',\Omega}
\leq 1}\left|\int_\Omega p(t)\overline{g}\d x\right|,\qquad
L^q_0(\Omega):=\left\{F\in L^q(\Omega)\relmiddle|
\int_\Omega F\d x=0\right\},
\end{align*}
we have \eqref{eq2} and \eqref{pest3},
thereby, the estimate \eqref{pest2} holds.
We next take a general $f\in L^q_\sigma(\Omega)$,
then approximation procedure yields
$p_R\in C((0,\infty);L^q(\Omega_R))$.
Together with $u(t)=e^{-t\CL_{\Omega,\eta,\omega}}f$, we see that
\eqref{eq2}, \eqref{pest2}, and \eqref{pest3} follow.
By virtue of this procedure,
for any integer $k>0$, we obtain the pressure $p_{R+k}$
over $\Omega_{R+k}$ satisfying
$\int_{\Omega_{R+k}}p_{R+k}\d x=0$, but
it follows from
$(p_{R+k}(t)-p_{R+j},\div\psi)_{\Omega_{R+j}}=0$ for
$\psi\in C_0^\infty(\Omega_{R+j})$ and $k>j\geq 0$
that $p_{R+k}(x,t)=p_R(x,t)=c_k(t)$ almost everywhere
in $\Omega_R$ with
some $c_k(t)$ independent of $x\in \Omega_R$.
We thus obtain the pressure over $\Omega$ defined by
\begin{align*}
p(x,t)=
\begin{cases}
p_R(x,t),&x\in \Omega_R, \\
p_{R+k}(x,t)-c_k(t),\quad
&x\in \Omega_{R+k}\setminus \Omega_{R+k-1}, \quad k \in \BN,
\end{cases}
\end{align*}
which fulfills $\int_{\Omega_R}p(t)\d x=0$ as well as
\eqref{pest2} and \eqref{pest3}. The proof is complete.
\end{proof}

We also prepare the following lemma.

\begin{lem}\label{lemlrlr}
Let $\varepsilon_2$ and $q$ be subject to Convention \ref{conv-2}
and $\eta, \omega \in \BR^3$. 
Given $c_0>0,$ assume $\lvert \eta \rvert + \lvert \omega \rvert \leq c_0.$
Let $r\in[q,\infty)$.
Then there hold
\begin{align}
\|e^{-t\CL_{\Omega,\eta,\omega}}f\|_{W^{1,r}(\Omega)}
\leq C\|e^{-(t-1)\CL_{\Omega,\eta,\omega}}f\|_{r,\Omega}\label{lrlr}
\end{align}
and
\begin{align}\label{partialtlrlr}
\|\partial_te^{-t\CL_{\Omega,\eta,\omega}}f\|_{W^{-1,r}(\Omega_R)}
\leq C\|e^{-(t-1)\CL_{\Omega,\eta,\omega}}f\|_{r,\Omega}
\end{align}
for $t>1$ and $f\in L^q_\sigma(\Omega)$.
\end{lem}

\begin{proof}
We use the form introduced in the proof of Theorem \ref{thm1}.
We fix $q_0$ satisfying
\begin{align*}
\max\{3,q\}<q_0<\frac{3}{1-3\varepsilon_2}
\quad{\rm if}~r\geq \frac{3}{1-3\varepsilon_2},\qquad q_0=r
\quad{\rm if}~r< \frac{3}{1-3\varepsilon_2}
\end{align*}
and take $\alpha=\alpha(q_0)$ so that \eqref{alphacondi}
with $q=q_0$
and $2\alpha\leq 1-3/q_0$ are fulfilled,
which implies $1/q_0-1/r\leq (1-2\alpha)/3$.
Let $T>0$.
Then the operators
$T_{\Omega,\eta,\omega,0}(t):L^{q_0}_\sigma(\Omega)\cap
L^r(\Omega)\rightarrow W^{1,r}(\Omega)$ and
$H(t):L^{q_0}_\sigma(\Omega)\cap L^r(\Omega)
\rightarrow L^{q_0}_\sigma(\Omega)
\cap L^r(\Omega)$ fulfill
\begin{align*}
\|T_{\Omega,\eta,\omega,0}(t)f\|_{W^{1,r}(\Omega)}
& \leq Ct^{-\frac{3}{2}(\frac{1}{q_0}-\frac{1}{r})-\frac{1}{2}}
\|f\|_{q_0,\Omega}\leq Ct^{-\frac{3}{2}(\frac{1}{q_0}-\frac{1}{r})
-\frac{1}{2}}\|f\|_{L^{q_0}(\Omega)\cap L^r(\Omega)},\\
\|H(t)f\|_{L^{q_0} (\Omega) \cap L^r (\Omega)}
& \leq Ct^{-1+\alpha}\|f\|_{r,\Omega}
\end{align*}
for $0 < t\leq T$ and 
$f\in L^{q_0}_\sigma(\Omega)\cap L^r(\Omega)$, 
where $\|f\|_{L^{q_0}(\Omega)\cap L^r(\Omega)}=
\max\{\|f\|_{q_0,\Omega},\|f\|_{r,\Omega}\}$. 
From these estimates, carrying out the same procedure as in the proof of 
Lemma \ref{lemserires} implies 
\begin{align}
\label{w1r-est}
\|e^{-t\CL_{\Omega,\eta,\omega}}f\|_{W^{1,r}(\Omega)}
\leq Ct^{-\frac{3}{2}(\frac{1}{q_0}-\frac{1}{r})-\frac{1}{2}}
\|f\|_{r,\Omega}
\end{align}
for $0 < t \le T$ and 
$f\in L^{q_0}_\sigma(\Omega)\cap L^r(\Omega)$. 
Here, we note that \eqref{w1r-est} do not directly follow
from Lemma \ref{lemserires} since 
Lemma \ref{lemserires} with 
$X_1=L^{q_0}_\sigma(\Omega)\cap L^r(\Omega)$ and
$X_2=W^{1,r}(\Omega)$ yields \eqref{w1r-est}, in which 
$\|f\|_{r,\Omega}$ is replaced 
by  $\|f\|_{L^{q_0}_\sigma(\Omega)\cap L^r(\Omega)}
(\ge \|f\|_{r,\Omega})$.
The estimate \eqref{w1r-est} combined with 
$e^{-t\CL_{\Omega,\eta,\omega}}f=e^{-\CL_{\Omega,\eta,\omega}}
e^{-(t-1)\CL_{\Omega,\eta,\omega}}f$ asserts \eqref{lrlr}.
Similarly, it follows from \eqref{partialt0} and \eqref{partialt2} that
\begin{align*}
\|\partial_\tau e^{-\tau\CL_{\Omega,\eta,\omega}}f\|_{W^{-1,r}(\Omega_R)}
\leq C\tau^{-1+\alpha}\|f\|_{r,\Omega}
\end{align*}
for $0 < \tau \le T$ and $f\in L^{q_0}_\sigma(\Omega)\cap L^r(\Omega)$,
where $q_0$ and $\alpha=\alpha(q_0)$ are defined as above.
Moreover, we have
\begin{align*}
\frac1h \big(e^{-(t+h)\CL_{\Omega,\eta,\omega}}f-e^{-t\CL_{\Omega,\eta,\omega}}f \big)
=\frac1h \big(e^{-(1+h)\CL_{\Omega,\eta,\omega}}-e^{-\CL_{\Omega,\eta,\omega}}\big) e^{-(t-1)
\CL_{\Omega,\eta,\omega}}f
\end{align*}
for $t>1$, $h>-1$, and $f\in L^{q}_\sigma(\Omega)$.
Hence, there holds
\begin{equation}
\partial_te^{-t\CL_{\Omega,\eta,\omega}}f
=\big(\partial_{\tau}e^{-\tau\CL_{\Omega,\eta,\omega}}
e^{-(t-1)\CL_{\Omega,\eta,\omega}}f\big)\big|_{\tau=1} 
\qquad \text{in $W^{-1,r}(\Omega_R)$}
\end{equation}
with
$e^{-(t-1)\CL_{\Omega,\eta,\omega}}f\in L^{q_0}_\sigma(\Omega)
\cap L^r(\Omega)$.
Thus the estimate \eqref{partialtlrlr} follows.
\end{proof}

We now investigate how the rate $\sigma_2$ in 
\eqref{local} below inherits from the rate $\sigma_1$ in the 
local energy decay estimates \eqref{drest}.

\begin{lem}\label{lemsigma1}
Let $\varepsilon_2$ and $q$ be subject to Convention \ref{conv-2} 
and $\eta, \omega \in \BR^3$. 
Given $c_0>0,$ assume $\lvert \eta \rvert + \lvert \omega \rvert \leq c_0.$
Let $R\in(R_0+1,\infty)$ and let $\phi\in C_0^\infty(B_R(0))$
fulfill $\phi=1$ in $B_{R_0+1}(0)$.
Assume that the following estimate holds with some $\sigma_1>0:$
\begin{align}\label{drest}
\|e^{-t\CL_{\Omega,\eta,\omega}}f\|_{W^{1,q}(\Omega_R)}+
\|\partial_te^{-t\CL_{\Omega,\eta,\omega}}f\|_{W^{-1,q}(\Omega_R)}
\leq Ct^{-\sigma_1}
\|f\|_{q,\Omega}
\end{align}
for $t\geq 2$ and $f\in L^q_\sigma(\Omega)$ with
$f=0$ a.e. $\BR^3\setminus B_{R_0}(0)$. Then there holds
\begin{multline}
\label{local}
\|e^{-t\CL_{\Omega,\eta,\omega}}f\|_{W^{1,q}(\Omega_R)}+
\|\partial_te^{-t\CL_{\Omega,\eta,\omega}}f\|_{W^{-1,q}(\Omega_R)} \\
+ \|\BB_{A_{R_0,R}}[\partial_te^{-t\CL_{\Omega,\eta,\omega}}
f\cdot\nabla\phi]\|_{q,\Omega_R}
+\|p(t)\|_{q,\Omega_R}
\leq Ct^{-\sigma_2}\|f\|_{q,\Omega}
\end{multline}
for $t\geq 2$ and $f\in L^q_\sigma(\Omega)$,
where $p(t)$ is the pressure associated with
$e^{-t\CL_{\Omega,\eta,\omega}}f$ and
\begin{align}\label{sigma1def}
\sigma_2:=\min\left\{\sigma_1,\frac{3}{2q}+\sigma_1-1
-\epsilon,\frac{3}{2q}\right\}
\end{align}
with arbitrarily small $\epsilon>0.$
Here, the notation $A_{R_0,R}=\{x\in\BR^3\mid R_0<|x|<R\}$ has been used.
\end{lem}
\begin{proof}
For a technical reason, let  $\phi_2\in C_0^\infty(B_{R+1}(0))$ 
fulfill $\phi_2=1$ in $B_{R}(0)$.
We regard $e^{-t\CL_{\Omega,\eta,\omega}}f$ as the perturbation
from the modification of $T_{\BR^3,\eta,\omega}(t)$, i.e.,
\begin{equation}
\label{f0def}
\begin{split}
e^{-t\CL_{\Omega,\eta,\omega}}f&=(1-\phi_2)T_{\BR^3,\eta,\omega}(t)f_0
+\BB_{A_{R_0+1,R+1}}[T_{\BR^3,\eta,\omega}(t)f_0\cdot\nabla \phi_2]+v(t), \\
f_0&=
\begin{cases}
(1-\phi)f+\BB_{A_{R_0,R}}[f\cdot\nabla \phi]\quad &\text{if}~x\in D,\\
0&\text{if}~x\in\BR^3\setminus D,
\end{cases}
\end{split}
\end{equation}
where $v(t)$ denotes the perturbation.
Then the $L^q$-$L^\infty$-estimate
of $T_{\BR^3,\eta,\omega}(t)$ together with 
Proposition \ref{prop-Bogovskii} implies
that the term $(1-\phi)T_{\BR^3,\eta,\omega}(t)f_0
+\BB_{A_{R_0+1,R+1}}[T_{\BR^3,\eta,\omega}(t)f_0\cdot\nabla \phi]$ has
the better decay rate $t^{-3/(2q)}$ than $t^{-\sigma_2}$.
Therefore, it suffices to estimate
\begin{align}\label{integral1}
v(t)=e^{-t\CL_{\Omega,\eta,\omega}}v(0)+
\int_0^te^{-(t-\tau)\CL_{\Omega,\eta,\omega}}F_2(\tau)\d \tau
\end{align}
and
\begin{align}\label{integral2}
\partial_tv(t)=\partial_te^{-t\CL_{\Omega,\eta,\omega}}v(0)+F_2(t)+
\int_0^t\partial_te^{-(t-\tau)\CL_{\Omega,\eta,\omega}}F_2 (\tau)\d \tau,
\end{align}
where $v(0)=\phi_2f
-\BB_{A_{R_0+1,R+1}}[f\cdot\nabla \phi_2]$ and
\begin{align*}
F_2(x,t)&=-2\nabla\phi_2
\cdot\nabla T_{\BR^3,\eta,\omega}(t)f_0
-[\Delta\phi_2+(\eta + \omega\times x)\cdot\nabla
\phi_2]T_{\BR^3,\eta,\omega}(t)f_0 \\
&\quad -\BB_{A_{R_0+1,R+1}}[\partial_tT_{\BR^3,\eta,\omega}(t)f_0\cdot\nabla\phi_2]
+\Delta\BB_{A_{R_0+1,R+1}}[T_{\BR^3,\eta,\omega}(t)f_0\cdot\nabla\phi_2] \\
& \quad +(\eta + \omega\times x)\cdot
\nabla \BB_{A_{R_0+1,R+1}}[T_{\BR^3,\eta,\omega}(t)f_0\cdot\nabla\phi_2]
-\omega\times \BB_{A_{R_0+1,R+1}}[T_{\BR^3,\eta,\omega}(t)f_0\cdot\nabla\phi_2].
\end{align*}
It follows that $\mathrm{supp}\, v(0)\subset \Omega_{R}$,
$\mathrm{supp}\, F_2\subset \Omega_{R}$. 
By Proposition \ref{prop-Bogovskii} and 
Proposition \ref{prop-Bogovskii2}, we also have 
$\|v(0)\|_{q,\Omega}\leq C\|f\|_{q,\Omega}$ and 
\begin{equation}
\|F_2(t)\|_{q,\Omega}
\leq Ct^{-\frac{1}{2}}(1+t)^{\frac{1}{2}-\frac{3}{2q}}
\|f\|_{q,\Omega}, \qquad t > 0.
\end{equation}
Together with \eqref{lqlrsmooth2} and the assumption \eqref{drest}
we find that
\begin{align*}
\|e^{-t\CL_{\Omega,\eta,\omega}}v(0)\|_{W^{1,q}(\Omega_R)}
\leq Ct^{-\sigma_1}\|f\|_{q,\Omega}
\end{align*}
for $t\geq 2$ and that
\begin{align*}
\|e^{-(t-\tau)\CL_{\Omega,\eta,\omega}}F_2(\tau)\|_{W^{1,q}(\Omega_R)}
\leq C(t-\tau)^{-\frac{1}{2}}(1+t-\tau)^{\frac{1}{2}-\sigma_1}
\tau^{-\frac{1}{2}}(1+\tau)^{\frac{1}{2}-\frac{3}{2q}}
\|f\|_{q,\Omega}
\end{align*}
for all $t>\tau$.
By applying these estimates to \eqref{integral1}, we obtain
\begin{align}
\|v(t)\|_{W^{1,q}(\Omega_R)} \leq Ct^{-\sigma_1}\|f\|_{q,\Omega}
+C\int_0^t(t-\tau)^{-\frac{1}{2}}(1+t-\tau)^{\frac{1}{2}-\sigma_1}
\tau^{-\frac{1}{2}}(1+\tau)^{\frac{1}{2}-\frac{3}{2q}}\d \tau
\|f\|_{q,\Omega}.
\end{align}
Setting
\begin{equation}
G_1 (t, \tau) := (t-\tau)^{-\frac{1}{2}}(1+t-\tau)^{\frac{1}{2}-\sigma_1}
\tau^{-\frac{1}{2}}(1+\tau)^{\frac{1}{2}-\frac{3}{2q}},
\end{equation}
we see that
\begin{align}
\label{int00}
\begin{split}
\int_0^1 G_1 (t,\tau) \d \tau &\leq Ct^{-\sigma_1}\int_0^1
\tau^{-\frac{1}{2}}(1+\tau)^{\frac{1}{2}-\frac{3}{2q}}\d \tau\leq
Ct^{-\sigma_1}, \\
\int_1^{\frac{t}{2}}G_1 (t,\tau) \d \tau&\leq Ct^{-\sigma_1}\int_1^{\frac{t}{2}}
\tau^{-\frac{3}{2q}}\d \tau=
Ct^{-\sigma_1} g_1 (t) \leq Ct^{-\sigma_2}, \\
\int_{\frac{t}{2}}^{t-1}G_1 (t,\tau) \d \tau&\leq Ct^{-\frac{3}{2q}}
\int_{\frac{t}{2}}^{t-1}
(t-\tau)^{-\sigma_1}\d \tau
=Ct^{-\frac{3}{2q}} g_2 (\sigma_1, t)
\leq Ct^{-\sigma_2}, \\
\int_{t-1}^tG_1 (t,\tau) \d \tau&\leq Ct^{-\frac{3}{2q}}\int_{t-1}^t
(t-\tau)^{-\frac{1}{2}}\d \tau
\leq Ct^{-\frac{3}{2q}}
\end{split}
\end{align}
for $t\geq 2$ with
\begin{equation}
g_1 (t) := \begin{cases}
Ct^{-\frac{3}{2q}+1}\quad&\text{if}~q>
\displaystyle\frac{3}{2},\\[12pt]
\log t &\text{if}~q=\displaystyle\frac{3}{2},\\[12pt]
1 &\text{if}~q<\displaystyle\frac{3}{2},
\end{cases}
\qquad
g_2 (\sigma_1, t) :=
\begin{cases}
t^{-\sigma_1+1}\quad&\text{if}~\sigma_1<1,\\
\log t&\text{if}~\sigma_1=1,\\
1&\text{if}~\sigma_1>1.
\end{cases}
\end{equation}
Combining these estimates implies
\begin{align*}
\|e^{-t\CL_{\Omega,\eta,\omega}}f\|_{W^{1,q}(\Omega_R)}
\leq Ct^{-\sigma_2}\|f\|_{q,\Omega}
\end{align*}
for $t\geq 2$ and $f\in L^q_\sigma(\Omega)$.
By taking into account \eqref{partialtest2}, \eqref{drest},
Proposition \ref{prop-Bogovskii}, and 
Proposition \ref{prop-Bogovskii2}, we have
\begin{align*}
\|\partial_te^{-t\CL_{\Omega,\eta,\omega}}v(0)
\|_{W^{-1,q}(\Omega_R)}\leq Ct^{-\sigma_1}\|f\|_{q,\Omega}
\end{align*}
for $t\geq 2$ and
\begin{align*}
\|\partial_te^{-(t-\tau)\CL_{\Omega,\eta,\omega}}F_2(\tau)
\|_{W^{-1,q}(\Omega_R)}
\leq C(t-\tau)^{-1+\alpha}(1+t-\tau)^{1-\alpha-\sigma_1}
\tau^{-\frac{1}{2}}(1+\tau)^{\frac{1}{2}-\frac{3}{2q}}
\|f\|_{q,\Omega}
\end{align*}
for $t>\tau$. Hence, it follows from \eqref{integral2} that
\begin{align*}
\|\partial_te^{-t\CL_{\Omega,\eta,\omega}}f\|_{W^{-1,q}(\Omega_R)}
\leq Ct^{-\sigma_2}\|f\|_{q,\Omega},
\end{align*}
which together with Proposition \ref{prop-Bogovskii2} 
and \eqref{pest3} completes the proof.
\end{proof}

\subsection{Proofs of Theorems \ref{thmlqlr} and \ref{thmlqlr2}}
To prove Theorem \ref{thmlqlr}, 
it remains to establish decay estimates outside $B_R(0)$
(near infinity). In paricular, 
we first investigate the relation between 
the rate ($\sigma_1$ or $\sigma_2$) 
in the estimates near the boundary and the one 
near infinity.  
Together with the local energy
decay estimates concerned in the previous subsection,
we will prove Theorem \ref{thmlqlr}. 
By real interpolation, we see that 
Theorem \ref{thmlqlr} induces Theorem \ref{thmlqlr2}. 
\begin{lem}\label{lemsigma2}
Let $\varepsilon_2$ and $q$ be subject to Convention \ref{conv-2}
and $\eta, \omega \in \BR^3$. 
Given $c_0>0,$ assume $\lvert \eta \rvert + \lvert \omega \rvert \leq c_0.$
Let $R\in(R_0+1,\infty)$.
Assume that \eqref{drest} with some $\sigma_1>0$ holds for
$t\geq 2$ and $f\in L^q_\sigma(\Omega)$ provided that
$f=0$ a.e. $\BR^3\setminus B_{R_0}$.
Then there holds
\begin{align}\label{lqlrinfy}
\|e^{-t\CL_{\Omega,\eta,\omega}}f\|_{r,\BR^3\setminus B_R(0)}\leq
Ct^{-\sigma_3}\|f\|_{q,\Omega}
\end{align}
for $t\geq 2$ and $f\in L^q_\sigma(\Omega)$ provided that
$q\leq r\leq\infty$ and $1/q-1/r<2/3$, where $\sigma_3$ is given by
\begin{align}\label{sigma2def}
\sigma_3:=\min\left\{\sigma_2,
\frac{3}{2}\left(\frac{1}{q}-\frac{1}{r}\right)\right\}.
\end{align}
In particular, if \eqref{drest} holds with some $\sigma_1>1$, then there holds
\begin{align}\label{gradlqlrinfy}
\|\nabla e^{-t\CL_{\Omega,\eta,\omega}}f\|_{r,\BR^3\setminus B_R(0)}\leq
\begin{cases}
Ct^{-\frac{3}{2}(\frac{1}{q}-\frac{1}{r})-\frac{1}{2}}
\|f\|_{q,\Omega}
\quad &\text{if \enskip $q\leq r\leq 3$},\\
Ct^{-\frac{3}{2q}}\|f\|_{q,\Omega}
&\text{if \enskip $\dfrac{3}{2\sigma_1} \leq q$ and $3<r\leq \infty$},
\end{cases}
\end{align}
for $t\geq 2$ and $f\in L^q_\sigma(\Omega)$ provided that
$1 \slash q - 1 \slash r < 1 \slash 3$.
\end{lem}

\begin{proof}
Let $r\leq 3$. In this case, there holds $|1/r-1/2|<1/6+\varepsilon_2$.
Lemma \ref{proplqlr} yields
\begin{align*}
\|e^{-t\CL_{\Omega,\eta,\omega}}f\|_{r,\BR^3\setminus B_R(0)}\leq
\|e^{-t\CL_{\Omega,\eta,\omega}}f\|_{r,\Omega}\leq
Ct^{-\frac{3}{2}(\frac{1}{q}-\frac{1}{r})}\|f\|_{q,\Omega}\leq
Ct^{-\sigma_3}\|f\|_{q,\Omega},
\end{align*}
from which \eqref{lqlrinfy} follows for $r\leq 3$. \par
Next, let $r>3$.
Let $\phi\in C_0^\infty(B_R(0))$ be as in Lemma \ref{lemsigma1}.
For $f\in C_{0,\sigma}^\infty(\Omega)$, let
$p(t)$ be the pressure associated with $e^{-t\CL_{\Omega,\eta,\omega}}f$
subject to $\int_{\Omega_R}p(t)\d x=0$.
Set
\begin{align*}
v(t)=(1-\phi)e^{-t\CL_{\Omega,\eta,\omega}}f
+\BB_{A_{R_0,R}}[e^{-t\CL_{\Omega,\eta,\omega}}f\cdot\nabla \phi],
\qquad p_v(t)=(1-\phi)p.
\end{align*}
It suffices to deduce the estimate of $\|v(t)\|_{r,\Omega}$
by using the integral equation
\begin{align}\label{integraleq}
v(t)=T_{\BR^3,\eta,\omega}(t)f_0+\int_0^t T_{\BR^3,\eta,\omega}(t-\tau)
P_{\BR^3}F_3(\tau)\d \tau,
\end{align}
where $f_0$ is defined by \eqref{f0def}$_2$ and $F_3$ is given by
\begin{align*}
F_3 (x,t)=&\,2\nabla\phi
\cdot\nabla e^{-t\CL_{\Omega,\eta,\omega}}f
-[\Delta\phi+(\eta + \omega\times x)\cdot\nabla \phi]
e^{-t\CL_{\Omega,\eta,\omega}}f
-\Delta\BB_{A_{R_0,R}}[e^{-t\CL_{\Omega,\eta,\omega}}f\cdot\nabla\phi]
\\
&-(\eta + \omega\times x)\cdot
\nabla \BB_{A_{R_0,R}}[e^{-t\CL_{\Omega,\eta,\omega}}f\cdot\nabla \phi]
+\omega\times \BB_{A_{R_0,R}}[e^{-t\CL_{\Omega,\eta,\omega}}f\cdot\nabla \phi]\\
&+\BB_{A_{R_0,R}}[\partial_te^{-t\CL_{\Omega,\eta,\omega}}f
\cdot\nabla \phi]-(\nabla\phi)p.
\end{align*}
Since $\mathrm{supp}\, F_3 \subset B_R(0)$,
we may employ \eqref{partialtest2}, \eqref{pest2},
\eqref{lqlrsmooth2}, Proposition \ref{prop-Bogovskii},  
Proposition \ref{prop-Bogovskii2} and Lemma \ref{lemsigma1} 
to obtain
\begin{align}
\|F_3(t)\|_{s,\BR^3}\leq Ct^{-1+\alpha}(1+t)^{1-\alpha-\sigma_2}
\|f\|_{q,\Omega}
\end{align}
for $t>0$ and $s\in(1,q]$. Hence, it follows that
\begin{align*}
\|T_{\BR^3,\eta,\omega}(t)f_0\|_{r,\BR^3}
& \leq Ct^{-\frac{3}{2}(\frac{1}{q}-\frac{1}{r})}\|f\|_{q,\Omega},\\
\|T_{\BR^3,\eta,\omega}(t-\tau)P_{\BR^3}F_3(\tau)\|_{r,\BR^3}
& \leq C G_2 (t, \tau)
\|f\|_{q,\Omega}
\end{align*}
for $t>\tau$. Here, we have set
\begin{equation}
G_2 (t,\tau) := (t-\tau)^{-\frac{3}{2}(\frac{1}{q}-\frac{1}{r})}
(1+t-\tau)^{-\frac{3}{2}(\frac{1}{s_0}-\frac{1}{q})}
\tau^{-1+\alpha}(1+\tau)^{1-\alpha-\sigma_2}.
\end{equation}
with $s_0$ (close to $1$)
such that $1<s_0<\min\{3/2,q\}$ and that $1/s_0-1/r>2/3$.
Then it turns out that
\begin{align*}
\|v(t)\|_{r,\Omega}&\leq
Ct^{-\frac{3}{2}(\frac{1}{q}-\frac{1}{r})}\|f\|_{q,\Omega}
+\int_0^t G_2 (t,\tau) \d \tau
\|f\|_{q,\Omega}.
\end{align*}
By the condition on $s_0$ and by $1/q-1/r<2/3$, we carry out
the same calculation as in \eqref{int00} to
conclude \eqref{lqlrinfy}.
To clarify the definition of $\sigma_3$,
we give the detailed calculation of the integral terms:
\begin{equation}
\label{integral3}
\begin{split}
\int_0^1G_2 (t,\tau) \d \tau&\leq
Ct^{-\frac{3}{2}(\frac{1}{s_0}-\frac{1}{r})}
\int_0^1\tau^{-1+\alpha}(1+\tau)^{1-\alpha-\sigma_2}\d \tau\leq
Ct^{-\frac{3}{2}(\frac{1}{q}-\frac{1}{r})}, \\
\int_1^{\frac{t}{2}}G_2 (t,\tau) \d \tau&
\leq Ct^{-\frac{3}{2}(\frac{1}{s_0}-\frac{1}{r})}
\int_1^{\frac{t}{2}}\tau^{-\sigma_2}\d \tau
\leq
C t^{- \frac32 (\frac1{s_0} - \frac1r)} g_2 (\sigma_2, t) \leq Ct^{-\sigma_3},\\
\int_{\frac{t}{2}}^{t-1}G_2 (t,\tau) \d \tau&
\leq Ct^{-\sigma_2}\int_{\frac{t}{2}}^{t-1}
(t-\tau)^{-\frac{3}{2}(\frac{1}{s_0}-\frac{1}{r})}\d \tau
\leq Ct^{-\sigma_2}, \\
\int_{t-1}^t G_2 (t,\tau) \d \tau& \leq Ct^{-\sigma_2}
\int_{t-1}^t(t-\tau)^{-\frac{3}{2}(\frac{1}{q}-\frac{1}{r})}\d \tau
\leq Ct^{-\sigma_2}.
\end{split}
\end{equation}
This concludes \eqref{lqlrinfy}.
\par To derive \eqref{gradlqlrinfy} under the assumption
$\sigma_1>1$, we use \eqref{integraleq}. Consider
the integral equation
\begin{align*}
\nabla v(t)=\nabla T_{\BR^3,\eta,\omega}(t)f_0
+\int_0^t \nabla T_{\BR^3,\eta,\omega}(t-\tau)P_{\BR^3}F_3(\tau)\d \tau.
\end{align*}
We take $s_1$ so that $1<s_1<\min\{3/2,q\}$.
We also suppose $1/s_1-1/r>1/3$ if $r>3/2$.
Then there holds 
\begin{equation}
\|\nabla v(t)\|_{r,\Omega}\leq
Ct^{-\frac{3}{2}(\frac{1}{q}-\frac{1}{r})-\frac{1}{2}}
\|f\|_{q,\Omega}
+\int_0^t G_3 (t, \tau) \d \tau
\|f\|_{q,\Omega}. \label{nablav}
\end{equation}
with
\begin{equation}
G_3 (t, \tau) := (t-\tau)^{-\frac{3}{2}(\frac{1}{q}-\frac{1}{r})-\frac{1}{2}}
(1+t-\tau)^{-\frac{3}{2}(\frac{1}{s_1}-\frac{1}{q})}
\tau^{-1+\alpha}(1+\tau)^{1-\alpha-\sigma_2}
\end{equation}
Similarly to \eqref{integral3}$_{1,2}$, there holds
\begin{equation}
\int_0^{t\slash 2} G_3 (t,\tau) \d \tau
= \bigg(\int_0^1 + \int_1^{t \slash 2} \bigg) G_3 (t, \tau) \d \tau
\le C t^{- \frac32 (\frac1q - \frac1r) - \frac12},
\end{equation}
where we have used $\sigma_2=3/(2q)$
in the case $\sigma_2<1$ due to $\sigma_1>1$.
We also deduce
\begin{equation}
\begin{split}
\label{int50}
\int_{\frac{t}{2}}^{t-1} G_3 (t,\tau) \d \tau &\leq Ct^{-\sigma_2}
\begin{dcases}
\displaystyle\int_{\frac{t}{2}}^{t-1}
(t-\tau)^{-\frac{3}{2}(\frac{1}{q}-\frac{1}{r})-\frac{1}{2}}\d \tau
\quad
&\text{if}~r\leq \displaystyle\frac{3}{2},\\
\displaystyle\int_{\frac{t}{2}}^{t-1}
(t-\tau)^{-\frac{3}{2}(\frac{1}{s_1}-\frac{1}{r})-\frac{1}{2}}
\d \tau
\quad
&\text{if}~r>\displaystyle\frac{3}{2},
\end{dcases} \\
&\leq
\begin{dcases}
Ct^{-\frac{3}{2}(\frac{1}{q}-\frac{1}{r})+\frac{1}{2}-\sigma_2}
\quad
&\text{if}~r\leq \displaystyle\frac{3}{2},\\
Ct^{-\sigma_2}
\quad
&\text{if}~r>\displaystyle\frac{3}{2},
\end{dcases}
\end{split}
\end{equation}
where we have used $1/s_1-1/r>1/3$ in the second inequality
if $r>3/2$. 
If $r\leq 3/2$, from $\sigma_1>1$ 
and $\sigma_2=\{\sigma_1,3/(2q)\}$, the last term in 
\eqref{int50} is bounded by $t^{-3(1/q-1/r)/2-1/2}$. If 
$r>3/2$ and $\sigma_1\ge 3/(2q)$, then we also find that 
the last term in \eqref{int50} is 
bounded by $t^{-\min\{3(1/q-1/r)/2+1/2,3/(2q)\}}$. 
If $3/2<r\leq 3$ and $\sigma_1\leq 3/(2q)$ 
i.e. $\sigma_1=\sigma_2$, then 
we impose the additional condition $1/q-1/r<(2\sigma_1-1)/3$ 
to find that the last term in \eqref{int50} is 
bounded by $t^{-3(1/q-1/r)/2-1/2}$. We thus conclude
\begin{equation}
\label{g3est2}
\int_{\frac{t}{2}}^{t-1} G_3 (t,\tau) \d \tau
\leq \begin{cases}
Ct^{-\frac{3}{2}(\frac{1}{q}-\frac{1}{r})-\frac{1}{2}}
\|f\|_{q,\Omega}
\quad &\text{if \enskip $ q\leq r\leq 3$ and
$\dfrac{1}{q}-\dfrac{1}{r}<\dfrac{2\sigma_1-1}{3}$}, \\
Ct^{-\frac{3}{2q}}\|f\|_{q,\Omega}
&\text{if \enskip $\dfrac{3}{2\sigma_1}\leq q$ and $3<r\leq \infty$}.
\end{cases}
\end{equation}
Under the condition $1/q-1/r<1/3$,
the term $\int_{t-1}^t G_2 (t,\tau) \d \tau$ may
be bounded from above by $Ct^{-\sigma_2}$.
In the case $\sigma_2=3/(2q)$,
we see $-\sigma_2\leq -\min\{3(1/q-1/r)/2+1/2,3/(2q)\}$, whereas
in the other case $\sigma_2=\sigma_1$, we have
$-\sigma_2\leq -3(1/q-1/r)/2-1/2$ if $1/q-1/r<1/3$ is fulfilled.
This observation together with \eqref{nablav}--\eqref{g3est2}
asserts \eqref{gradlqlrinfy} for $t\geq 2$ and
$f\in L^q_\sigma (\Omega)$ provided that $1/q-1/r<1/3$.
\end{proof}

\begin{rmk}
If $\sigma_1=1$, then we may not deduce
\begin{align*}
\|\nabla e^{-t\CL_{\Omega,\eta,\omega}}f\|_{r,\BR^3\setminus B_R(0)}\leq
Ct^{-\frac{3}{2}(\frac{1}{q}-\frac{1}{r})-\frac{1}{2}}
\|f\|_{q,\Omega}
\end{align*}
with $r=3$ by
the calculation as in the proof of Lemma \ref{lemsigma2}.
In fact, from \eqref{int50}, we need to prove
\begin{equation}
\int_{t/2}^{t-1} G_2 (t, \tau) \d \tau \leq Ct^{-\sigma_2}\|f\|_{q,\Omega}
\leq Ct^{-\frac32(\frac1q - \frac1r)-\frac12}\|f\|_{q,\Omega}
\end{equation}
if $r>3/2$. However, by $\sigma_1=1$, we know
$\sigma_2=\min\{1,3/(2q)-\epsilon\}$ 
with arbitrarily small $\epsilon>0$, and thus
$r\leq 3/(1+2\epsilon)$ is imposed.
\end{rmk}

The following proposition is a direct consequence of
Lemma \ref{lemsigma1} and Lemma \ref{lemsigma2}.
\begin{prop}\label{propr0}
Let $\varepsilon_2$ and $q$ be subject to Convention \ref{conv-2}
and $\eta, \omega \in \BR^3$. 
Given $c_0>0,$ assume $\lvert \eta \rvert + \lvert \omega \rvert \leq c_0.$
Let $R\in(R_0+1,\infty)$.
Assume that \eqref{drest} holds with some $\sigma_1=1/2+\kappa > 1 \slash 2$
for $t\geq 2$ and $f\in L^q_\sigma(\Omega)$ with
$f=0$ a.e. $\BR^3\setminus B_{R_0}(0)$.
Set
\begin{align*}
r_0:=
\begin{cases}
\displaystyle\frac{3}{1-2\kappa+2\epsilon}
\quad &{\rm if}~\kappa\leq\displaystyle\frac{1}{2},\\[12pt]
\infty&{\rm if}~\kappa>\displaystyle\frac{1}{2}
\end{cases}
\end{align*}
with arbitrarily small $\epsilon>0$.
Then the estimate \eqref{lplqomega} holds for
$t>0$ and $f\in L^q_\sigma(\Omega)$
provided that $q\leq r\leq\max\{r_0,3/(1-3\varepsilon_2)\}$.
In particular, if \eqref{drest} holds
with some $\sigma_1=1/2+\kappa > 1$ for
$t\geq 2$ and $f\in L^q_\sigma(\Omega)$ with
$f=0$ a.e. $\BR^3\setminus B_{R_0}(0)$, then there holds
\begin{align}\label{gradlqlq}
\|\nabla e^{-t\CL_{\Omega,\eta,\omega}}f\|_{q,\Omega}\leq
Ct^{-\min\{\frac{1}{2},\frac{3}{2q}\}}\|f\|_{q,\Omega}
\end{align}
for $t\geq 2$ and $f\in L^q_\sigma(\Omega)$.
\end{prop}

\begin{proof}
In light of \eqref{sigma1def}, we know
\begin{align*}
\sigma_2=
\begin{cases}
\min\left\{\sigma_1,\displaystyle\frac{3}{2q}+\sigma_1-1
-\epsilon\right\}
\quad &{\rm if}~\kappa\leq\displaystyle\frac{1}{2},\\[12pt]
\min\left\{\sigma_1,\displaystyle\frac{3}{2q}\right\}
&{\rm if}~\kappa>\displaystyle\frac{1}{2}.
\end{cases}
\end{align*}
In the case $\kappa\leq 1/2$, we suppose
$q\geq 3/(2+2\epsilon)$
and $r\leq r_0=3/(1-2\kappa+2\epsilon)$, then there holds
\begin{equation}
\sigma_2=\frac{3}{2q}+\sigma_1-1
-\epsilon\geq
\frac32\biggl(\frac1q-\frac1r\biggr)=\sigma_3.
\end{equation}
Hence, by virtue of the embedding $W^{1,q}(\Omega_R)\hookrightarrow L^r(\Omega_R)$
provided $1/q-1/r<1/3$, we infer from Lemma~\ref{lemsigma1} and the estimate \eqref{lqlrinfy}
that \eqref{lplqomega} is valid
if there hold $\eqref{p'},
3/(2+2\epsilon)\leq q\leq r\leq r_0$, and $1/q-1/r<1/3$.
Since the conditions $3/(2+2\epsilon)\leq q$ and
$1/q-1/r<1/3$ may be eliminated by using
the semigroup law and Lemma \ref{proplqlr},
we have \eqref{lplqomega} for $t>0$ and $f\in L^q_\sigma(\Omega)$
provided that $q\leq r\leq\max\{r_0,3/(1-3\varepsilon_2)\}$. \par
We may deal with the case $\kappa>1/2$ by the same procedure as above.
Namely, we conclude
\eqref{lplqomega} for $t>0$ and $f\in L^q_\sigma(\Omega)$
provided that $q\leq r\leq\infty$.
In addition, in the case $\kappa>1/2$, we have $\sigma_1>1$ and
$\sigma_2=\min\{\sigma_1,3/(2q)\}\geq \min\{1/2,3/(2q)\}$, and hence
we may employ Lemma \ref{lemsigma1} and
\eqref{gradlqlrinfy} with $q=r$ to obtain \eqref{gradlqlq}
for $t\geq 2$ and $f\in L^q_\sigma(\Omega)$ provided that
\eqref{p'}. Here, notice that
the condition $q\geq 3/(2\sigma_1)$ in \eqref{gradlqlrinfy}
is automatically fulfilled if $q\geq 3$. The proof is complete.
\end{proof}

Let us close this section with the completion of the proofs of
Theorems \ref{thmlqlr} and \ref{thmlqlr2}.

\begin{proof}[Proof of Theorem \ref{thmlqlr}]
We already know \eqref{lqlr} and \eqref{gradlqlr} for $t\leq 2$
as follows from Theorem \ref{thm1}. For the decay estimates, by virtue of
Lemma \ref{proplqlr}, we have \eqref{lqlr} provided that
\eqref{p'} and \eqref{cond-r} are valid.
We proceed to prove \eqref{lqlr} for $r\leq\infty$.
Let $\delta_0>0$ be an arbitrarily small constant.
Given $q$ fulfilling \eqref{p'},
let us take $p_0$ and $q_0$ so that
\begin{align*}
p_0<q<q_0,\quad
\left|\frac{1}{p_0}-\frac{1}{2}\right|
<\frac{1}{6}+\varepsilon_2,\quad
\left|\frac{1}{q_0}-\frac{1}{2}\right|
<\frac{1}{6}+\varepsilon_2,\quad
-\frac{3}{2}\left(\frac{1}{p_0}-\frac{1}{q_0}\right)
=-\frac{1}{2}-3\varepsilon_2+\delta_0
\end{align*}
Then we employ Lemma \ref{proplqlr}
and Lemma \ref{lemlrlr} to deduce
\begin{equation}
\label{lp0lq0}
\begin{split}
\|\pd_t^\ell e^{-t\CL_{\Omega,\eta,\omega}}f\|_{W^{1 - 2\ell,q}(\Omega_R)}
& \leq C
\|\pd_t^\ell e^{-t\CL_{\Omega,\eta,\omega}}f\|_{W^{1 - 2 \ell,q_0}(\Omega_R)} \\
& \leq C
\|e^{-(t-1)\CL_{\Omega,\eta,\omega}}f\|_{q_0,\Omega} \\
&\leq Ct^{-\frac{1}{2}-3\varepsilon_2+\delta_0}
\|f\|_{p_0,\Omega} \\
&\leq Ct^{-\frac{1}{2}-3\varepsilon_2+\delta_0}
\|f\|_{q,\Omega}
\end{split}
\end{equation}
for $\ell \in \{0, 1\}$, $t\geq 2$, and $f\in L^q_\sigma(\Omega)$
with $f=0$ a.e. $\BR^3\setminus B_{R_0}(0)$.
If $3\varepsilon_2\leq 1/2$, then Proposition \ref{propr0} with
$\kappa=3\varepsilon_2-\delta_0 < 1/2$ implies
\eqref{lqlr} for $t>0$ provided that 
$q$ and $r$ satisfy \eqref{p'} and
\begin{align*}
r\leq \frac{3}{1-6\varepsilon_2+2\delta_0+2\epsilon}
\end{align*}
with arbitrarily small $\epsilon>0$, respectively. Notice that
\begin{align}
\frac{3}{1-6\varepsilon_2+2\delta_0+2\epsilon}
>\frac{3}{1-3\varepsilon_2}.
\end{align}
Hence, the same argument as in
\eqref{lp0lq0} yields
\begin{align*}
\|\pd_t^\ell e^{-t\CL_{\Omega,\eta,\omega}}f\|_{W^{1-2\ell,q}(\Omega_R)}
& \leq C \|\pd_t^\ell e^{-t\CL_{\Omega,\eta,\omega}}f\|_{W^{1-2\ell,q_1}(\Omega_R)} \\
& \leq Ct^{-\frac{1}{2}-\frac{9}{2}\varepsilon_2
+\delta_0+\delta_1+\epsilon}\|f\|_{p_1,\Omega} \\
& \leq Ct^{-\frac{1}{2}-\frac{9}{2}\varepsilon_2
+\delta_0+\delta_1+\epsilon}\|f\|_{q,\Omega},
\end{align*}
for $\ell \in \{0,1 \}$, $t\geq 2$, and $f\in L^q_\sigma(\Omega)$
with $f=0$ a.e. $\BR^3\setminus B_{R_0}(0)$,
where $p_1$ and $q_1$ satisfy
\begin{gather}
p_1<q<q_1\leq\frac{3}{1-6\varepsilon_2+2\delta_0+2\epsilon},\quad
\left|\frac{1}{p_1}-\frac{1}{2}\right|
<\frac{1}{6}+\varepsilon_2,\\
-\frac{3}{2}\left(\frac{1}{p_1}-\frac{1}{q_1}\right)
=-\frac{1}{2}-\frac{9}{2}\varepsilon_2
+\delta_0+\delta_1+\epsilon
\end{gather}
with arbitrarily small $\delta_1>0$.
If $(9\varepsilon_2)/2 \leq 1/2$,
then Proposition \ref{propr0} with
$\kappa=(9\varepsilon_2)/2-\delta_0-\delta_1
-\epsilon < 1/2$
implies
\eqref{lqlr} for $t>0$ provided that 
$q$ and $r$ satisfy \eqref{p'} and
\begin{align*}
r\leq \frac{3}{1-9\varepsilon_2+2\delta_0+2\delta_1
+4\epsilon},
\end{align*}
respectively. Notice that
\begin{align*}
\frac{3}{1-9\varepsilon_2+2\delta_0+2\delta_1
+4\epsilon}
>\frac{3}{1-6\varepsilon_2+2\delta_0+2\epsilon}.
\end{align*}
Set $k_0:=\max\{k\in\BN \cup \{0\} \mid (3+(3k)/2)
\varepsilon_2 \leq 1/2\}$.
Then, carrying out the aforementioned argument $(k_0+1)$times
asserts
\begin{align*}
\|e^{-t\CL_{\Omega,\eta,\omega}}f\|_{W^{1,q}(\Omega_R)}+
\|\partial_te^{-t\CL_{\Omega,\eta,\omega}}f\|_{W^{-1,q}(\Omega_R)}
\leq
Ct^{-\frac{1}{2}-\left\{3+\frac{3}{2}(k_0+1)\right\}\varepsilon_2
+\delta}\|f\|_{q,\Omega}
\end{align*}
for $t\geq 2$ and $f\in L^q_\sigma(\Omega)$
with $f=0$ a.e. $\BR^3\setminus B_{R_0}(0)$, where
$\delta>0$ is an arbitrarily small constant.
This estimate allows us to employ Proposition \ref{propr0} with
$\kappa=\{3+3(k_0+1)/2\}\varepsilon_2-\delta>1/2$.
Thus, we end up with \eqref{lqlr}
for $t>0$ and $f\in L^q_\sigma(\Omega)$
provided that $q$ fulfills \eqref{p'} and $q\leq r\leq\infty$.
We may also obtain \eqref{gradlqlr2}
by using \eqref{gradlqlq} with $q=r$ and the semigroup law.
The proof is complete.
\end{proof}

\begin{proof}[Proof of Theorem \ref{thmlqlr2}]
The estimates \eqref{lqlr01} and \eqref{gradlqlr01} follow from 
real interpolation. 
As was discussed
in Hishida and Shibata \cite{HS09} and Hishida \cite{H20}, 
to deduce 
\eqref{gradlqlr02}, in particular \eqref{gradlqlr02} with $r=3$, 
we fix some constant $\sigma_1>1$ and 
perform real interpolation in \eqref{partialtest3},
\eqref{pest2}, and \eqref{local}. Then given 
$q \ne 3/(2\sigma_1)$ satisfying \eqref{p'}
and $1\leq\rho\leq \infty$, there holds
\begin{align}
\label{local2}
\|\nabla^je^{-t\CL_{\Omega,\eta,\omega}}f\|_{q, \rho, \Omega_R}&\leq
Ct^{-\frac{j}{2}}(1+t)^{\frac{j}{2}-\sigma_2}\|f\|_{q,\Omega},\\
\|\BB_{A_{R_0,R}}[\partial_te^{-t\CL_{\Omega,\eta,\omega}}
f\cdot\nabla\phi]\|_{q, \rho, \Omega_R}
+\|p(t)\|_{q, \rho, \Omega_R}& \leq
Ct^{-1+\alpha} (1+t)^{1-\alpha-\sigma_2}\|f\|_{q,\Omega}
\end{align}
for $j=0,1$, $t>0$, and $f\in L^{q,\rho}_\sigma(\Omega)$,
where $\alpha=\alpha(q)\in(0,1)$ and
$\sigma_2=\min\left\{\sigma_1,3/(2q)\right\}$,
see \eqref{sigma1def}.
By these estimates, $L^{q,\rho}$-$L^{r,\rho}$ estimates of
$\nabla T_{\BR^3,\eta,\omega}(t)$ and 
the estimate of the Bogovski\u{\i} operator in the Lorentz space,
which follows from Proposition \ref{prop-Bogovskii}
and real interpolation, carrying out the same procedure as in the
proof of Lemma \ref{lemsigma2} yields \eqref{gradlqlrinfy} 
with the Lebesgue norm replaced by the Lorentz norm.
Since the continuity is needed to conclude \eqref{gradlqlrinfy},
the case $\rho=\infty$ is missing in \eqref{gradlqlr02}.
The estimate \eqref{duality} follows from the same argument as
in the proof of Yamazaki \cite[Cor. 2.3]{Y00}. 
\end{proof}

\section{Nonlinear problem}
\label{sec-7}
In this section, we apply Theorems \ref{thmlqlr} 
and \ref{thmlqlr2} to 
the Navier--Stokes initial value problem \eqref{eq-main}.
In particular, this section aims to prove Theorem \ref{thmnonlinear} with the
aid of Theorems \ref{thmlqlr} and \ref{thmlqlr2}. 
Before the proof, let us introduce several symbols 
to reformulate the problem. Set 
\begin{align}
f(x)=(f_1(x),f_2(x),f_3(x))^\top=\frac12 \Bigl\{\zeta(x) \Bigl(\eta \times x - |x|^2\omega \Bigr)\Bigr\}.
\end{align}
Since $b (x)$ is given by \eqref{def-b}, we see that
\begin{align}
\Delta b (x)&= \div
\begin{pmatrix}
0 & -\Delta f_3 & \Delta f_2\\
\Delta f_3 & 0 & -\Delta f_1\\
-\Delta f_2 & \Delta f_1 & 0
\end{pmatrix}, \\ 
\omega \times b & = \div
\begin{pmatrix}
\omega_2f_2+\omega_3f_3&-\omega_1f_2&-\omega_1f_3\\
-\omega_2f_1&\omega_1f_1+\omega_3f_3&-\omega_2f_3\\
-\omega_3f_1&-\omega_3f_2&\omega_1f_1+\omega_2f_2
\end{pmatrix}.
\end{align}
Hence, we may write $\CN_2(b)$, which is given by \eqref{def-n2},
as $\CN_2(b)=\div \CF (b)$. 
Here, we have used the notation introduced in \eqref{def-F}. 
In order to prove Theorem \ref{thmnonlinear},
we first consider the integral equation in the weak form: 
\begin{equation}
\label{eq-dual}
\begin{split}
(u(t),\varphi)& =(v_0-b,e^{-t\CL_{\Omega,- \eta, - \omega}}\varphi) \\
& \quad + \int_0^t \Big(u(s)\otimes u(s)+u(s)\otimes b
+b\otimes u(s)-\CF (b),
\nabla e^{-(t-s)\CL_{\Omega,- \eta, - \omega}}\varphi \Big) \d s
\end{split}
\end{equation}
for all $\varphi \in C_{0,\sigma}^\infty(\Omega)$. 
\par Since $\CF (b) \in L^{3/2,\infty}(\Omega)$,
by following Yamazaki \cite{Y00} (cf. Hishida and Shibata \cite{HS09}),
we may construct a unique global solution to \eqref{eq-dual}
provided that $\lVert u_0-b\rVert_{3,\infty,\Omega}$, $|\omega|$, and $\lvert \eta \rvert$ are sufficiently small.
Notice that $\CF (b)$ do not have a temporal decay property, 
for $q > 3$ we may not expect 
the $L^{q,\infty}$-decay of the solution $u$ 
unless $b$ is a stationary solution to \eqref{eq-main},
but for each $t>0$
we still find $u(t)\in L^{q,\infty}(\Omega)$.
Moreover, for later use, we also deduce some estimates, which
are uniformly bounded in a neighborhood of each time $t>0$.
To describe this issue more precisely,
we prepare the following lemma.
The proof is same as in Yamazaki \cite[Sec. 3]{Y00}
(see also Hishida and Shibata \cite[Sec. 8]{HS09}), 
and hence we may omit the proof.
Here and hereafter, we shall use the notation
\begin{equation}
\label{subnorm}
[u]_{r,t_0,t_1}:=\sup_{t_0<s<t_1}
(s-t_0)^{\frac{1}{2}-\frac{3}{2r}}
\lVert u(s)\rVert_{r,\infty,\Omega},
\qquad r\geq 3,\enskip t_0<t_1\leq\infty.
\end{equation}

\begin{lem}
\label{lem-est}
Let $\varepsilon_2$ be subject to Convention \ref{conv-2}
and $\eta,\omega \in \BR^3$. Given $c_0 > 0$, assume $\lvert \eta \rvert + \lvert \omega \rvert \le c_0$. 
Let $0\leq t_0<t_1\leq\infty$.
Set
\begin{align*}
\braket{\CI(u,v)(t_0;t),\varphi} & :=
\int_{t_0}^t \Big(u(\tau)\otimes v(\tau),
\nabla e^{-(t-\tau)\CL_{\Omega,- \eta, - \omega}}\varphi \Big)\d\tau,\\
\braket{\CJ(u)(t_0;t),\varphi} & :=
\int_{t_0}^t \Big(u(\tau)\otimes b+b\otimes u(\tau),
\nabla e^{-(t-\tau)\CL_{\Omega,- \eta, - \omega}}\varphi \Big)\d\tau,\\
\braket{\CK(t_0;t),\varphi} & :=
\int_{t_0}^t \Big(\CF (b),\nabla e^{-(t-\tau)\CL_{\Omega,- \eta, - \omega}}\varphi \Big) \d\tau,
\end{align*}
where $b$ and $\CF (b)$ are the functions defined by \eqref{def-b}
and \eqref{def-F}, respectively.
Then the following assertions are valid.
\begin{enumerate}
\item There exists a constant $C$ independent of $t_0$ and $t_1$
such that
\begin{align*}
\lVert\CI(u,v)(t_0;t)\rVert_{3,\infty}
& \le C[u]_{3,t_0,t}[v]_{3,t_0,t},\\
\lVert\CJ(u)(t_0;t)\rVert_{3,\infty}
& \le C[u]_{3,t_0,t}\lVert b\rVert_{3,\infty},\\
\lVert\CK(t_0;t)\rVert_{3,\infty}
&\le C\lVert \CF (b) \rVert_{\frac{3}{2},\infty}
\end{align*}
for every measurable functions $u$ and $v$ subject to
$[u]_{3,t_0,t_1} < \infty$ and $[v]_{3,t_0,t_1}<\infty$, respectively, and for every $t_0<t<t_1$.
\item Let $q$ satisfy $1/3-\varepsilon_2<1\slash q<1/3$.
Then there exists a constant $C=C(q)$ independent of $t_0$ and $t_1$
such that
\begin{align*}
\lVert\CI(u,v)(t_0;t)\rVert_{q,\infty}
& \le C(t-t_0)^{-\frac{1}{2}+\frac{3}{2q}}
[u]_{q,t_0,t}[v]_{3,t_0,t},\\
\lVert\CJ(u)(t_0;t)\rVert_{q,\infty}
& \le C(t-t_0)^{-\frac{1}{2}+\frac{3}{2q}}[u]_{q,t_0,t}
\lVert b\rVert_{3,\infty},\\
\lVert\CK(t_0;t)\rVert_{q,\infty}
& \le C(t-t_0)^{\frac{1}{2}}\lVert \CF (b) \rVert_{\frac{3}{2},\infty}
\end{align*}
for every measurable functions $u$ and $v$ subject to
$[u]_{q,t_0,t_1} < \infty$ and $[v]_{3,t_0,t_1}<\infty$, respectively, and for every $t_0<t<t_1$.
\end{enumerate}
\end{lem}

Using Lemma \ref{lem-est}, we may construct a global solution to Problem \eqref{eq-dual}.

\begin{prop}\label{prop-sol}
Let $\varepsilon_2$ be subject to Convention \ref{conv-2} and $\eta,\omega \in \BR^3$.
The following assertions hold.
\begin{enumerate}
\item There exist constants $\kappa_1,\kappa_2>0$ such that if
$\lvert\eta \rvert + \lvert \omega \rvert <\kappa_1$ and
$v_0 - b \in L^{3,\infty}_\sigma (\Omega)$ satisfies
$\lVert v_0 - b\rVert_{3,\infty,\Omega}<\kappa_2$,
then Problem \eqref{eq-dual}
admits a global solution with the following properties:
\begin{enumerate}
\item $u\in BC((0,\infty);L^{3,\infty}_\sigma(\Omega)).$
\item $u(t)\rightarrow v_0-b$
weakly $*$ in $L^{3,\infty}(\Omega)$ as $t\searrow +0$.
\item There is a constant $C>0$ such that
\begin{equation}
\label{apriori}
\lVert u(t)\rVert_{3,\infty,\Omega}\le C \bigl(\lvert \eta \rvert
+ \lvert \omega \rvert +\lVert v_0 - b\rVert_{3,\infty,\Omega}\bigr)
\end{equation}
for $t>0$.
\item The solution $u$ is unique among solutions with small
$\sup_{t>0}\|u(t)\|_{3,\infty}$.
\end{enumerate}
\item Let $\kappa_1, \kappa_2 > 0$ be the same numbers as in the first assertion.
For every $q$ satisfying $1/3-\varepsilon_2<1/q<1/3$, there
exist constants
$c_1=c_1(q) \le \kappa_1$ and $c_2 = c_2(q) \le \kappa_2$ such that if
$\lvert\eta \rvert + \lvert \omega \rvert < c_1$ and
$v_0 - b \in L^{3,\infty}_\sigma (\Omega)$ satisfies
$\lVert v_0 - b\rVert_{3,\infty,\Omega}< c_2$,
then the solution $u$ obtained in the first assertion
also satisfies
$v\in C((0,\infty);L^{q,\infty}(\Omega))$.
Furthermore, there exist constants $C,\CT>0$ such that
\begin{align}
\label{est-qinfty}
\|u(t)\|_{q,\infty,\Omega}\leq
C\Bigl\{\bigl(\lvert\eta\rvert + \lvert \omega \rvert \bigr) 
\CT^{1-\frac{3}{2q}}+\lvert\eta\rvert+ \lvert \omega \rvert + 
\lVert v_0 - b\rVert_{3,\infty,\Omega} \Bigr\}
(t-k\CT)^{-\frac{1}{2}+\frac{3}{2q}}
\end{align}
for $k \in \BN \cup\{0\}$ and for $k\CT<t\leq (k+1)\CT$.
\end{enumerate}
\end{prop}
\begin{proof}
The proof is essentially the same as in
Yamazaki \cite{Y00} and Hishida and Shibata \cite[Sec. 8]{HS09},
and thus we may omit the proof of the first assertion. However, 
to reveal the estimate \eqref{est-qinfty}
and a constant $\CT$,
we intend to give the proof of the second assertion.
We define the right-hand side of \eqref{eq-dual}
by $\braket{(\Phi v)(t),\varphi}$.
Let $t_0\geq 0$ and $\CT_0>0$ and set
\begin{align*}
X_{q,t_0,\CT_0}:=\{u\in BC((t_0,\CT_0];L^{3\infty}_\sigma(\Omega))
\mid (t-t_0)^{\frac{1}{2}+\frac{3}{2q}}u(t)\in
BC((t_0,\CT_0];L^{q,\infty}(\Omega))\},
\end{align*}
equipped with norm $\lVert u\rVert_{X_{q,t_0,\CT_0}}:=
[u]_{q,t_0,\CT_0}+[v]_{q,t_0,\CT_0}$,
where $[u]_{r,t_0,\CT_0}$ is the notation defined by \eqref{subnorm}.
Lemma \ref{lem-est} with $t_0=0$ implies
\begin{align*}
\lVert \Phi u\rVert_{X_{q,0,\CT_0}}\leq C_1
\lVert u\rVert_{X_{q,0,\CT_0}}^2+C_2\bigl(\lvert\eta\rvert + \lvert \omega \rvert \bigr) 
\lVert u\rVert_{X_{q,0,\CT_0}}+
C_3\Bigl\{\bigl(\lvert\eta\rvert + \lvert \omega \rvert \bigr) \CT_0^{1-\frac{3}{2q}}
+\lVert v_0-b\rVert_{3,\infty,\Omega}\Bigr\}.
\end{align*}
Hence, if $\lvert\eta\rvert + \lvert \omega \rvert < 1/(2C_2)$,
$\lVert v_0-b\rVert_{3,\infty,\Omega}<1/(32C_1C_2)$,
and
\begin{equation}
\CT_0^{1-\frac{3}{2q}}<\frac{C_2}{16C_1C_3},
\end{equation}
then we deduce
$u(t)\in X_{q,t_0,\CT_0}$ with
\begin{align}
\label{est-Xq}
\lVert u(t)\rVert_{q,\infty,\Omega}\leq
4C_3\Big(\bigl(\lvert\eta\rvert + \lvert \omega \rvert \bigr) \CT_0^{1-\frac{3}{2q}}
+\lVert v_0-b\rVert_{3,\infty}\Big)t^{-\frac{1}{2}+\frac{3}{2q}}
\end{align}
for $0<t\leq \CT_0$. We fix $t_0>0$ arbitrarily. For $t>t_0$,
the equation \eqref{eq-dual} is converted into
\begin{align}
\label{eq-dual2}
(u(t),\varphi)&=(u(t_0),e^{-(t-t_0)\CL_{\Omega,- \eta, - \omega}}\varphi)\\
&\qquad +\int_0^t(u(\tau)\otimes u(\tau)+u(\tau)\otimes b
+b\otimes u(\tau)-\CF (b),
\nabla e^{-(t-\tau)\CL_{\Omega,- \eta, - \omega}}\varphi)\d\tau\\
&=:\braket{(\Phi_{t_0}v)(t),\varphi}.
\end{align}
Lemma \ref{lem-est} implies
\begin{align*}
\lVert \Phi_{t_0} v\rVert_{X_{q,t_0,t_0+\CT_0}}&\le C_1
\lVert v\rVert_{X_{q,t_0,t_0+\CT_0}}^2+C_2\bigl(\lvert\eta\rvert + \lvert \omega \rvert \bigr)
\lVert v\rVert_{X_{q,t_0,t_0+\CT_0}} \\
& \quad +
C_3\Big\{\bigl(\lvert\eta\rvert + \lvert \omega \rvert \bigr) \CT_0^{1-\frac{3}{2q}}
+\lVert u(t_0)\rVert_{3,\infty,\Omega}\Big\},
\end{align*}
which together with \eqref{apriori} yields
\begin{align}
\lVert u(t)\rVert_{q,\infty,\Omega}\leq
4C_3\Big\{\bigl(\lvert\eta\rvert + \lvert \omega \rvert \bigr) \CT_0^{1-\frac{3}{2q}}
+C\lVert v_0-b\rVert_{3,\infty}\Big\}
(t-t_0)^{-\frac{1}{2}+\frac{3}{2q}}
\end{align}
for $t_0<t\leq t_0+\CT_0$ provided that
$\lvert \eta \rvert$, $\lvert \omega \rvert$, and
$\lVert v_0-b\rVert_{3,\infty,\Omega}$ are small enough. 
Notice that these smallness conditions are
independent of $t_0$ since the constants $C_i~(i=1,2,3)$
and $C$ (given by \eqref{apriori}) are independent of $t_0$.
Carrying out the aforementioned calculation
with $t_0=k\CT_0~(k\in\BN)$,
we conclude the second assertion with $\CT=\CT_0$.
\end{proof}

To prove that the global solution $u$ obtained by
Proposition \ref{prop-sol} possesses
further regularity properties and actually satisfies
the integral equation \eqref{eq-integral} for each $t>0$,
we next consider
\begin{equation}
\label{eq-t0}
u(t)=e^{-(t-t_0)\CL_{\Omega,\eta,\omega}}u_0
-\int_{t_0}^t e^{-(t-s)\CL_{\Omega,\eta,\omega}}
P_\Omega\Big[u\cdot \nabla u+\CN_1(b,u)+\CN_2(b)\Big]\d s
\end{equation}
with $u_0\in L^r_\sigma(\Omega)$, where $r$ is assumed to satisfy
$1/3-\varepsilon_2<1/r<1/3$. Here, recall that
$\CN_1(b,u)$ and $\CN_2(b)$ are defined by
\eqref{def-n1} and \eqref{def-n2}, respectively.
For a local solution to \eqref{eq-t0}, we have the following.
\begin{lem}
\label{lem-local}
Let $\varepsilon_2$ be subject to Convention \ref{conv-2} and 
$\eta, \omega \in \BR^3$. Let
$r$ satisfy $1/3-\varepsilon_2<1/r<1/3$.
For $t_0\geq 0$ and $u_0\in L^r_{\sigma}(\Omega)$,
there exists $t_1\in (t_0,t_0+1]$ such that \eqref{eq-t0}
admits a unique local solution $u\in Y_r(t_0,t_1)$, where
\begin{equation}
Y_r(t_0,t_1):=\{v\in C([t_0,t_1];L^r_\sigma(\Omega))\mid
(\, \cdot-t_0)^{\frac{1}{2}}\nabla v(\, \cdot \,)
\in BC((t_0,t_1];L^r(\Omega))\}.
\end{equation}
The local solution $u \in Y_r(t_0,t_1)$ also fulfills
\begin{equation}
\label{local-reg}
u\in C\big((t_0,t_1];L^\kappa(\Omega)\big),\qquad
\nabla u\in C\big((t_0,t_1];L^\gamma(\Omega)\big)
\end{equation}
for every $\kappa\in (r,\infty]$ and $\gamma \in(r,\infty)$.
Moreover, there exists a non-increasing function
$\zeta(\,\cdot\, ):[0,\infty)\rightarrow (0,1)$ such that
the length of the existence interval may be estimated
from below by
\begin{align*}
t_1-t_0\geq \zeta(\lVert u_0\rVert_{r,\Omega}).
\end{align*}
\end{lem}

\begin{proof}
The proof of the existence of a local solution in
$Y_r(t_0,t_1)$ is similar to
Giga and Miyakawa \cite{GM85} as well as
Kozono and Yamazaki \cite[Thm. 4.1]{KY98},
see also \cite[Prop. 4.7]{T22}.
In fact, we may construct a local solution $u(t)$ to \eqref{eq-t0}
in $[t_0,t_1]$ by applying the $L^q$-$L^r$ estimates of the semigroup
$(e^{- t \CL_{\Omega,\eta,\omega}})_{t\ge 0}$
to the right hand side of \eqref{eq-t0} and
by using the contraction mapping
principle to a suitable closed ball in $Y_r(t_0,t_1)$. The property \eqref{local-reg} is also deduced from 
applying the $L^q$-$L^r$ estimates of the semigroup $(e^{- t \CL_{\Omega,\eta,\omega}})_{t\ge 0}$;
Theorem \ref{thmlqlr} asserts
$u(t)\in L^\infty(\Omega)$ and $\nabla u(t)\in L^\gamma (\Omega)$ with
\begin{equation}
\lVert u(t)\rVert_{\infty,\Omega}
\leq C(t-t_0)^{-\frac{3}{2r}}
\Big\{\lVert u_0\rVert_{r,\Omega}+\lVert u\rVert_{Y_r(t_0,t_1)}^2\\
+\lVert u\rVert_{Y_r(t_0,t_1)}(\lVert b\rVert_{r,\Omega}+
\lVert \nabla b\rVert_{r,\Omega})
+\lVert\CN_2(b)\rVert_{r,\Omega}
\Big\}
\end{equation}
and
\begin{multline}
\lVert \nabla u(t)\rVert_{\gamma,\Omega}\leq
C(t-t_0)^{-\frac{3}{2}(\frac{1}{r}-\frac{1}{\gamma})-\frac{1}{2}}
\Big\{\lVert u_0\rVert_{r,\Omega}+
\llbracket u\rrbracket_{\infty,t_0,t_1}\llbracket\nabla u\rrbracket_{r,t_0,t_1}
+\llbracket u\rrbracket_{\infty,t_0,t_1}\lVert b\rVert_{r,\Omega}\\
 +\llbracket\nabla u\rrbracket_{r,t_0,t_1}\lVert \nabla b\rVert_{r,\Omega}+
\lVert\CN_2(b)\rVert_{r,\Omega}
\Big\}
\end{multline}
for every $\gamma\in(r,\infty)$ and $t\in(t_0,t_1]$, where we have set
\begin{align*}
\lVert u\rVert_{Y_r(t_0,t_1)}:=\llbracket u\rrbracket_{r,t_0,t_1}+
\llbracket\nabla u\rrbracket_{r,t_0,t_1},\quad
\llbracket\nabla^k u\rrbracket_{s,t_0,t_1}:=\sup_{t_0<\tau<t_1}
(\tau-t_0)^{\frac{3}{2}(\frac{1}{r}-\frac{1}{s})+\frac{k}{2}}
\lVert \nabla^ku(\tau)\rVert_{s,\Omega}.	
\end{align*}
The proof is complete.
\end{proof}

In order to prove Theorem \ref{thmnonlinear},
we also prepare the following result on the uniqueness 
of the local solution. 
We refer to \cite[Lem. 4.5]{T22} for the proof.
\begin{lem}
\label{lem-uniquness}
Let $\varepsilon_2$ be subject to Convention \ref{conv-2} and
let $r$ satisfy $1/3-\varepsilon_2<1/r<1/3$.
For $0\leq t_0<t_1<\infty$ and $v_0\in L^r_\sigma(\Omega)$,
the equation \eqref{eq-dual2}
on $(t_0,t_1)$ admits at most one solution
within $L^\infty(t_0,t_1;L^r_\sigma(\Omega))$.
\end{lem}

Let us close this paper with completion of the proof of
Theorem \ref{thmnonlinear}.

\begin{proof}[Proof of Theorem \ref{thmnonlinear}]
Fix $q$ satisfying $1/3-\varepsilon_2<1\slash q<1/3$.
Let $c_1 = c_1(q)$, $c_2=c_2(q)$, and $\CT$ be the numbers
given by the second assertion of Proposition \ref{prop-sol}.
In addition, we assume that there hold
$\lvert \eta \rvert + \lvert\omega \rvert<c_1$ and
$\lVert v_0 - b\rVert_{3,\infty,\Omega}<c_2$
so that the solution $u(t)$ to \eqref{eq-dual} may be 
constructed on account of Proposition \ref{prop-sol}.
Notice that the solution $u (t)$ satisfies
$u(t)\in L^r(\Omega)$, $3<r<q$, with
\begin{align}
\label{lrbound}
\|u(t)\|_{r,\Omega}\leq
C\Bigl\{\bigl(\lvert\eta\rvert + \lvert \omega \rvert \bigr) \CT^{1-\frac{3}{2q}}
+\lvert \eta \rvert + \lvert\omega\rvert+
\lVert v_0 - b\rVert_{3,\infty,\Omega} \Bigr\}
(t-k\CT)^{-\frac{1}{2}+\frac{3}{2r}}
\end{align}
for $k \in \BN \cup \{0\}$ and for $k\CT<t\leq (k+1)\CT$.
Let $t_*\in(0,\infty)$. Then there exists $k_*$ such that
$k_*\CT<t_*\leq (k_*+1)\CT$, thereby, we infer from
\eqref{lrbound} that
\begin{align*}
\|u(t)\|_{r,\Omega}&\leq
C\Bigl\{\bigl(\lvert\eta\rvert + \lvert \omega \rvert \bigr) \CT^{1-\frac{3}{2q}}
+\lvert \eta \rvert + \lvert\omega\rvert+
\lVert v_0 - b\rVert_{3,\infty,\Omega} \Bigr\}
(t-k_*\CT)^{-\frac{1}{2}+\frac{3}{2r}}\\
&\leq
C\Bigl\{\bigl(\lvert\eta\rvert + \lvert \omega \rvert \bigr) \CT^{1-\frac{3}{2q}}
+\lvert \eta \rvert + \lvert\omega\rvert+
\lVert v_0 - b\rVert_{3,\infty,\Omega} \Bigr\}
\left\{\frac{1}{2}(t_*-k_*\CT)\right\}^{-\frac{1}{2}+\frac{3}{2r}}
\end{align*}
for all $t\in[(k_*\CT+t_*)/2,t_*]$.
By taking into account this estimate and
by employing Lemma \ref{lem-local},
for each $t_0\in[(k_*\CT+t_*)/2,t_*)$,
Problem \eqref{eq-t0} with $u_0=u(t_0)$ admits a
local solution $\widetilde{u}\in Y_r(t_0,t_1)$ with
\begin{equation}
t_1-t_0\geq \zeta\left(
C\Bigl\{\bigl(\lvert\eta\rvert + \lvert \omega \rvert \bigr) \CT^{1-\frac{3}{2q}}
+\lvert \eta \rvert + \lvert\omega\rvert+
\lVert v_0 - b\rVert_{3,\infty,\Omega} \Bigr\}
\left\{\frac{1}{2}(t_*-k_*\CT)\right\}^{-\frac{1}{2}+\frac{3}{2r}}
\right)=:\theta.
\end{equation}
Let us take $t_0=\max\{(k_*\CT+t_*)/2,t_*-\theta/2\}$ so that
$t_*\in(t_0,t_1)$, in which $u=\widetilde{u}$ since
both $u$ and $\widetilde{u}$ satisfy \eqref{eq-dual2} and since
we have Lemma \ref{lem-uniquness}. 
Since $t_*$ may be taken arbitrarily, we complete the proof.
\end{proof}

\end{document}